\documentclass[numbers,webpdf]{ima-authoring-template}%

\graphicspath{{Fig/}}

\theoremstyle{thmstyletwo}%

\numberwithin{equation}{section}

\usepackage{microtype}
\usepackage{graphicx}
\usepackage{subcaption}
\usepackage{caption}
\usepackage{booktabs} 
\usepackage[margin=1.0in]{geometry}
\usepackage{amsfonts}           
\usepackage{amsthm}
\usepackage{amsmath}
\usepackage{amssymb}
\usepackage{stmaryrd}
\usepackage{xcolor}
\usepackage{caption}
\usepackage{algorithm}
\usepackage{algpseudocode}
\usepackage[margin=1.0in]{geometry}
\usepackage{comment}
\usepackage{hyperref}
\usepackage{mathtools}

\newtheorem{prop}{Proposition}[section]
\newtheorem{theorem}{Theorem}
\newtheorem{lemma}[prop]{Lemma}
\newtheorem{corollary}[prop]{Corollary}
\newtheorem{defn}[prop]{Definition}

\theoremstyle{definition}
\newtheorem{example}[prop]{Example}
\newtheorem{remark}[prop]{Remark}

\newcommand{\littletaller}{\mathchoice{\vphantom{\big|}}{}{}{}}

\DeclareMathOperator*{\argmin}{arg\,min}
\newcommand{\ones}{\mathbf{1}}

\newcommand{\KL}{\operatorname{KL}}
\newcommand{\Ent}{\operatorname{H}}
\newcommand\restr[2]{{%
  \left.\kern-\nulldelimiterspace %
  #1 %
  \littletaller %
  \right|_{#2} %
  }}

\newcommand{\Pc}{\mathcal{P}}
\newcommand{\Nc}{\mathcal{N}}
\newcommand{\Ac}{\mathcal{A}}

\newcommand{\Lc}{\mathsf{L}}

\newcommand{\Hc}{\mathcal{H}}
\newcommand{\Mc}{\mathcal{M}}
\newcommand{\Dc}{\mathcal{D}}

\newcommand{\Proj}{\operatorname{Proj}}

\renewcommand{\vec}{\operatorname{vec}}
\newcommand{\mat}{\operatorname{mat}}
\newcommand{\diff}{\mathrm{d}}
\newcommand{\R}{\mathbb{R}}
\newcommand{\tr}{\operatorname{tr}}

\newcommand{\define}[1]{\textbf{#1}}

\DeclareMathOperator{\PL}{\mathcal{L}\mathcal{P}}

\allowdisplaybreaks

\hypersetup{
  colorlinks   = true, %
  urlcolor     = blue, %
  linkcolor    = blue, %
  citecolor   = blue %
}

\begin{document}

\DOI{}
\copyrightyear{}
\vol{}
\pubyear{}
\access{}
\appnotes{}
\copyrightstatement{}
\firstpage{1}

\title[Geometry of the Space of Partitioned Networks]{Geometry of the Space of Partitioned Networks:\\
A Unified Theoretical and Computational Framework}

\author{Stephen Y Zhang
\address{\orgdiv{School of Mathematics and Statistics}, \orgname{University of Melbourne}
}}
\author{Fangfei Lan
\address{\orgdiv{Scientific Computing and Imaging Institute }, \orgname{University of Utah}
}}
\author{Youjia Zhou
\address{\orgname{Meta}
}}
\author{Agnese Barbensi
\address{\orgdiv{School of Mathematics and Physics}, \orgname{University of Queensland}
}}
\author{Michael P H Stumpf
\address{\orgdiv{School of BioSciences and School of Mathematics and Statistics}, \orgname{University of Melbourne}
}}
\author{Bei Wang
\address{\orgdiv{Kahlert School of Computing}, \orgname{University of Utah}
}}
\author{Tom Needham*
\address{\orgdiv{Department of Mathematics}, \orgname{Florida State University}
}}

\authormark{Zhang, Lan, Zhou, Barbensi, Stumpf, Wang, Needham}

\corresp[*]{Corresponding author: \href{email:tneedham@fsu.edu}{tneedham@fsu.edu}}

\abstract{Network data are ubiquitous in the real world to capture pairwise or high-order relations among objects. We introduce a class of measure-theoretic network objects called \emph{partitioned measure networks} that generalize a number of objects used to model network data in the literature, such as graphs, hypergraphs, and augmented graphs (i.e.,~graphs whose nodes are assigned categorical classes). We then propose a metric called a \emph{partitioned network distance} between partitioned measure networks that extends the Gromov-Wasserstein distance between graphs and the co-optimal transport distance between hypergraphs. We characterize the geometry of the space of partitioned measure networks, thereby providing a unified theoretical treatment of generalized network structures that encompass both pairwise and higher-order relations. In particular, we show that our metric defines an Alexandrov space of non-negative curvature, and leverage this structure to define gradients for certain functionals commonly arising in geometric data analysis tasks. We extend our framework to the setting where nodes have additional label information, and derive efficient computational schemes to utilize the partitioned network distance in practice. Equipped with these theoretical and computational tools, we demonstrate the utility of our framework in a suite of applications, including hypergraph alignment, clustering and dictionary learning from ensemble data, multi-omics alignment, as well as multiscale network alignment.}

\maketitle

\section{Introduction}

Modelling relations among objects or concepts is a central task across the natural sciences, engineering, as well as arts and humanities. Graphs are a conventional approach to model pairwise relations between objects. Many real-world systems, however, involve higher-order interactions among three or more objects.
In biochemical reaction networks, reactions typically involve multiple chemical species simultaneously~\cite{jost2019hypergraph}. In coauthorship networks, papers are written jointly by any number of authors \cite{zhou2021topological}. And in a theatre play, each scene can be viewed as a (higher-order) interaction between a set of characters~\cite{coupette2024all}. 
These systems cannot be modelled as graphs without information loss. Instead, we need to introduce more general structures such as hypergraphs, simplicial complexes, and cell complexes~\cite{bick2023higher}. 

A natural question arising from the study of graphs is how to compare them:~specifically, how to characterize the distance or degree of similarity between two graphs. This is not straightforward, since two graphs may vary in the number of nodes and comparisons must be invariant under permutation~\cite{Thorne:du} and other symmetries. 
The complexity involving graph data has driven the development of an extensive tool set for graph comparison and matching~\cite{wills2020metrics}, including spectral methods~\cite{feizi2016spectral, feizi2019spectral, nassar2018low, Thorne:du} and graph kernels \cite{shervashidze2011weisfeiler, borgwardt2005shortest}. 
Among these methods, Gromov-Wasserstein couplings of \emph{metric measure spaces} (i.e., metric spaces equipped with measures) have proven fruitful from theoretical and computational perspectives~\cite{memoli2011gromov}. 
A \emph{coupling} between measures is a relaxed notion of correspondence between objects~\cite{memoli2011gromov}; and it creates a joint probability space with the desired measures as its marginals. By modelling graphs as metric measure spaces and considering couplings between them, a notion of Gromov-Wasserstein distance between graphs emerges in terms of a \emph{least distortion} principle. This distance is in fact a pseudometric, and the space of graphs (considered up to a natural notion of equivalence) can thus be formalized as a metric space endowed with the Gromov-Wasserstein metric~\cite{chowdhury2019gromov}. The geometry of the space of metric measure spaces, endowed with this metric, was studied in detail by Sturm \cite{sturm2023space} and was shown to be an Alexandrov space with curvature bounded below. This is a powerful characterization that allows well-defined notions of geodesics, tangent spaces, and gradient flows.

Computationally, the Gromov-Wasserstein framework formulates the distance between two graphs as an optimal value of a non-convex quadratic program over a set of feasible couplings. Efficient computational schemes exist to find local minima of this problem, which have given rise to an algorithmic tool box for dealing with graphs that has gained popularity in the statistics and machine learning communities~\cite{peyre2016gromov}. 

While the Gromov-Wasserstein approach to graph comparison has been a foundational tool for understanding the space of graphs from a geometric viewpoint, in its basic form, it is insufficient to model higher-order systems such as hypergraphs. A \emph{hypergraph} consists of a set of nodes and a set of hyperedges (i.e.,~subsets of nodes); if each hyperedge contains exactly two nodes, then a hypergraph reduces to a graph. 
A hypergraph can be used to encode multi-way relations among nodes. Recently, Chowdhury et al.~\cite{chowdhury2023hypergraph} introduced Gromov-Wasserstein type distances between hypergraphs, based on the co-optimal transport framework of Redko et al.~\cite{titouan2020co}. Framing hypergraphs as metric measure spaces, Chowdhury et al.~showed that the space of hypergraphs can be characterized as a metric space; however, an in-depth geometric description of this metric space remains to be fully developed.  

\begin{figure}[!ht]
\centering 
\includegraphics[width=0.75\linewidth]{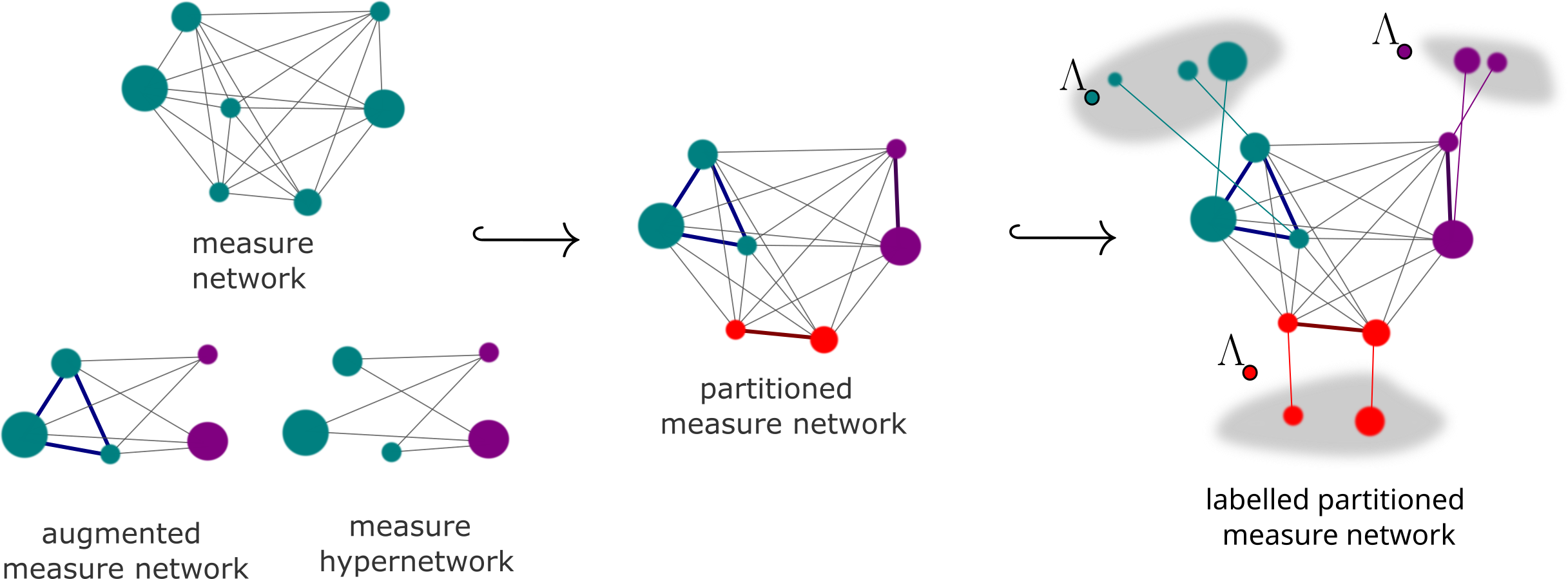}
\caption{A schematic representation of the types of generalized networks which can be embedded into the space of partitioned measure networks. From left to right: (i) Examples of a measure network (or a weighted graph) whose nodes are endowed with a probability measure; a measure hypernetwork (or a weighted bipartite graph) with a node probability measure; and an augmented measure network (or a measure network) whose node set has been partitioned into two classes. (ii) Illustration of a partitioned measure network with $k = 3$ partitions, that is, a measure network whose nodes have been partitioned into 3 classes. (iii) Illustration of a labelled partitioned measure network, or a partitioned measure network whose nodes are endowed with features in some auxiliary metric spaces.}
\label{fig:generalised_network}
\vspace{-4mm}
\end{figure}

In this paper, we introduce a general class of network objects called \emph{partitioned measure networks}, which generalize a number of measure-theoretic objects recently introduced in the literature: measure networks~\cite{chowdhury2019gromov}, measure hypernetworks~\cite{chowdhury2023hypergraph}, and augmented measure networks~\cite{demetci2023revisiting}; see Figure~\ref{fig:generalised_network}, and Remark~\ref{remark:reps_as_graph_structures} for further discussions. A \emph{$k$-partitioned measure network} is a graph structure whose nodes are separated into $k$ classes, and this partitioning should be taken into account when comparing these structures. For example, a hypergraph can be encoded as a bipartite graph between two partitions (that correspond to sets of nodes and hyperedges respectively), so that any hypergraph can be modelled as a $2$-partitioned measure network; further examples are provided in Example~\ref{ex:partitioned_networks}. We also consider $k$-partitioned networks whose nodes come with attributes in some auxiliary metric space, referred to as \emph{labelled $k$-partitioned measure networks}.

\smallskip\noindent\textbf{Contributions.}~We highlight our contributions below.   First, we equip the space of partitioned measure networks with a family of transport-based distances formulated as the minimum of a quadratic distortion functional over $k$-\emph{partitioned couplings} (i.e.,~measure couplings that respect the class structures of the partitioned networks). This choice of distance is shown to be a {\em bona fide} metric (up to a natural equivalence relation), and we explicitly construct isometric embeddings of measure networks, hypernetworks, and augmented measure networks into this space; see Definition~\ref{defn:generalized_network_embeddings} and Theorem~\ref{thm:isometric_embeddings}.

Second, we characterize geodesics in the space of partitioned measure networks and show that it is an Alexandrov space with curvature bounded below; see Theorem~\ref{thm:alexandrov_space}. As an extension to our analysis, we consider the addition of \emph{labels} to partitioned measure networks, which provides a generalization of the so-called \emph{Fused Gromov-Wasserstein} problem~\cite{vayer2020fused}. When labels reside in a Hilbert space, we show that our geometric characterization of geodesics and curvature also applies; see Theorem~\ref{thm:alexandrov_space_labelled}.

Third, our theoretical contributions provide a unified treatment of a family of generalized networks, which encompass multiple network objects recently introduced in the literature~\cite{chowdhury2019gromov,titouan2020co,chowdhury2023hypergraph,demetci2023revisiting}. This allows us to provide a common geometric description of these spaces. To the best of our knowledge, in the settings of labelled measure networks, measure hypernetworks, and augmented measure networks, these characterizations are new. We conclude our theoretical contributions in Section~\ref{sec:riemannian}, with a brief discussion of the Riemannian geometric concepts of tangent spaces, exponential and logarithmic maps, and gradients in the space of $k$-partitioned measure networks, which are crucial for practical applications in learning algorithms. We remark here that a recent paper~\cite{bauer2024z} also provides a general framework for studying several variants of the Gromov-Wasserstein distance, but that the results therein are disjoint from the ones presented here: the framework of the present paper captures different variants of Gromov-Wasserstein distance than that of~\cite{bauer2024z}, and the curvature bounds and Riemannian structures established here are not treated in~\cite{bauer2024z}.

Fourth, we demonstrate the utility of our framework on computational case studies, which bring together the theoretical ideas. We provide a connection between the Gromov-Wasserstein network matching problem and a family of spectral network alignment algorithms~\cite{feizi2016spectral, nassar2018low}. To our knowledge, this connection has not been explicitly pointed out in prior work. We show that partitioned measure networks provide a more natural and flexible extension to the hypergraph alignment problem, and we demonstrate in numerical experiments that transport-based approaches are more accurate and efficient than spectral methods.

We illustrate the practicality of our methods with applications to metabolic network alignment, simultaneous sample and feature alignment in multi-omics data, and multi-scale network matching. We formulate each of these problems in terms of partitioned network matching. In addition, we investigate some more complex tasks on the space of partitioned networks that exploit their geometric properties. For example, we provide  computational characterizations of geodesics and barycenters in the space of measure hypernetworks. We introduce geodesic dictionary learning as a bi-level problem on the space of (partitioned) measure networks, from which we motivate linearized dictionary learning~\cite{vincent2021online} as a fast approximate algorithm. We conclude by demonstrating the utility of geodesic dictionary learning in synthetic hypergraph block models as well as networks derived from atomic and topological representations of small molecules.

\smallskip\noindent\textbf{Overview.}~The plan for the paper is as follows. In Section~\ref{sec:spaces_generalized_networks}, we give  precise definitions of various spaces of generalized networks and define our new variant of Gromov-Wasserstein distance---the partitioned network distance---on them. We establish metric properties of the partitioned network distance. We generalize these ideas in Section~\ref{sec:extension_to_labelled_partitioned_measure_networks} by introducing node features to the partitioned networks. We define a new metric on these objects and establish its metric properties, including curvature bounds. The deferred proofs from previous sections follow as corollaries to these more general results. 
We study the Riemannian structures following from our curvature bounds in Section~\ref{sec:riemannian}, with a focus on computational examples. In Section \ref{sec:applications}, we give an extended collection of computational examples and applications of our framework. This is supplemented by an open source code repository, available at \url{https://github.com/zsteve/partitioned_networks}. We conclude the main paper with a discussion of future directions in Section~\ref{sec:discussion}, followed by technical aspects of our numerical methods in the Appendix~\ref{sec:algorithms_appendix}.

\section{Metric geometry of spaces of generalized networks}
\label{sec:spaces_generalized_networks}

\subsection{Spaces of generalized networks}

We first review some notions of generalized networks, and Gromov-Wasserstein type distances between them, which have appeared previously in the literature.

\subsubsection{Generalized networks}

Let us first recall some definitions of generalized network structures in the literature. These use the following notational conventions: given measure spaces $(X,\mu)$ and $(Y,\nu)$, we use $L^p(\mu)$ to denote the space of $p$-integrable functions on $X$ (for $p \in [1,\infty]$) and we use $\mu \otimes \nu$ to denote the product measure on $X \times Y$. 

\begin{defn}[Various notions of generalized networks]\label{def:generalized_networks}
Let $p \in [1,\infty]$. 
    \begin{enumerate}
        \item A \define{measure $p$-network}~\cite{chowdhury2019gromov} is a triple $N = (X,\mu,\omega)$, where $X$ is a Polish space, $\mu$ is a Borel probability measure on $X$, and $\omega:X \times X \to [0,\infty)$ is an element of $L^p(\mu \otimes \mu)$. Let $\mathcal{N}^p$ denote the collection of all measure $p$-networks.

        \item  A \define{measure $p$-hypernetwork}~\cite{titouan2020co,chowdhury2023hypergraph} is a five-tuple $H = (X,\mu,Y,\nu,\omega)$ consisting of Polish spaces $X$ and $Y$ endowed with Borel probability measures $\mu$ and $\nu$, respectively, and $\omega:X \times Y \to [0,\infty)$ is an element of $L^p(\mu \otimes \nu)$. Let $\mathcal{H}^p$ denote the collection of all measure $p$-hypernetworks.

        \item An \define{augmented measure $p$-network}~\cite{demetci2023revisiting} is a six-tuple $A = (X,\mu,Y,\nu,\omega_X,\omega_{XY})$, consisting of Polish spaces $X$ and $Y$ endowed with Borel probability measures $\mu$ and $\nu$, respectively, and  $\omega_X:X \times X \to [0,\infty)$ and $\omega_{XY}:X \times Y \to [0,\infty)$ are elements of $L^p(\mu \otimes \mu)$ and $L^p(\mu \otimes \nu)$, respectively. Let $\mathcal{A}^p$ denote the collection of all augmented measure $p$-networks.
    \end{enumerate}
\end{defn}

 We will frequently suppress explicit mention of the parameter $p$, as the appropriate value will typically be either unimportant or clear from context; e.g., when referring to certain $L^p$-type metrics. The functions $\omega$, $\omega_X$, $\omega_{XY}$, etc., will loosely be referred to as 
 \define{\emph{network kernels}}. 

\begin{example}\label{ex:generalized_network_examples}
    We now provide prototypical examples of the structures defined above.
    \begin{enumerate}
        \item A \define{metric measure space} is a measure network $(X,\mu,\omega)$ such that $\omega$ satisfies the axioms of a metric, and the topology of $X$ is induced by the metric. This was the original setting where the Gromov-Wasserstein distances were formulated; see~\cite{memoli2007,memoli2011gromov} and the related work~\cite{sturm2006geometry}. More general examples of measure networks frequently come from the setting of graph theory, where $X$ is a finite set of nodes, $\omega$ encodes node affinities, and $\mu$ is some choice of weights on the nodes (e.g., uniform). For example, weighted adjacency functions are used as network kernels in this framework to represent protein-protein interaction networks in~\cite{xu2019scalable}, and heat kernels are used in~\cite{chowdhury2021generalized} for the purpose of uncovering community structures in graph datasets.
        \item A \define{hypergraph} is a set $V$ of nodes and a set $E$ of subsets of $V$, each of which is referred to as a \define{hyperedge}. The containment or incidence relation can be encoded as a function $\omega:V \times E \to \{0,1\}$, so that picking probability measures $\mu$ and $\nu$ (say, uniform) yields a measure hypernetwork $(V,\mu,E,\nu,\omega)$. This representation was used in~\cite{chowdhury2023hypergraph}, where it was applied, for example, to simplify hypergraphs representing complicated social interactions. The notion of a measure hypernetwork therefore generalizes the notion of a hypergraph. One also obtains a measure hypernetwork from a \define{data matrix}, where $X$ is an indexing set for samples, $Y$ is an indexing set for features , $\omega(x,y)$ is the value of the matrix for sample $x$ and feature $y$ and $\mu$ and $\nu$ are some choices of weights. Applications to analysis of ensembles of data matrices was a main motivation for the introduction of this formalism in~\cite{titouan2020co}.
        
        \item One can obtain an \define{augmented measure network} by taking $(X,\mu,Y,\nu,\omega)$ to be a measure hypernetwork representation of a data matrix, setting $\omega_{XY} = \omega$ and taking $\omega_X$ to be some relational function on the rows, such as distance between samples in the data space. This approach was taken to model multi-omics data in~\cite{demetci2023revisiting}, with a view toward integrating several single-cell multi-omics datasets.
    \end{enumerate}
\end{example}

\begin{remark}[Representations as graph structures]\label{remark:reps_as_graph_structures}
    Figure~\ref{fig:generalised_network} provides schematic representations of the generalized networks defined above. To interpret this, one should conceptualize a measure network as a fully connected, weighted graph with a node set $X$ and edge weights encoded by $\omega$. From this perspective, it is natural to consider a measure hypernetwork as a bipartite weighted graph, where $X$ and $Y$ define the two classes of nodes. Then $\omega$ defines edge weights for the complete bipartite graph, while there are no edges joining pairs of nodes in $X$ or pairs of hyperedges in $Y$. Finally, an augmented measure network can be viewed similarly, with the difference being that edges between nodes in $X$ are permitted with weights encoded by $\omega_X$ (and bipartite edge weights being encoded by $\omega_{XY}$). These interpretations lead to the $k$-partitioned measure network formalism introduced in Section~\ref{sec:partitioned_measure_networks}. See Definition \ref{defn:generalized_network_embeddings} and Remark \ref{rem:intuition_for_embeddings} for a more formal justification of these representations.
\end{remark}

\begin{remark}[Finite and continuous spaces]
    The concept of measure hypernetwork was introduced in \cite{titouan2020co}, somewhat less formally, and primarily as a model for data matrices (so $X$ and $Y$ were assumed to be finite). This definition was formalized and extended to the setting of infinite spaces in \cite{chowdhury2023hypergraph} (in a slightly less general form than what is presented here, as network kernels were assumed therein to be bounded and the underlying Polish spaces were assumed to be compact for many results). The computational examples of practical interest are, of course, always defined over finite spaces; the main motivation for extending the definition to infinite spaces is to allow us to consider the collection of all measure hypernetworks as a complete metric space with respect to the distance defined below. Similarly, the notion of an augmented measure network was introduced for finite spaces in~\cite{demetci2023revisiting}, and the more formal definition provided above is novel.
\end{remark}

\subsubsection{Generalized network distances}

For each flavour of generalized network described in Definition \ref{def:generalized_networks}, there has been an associated notion of distance introduced in the literature. They all have a similar structure, defined in terms of optimizing over measure couplings, as in the Kantorovich formulation of classical optimal transport (see \cite{villani2009optimal,peyre2019computational} as general references on optimal transport). Below, we use $\mathrm{proj}_X:X \times X' \to X$ and $\mathrm{proj}_{X'}:X \times X' \to X'$ to denote coordinate projections on some product sets, and for a Borel measurable map $p:X \to Y$ of topological spaces, we use $p_\# \mu$ to denote the pushforward to $Y$ of a Borel measure $\mu$ on $X$.

\begin{defn}[Measure coupling]
For probability spaces $(X,\mu)$ and $(X',\mu')$, we say that a measure $\pi$ on $X \times X'$ is a \define{coupling} of $\mu$ and $\mu'$ if its left and right marginals are equal to $\mu$ and $\mu'$, respectively; that is,  $(\mathrm{proj}_X)_\# \pi = \mu$ and $(\mathrm{proj}_{X'})_\# \pi = \mu'$. Let $\Pi(\mu,\mu')$ denote the set of all couplings of $\mu$ and $\mu'$.
\end{defn}

Let us now establish some convenient notational conventions that will be used throughout the rest of the paper. We always use $N$ and $N'$ to stand for measure networks, with the underlying data always being implicitly denoted $N = (X,\mu,\omega)$ and $N' = (X',\mu', \omega')$.  We similarly use $H,H'$ for measure hypernetworks and $A,A'$ for augmented hypernetworks. That is, when referring to $H$, the underlying data is implicitly given by $H = (X,\mu,Y,\nu,\omega)$, unless explicitly stated otherwise.  For functions $\tau:A \times B \to \R$ and $\tau': A' \times B' \to \R$, the difference $\tau - \tau'$ is understood to be the function defined on $A \times A' \times B \times B'$ as follows: 
\begin{align*}
    \tau - \tau':A \times A' \times B \times B' &\to \mathbb{R} \\
    (a,a',b,b') &\mapsto \tau(a,b) - \tau'(a',b').
\end{align*}
Given a measure space $(X,\mu)$, we use $\|\cdot\|_{L^p(\mu)}$ to denote the standard norm on $L^p(\mu)$. With these conventions in mind, we recall some metrics which have been introduced on the generalized network spaces of Definition~\ref{def:generalized_networks}.

\begin{defn}[Generalized network distances]\label{defn:generalized_network_distances}
Let $p \in [1,\infty]$. 
\begin{enumerate}
    \item The \define{network $p$-distance} or \define{Gromov-Wasserstein $p$-distance}~\cite{chowdhury2019gromov} between measure networks $N,N' \in \mathcal{N}^p$ is 
    \[
    d_{\mathcal{N}^p}(N,N') \coloneqq \inf_{\pi \in \Pi(\mu,\mu')} \frac{1}{2}\|\omega - \omega'\|_{L^p(\pi \otimes \pi)}.
    \]
    \item The \define{hypernetwork $p$-distance}~\cite{titouan2020co,chowdhury2023hypergraph} between measure hypernetworks $H,H' \in \mathcal{H}^p$ is 
    \[
    d_{\mathcal{H}^p}(H,H') \coloneqq \inf_{\substack{\pi \in \Pi(\mu,\mu') \\ \xi \in \Pi(\nu,\nu')}} \frac{1}{2}\|\omega - \omega'\|_{L^p(\pi \otimes \xi)}.
    \]
    \item The \define{augmented network $p$-distance}~\cite{demetci2023revisiting} between augmented measure networks $A,A' \in \mathcal{A}^p$ is 
    \[
    d_{\mathcal{A}^p}(A,A') \coloneqq \inf_{\substack{\pi \in \Pi(\mu,\mu') \\ \xi \in \Pi(\nu,\nu')}} \frac{1}{2}\left(\|\omega_X - \omega_{X'}'\|_{L^p(\pi \otimes \pi)}^p + \|\omega_{XY} - \omega_{X'Y'}'\|_{L^p(\pi \otimes \xi)}^p \right)^{1/p}.
    \]
\end{enumerate}
\end{defn}

In the following section, we unify these generalized network concepts and distances (as well as others) under a common framework. We will use this common framework to derive various metric properties of these distances simultaneously.

\subsection{Partitioned measure networks and generalized networks}\label{sec:partitioned_measure_networks}

One should observe the similarities between the various notions of generalized network in Definition \ref{def:generalized_networks}, and the optimal transport-inspired distances between them described in Definition \ref{defn:generalized_network_distances}. In this section, we describe a new structure that simultaneously generalizes these ideas.

\subsubsection{Partitioned measure networks} Let us now introduce a new generalized network structure.

\begin{defn}[Partitioned measure network]\label{defn:partitioned_measure_network}
    Let $k$ be a positive integer and $p \in [1,\infty]$. A \define{$k$-partitioned measure $p$-network} is a structure of the form $P = \big((X_i,\mu_i)_{i=1}^k,\omega\big)$, where
    \begin{itemize}
        \item each $(X_i,\mu_i)$ is a Polish probability space such that $X_i \cap X_j = \emptyset$ for $i \neq j$, and
        \item using the notation $X \coloneqq \sqcup_i X_i$ and $\mu \coloneqq \sum_i \mu_i$, $\omega:X \times X \to [0,\infty)$ is an element of $L^p(\mu \otimes \mu)$.
    \end{itemize}
    To simplify notation, we sometimes write $(X_i,\mu_i)$ instead of $(X_i,\mu_i)_{i=1}^k$, $(X_i)$ instead of $(X_i)_{i=1}^k$ and $(\mu_i)$ instead of $(\mu_i)_{i=1}^k$. We use  $\Pc_k^p$ to denote the collection of all $k$-partitioned measure $p$-networks.
\end{defn}

It is often the case that the particular values of $k$ and $p$ are not important, in which case we abuse terminology and refer to the objects defined above as  \define{partitioned measure networks}. In line with those established above, we follow the notational convention that $P$ and $P'$ are implicitly assumed to stand for partitioned measure networks $P = ((X_i,\mu_i),\omega)$ and $P' = ((X'_i,\mu_i'), \omega')$. Observe that a $1$-partitioned measure network is just a measure network; that is, $\Nc^p = \Pc_1^p$. Next, we observe below that we could embed the various notions of generalized networks (from Definition \ref{def:generalized_networks}) into the space of partitioned measure networks $\Pc_2^p$. 

\begin{defn}[Generalized network embeddings]\label{defn:generalized_network_embeddings}
We have the following families of embeddings.
    \begin{enumerate}
        \item For each $k$, let $\varepsilon_{k,k+1}:\Pc_k^p \to \Pc_{k+1}^p$ be the map taking $P \in \Pc_k^p$ to 
        \[
        \varepsilon_{k,k+1}(P) \coloneqq \left( (X_i, \mu_i)_{i = 1}^{k+1}, \varepsilon_{k,k+1}(\omega)\right),
        \]
        where $X_{k+1}$ consists of a single abstract point, $\mu_{k+1}$ is the associated Dirac measure, and the network kernel is defined by
        \begin{align*}
        &\varepsilon_{k,k+1}(\omega):(\sqcup_{i=1}^{k+1} X_i) \times (\sqcup_{i=1}^{k+1} X_i) \to \R, \\ 
        &\varepsilon_{k,k+1}(\omega)(u,v) = \left\{\begin{array}{rl} 
        \omega(u,v) & (u,v) \in (\sqcup_{i=1}^{k} X_i) \times (\sqcup_{i=1}^{k} X_i); \\
        0 & \mbox{otherwise.}
        \end{array}\right.
        \end{align*}
        In particular, $\varepsilon_{1,2}$ gives an embedding $\Nc^p \hookrightarrow \Pc_2^p$. For $k < \ell$, we define $\varepsilon_{k,\ell}:\Pc^p_k \to \Pc^p_\ell$ by 
        \[
        \varepsilon_{k,\ell} \coloneqq \varepsilon_{\ell-1,\ell} \circ \cdots \circ \varepsilon_{k,k+1}.
        \]
        \item Let $\varepsilon_{\mathcal{H}}:\mathcal{H}^p \to \Pc^p_2$ be the map taking $H \in \mathcal{H}^p$ to 
        \[
        \varepsilon_{\mathcal{H}}(H) \coloneqq \left(((X, \mu), (Y, \nu)), \varepsilon_{\mathcal{H}}(\omega)\right),
        \]
        where
        \[
        \varepsilon_{\mathcal{H}}(\omega)(u,v) \coloneqq \left\{\begin{array}{rl}
        \omega(u,v) & u \in X \mbox{ and } v\in Y;\\
        0 & \mbox{otherwise}.
        \end{array}\right.
        \]
        \item Let $\varepsilon_{\mathcal{A}}:\mathcal{A}^p \to \Pc_2^p$ be the map taking $A \in \mathcal{A}^p$ to 
        \[
          \varepsilon_{\mathcal{A}}(A) \coloneqq \left(((X, \mu), (Y, \nu)), \varepsilon_{\mathcal{A}}(\omega_X,\omega_{XY})\right),
        \]
        where 
        \[
        \varepsilon_{\mathcal{A}}(\omega_X,\omega_{XY})(u,v) \coloneqq \left\{\begin{array}{rl}
        \omega_X(u,v) & u,v \in X;\\
        \omega_{XY}(u,v) & u \in X \mbox{ and } v \in Y;\\
        0 & \mbox{otherwise}.
        \end{array}\right.
        \]
    \end{enumerate}
\end{defn}

\begin{example}[Partitioned networks]\label{ex:partitioned_networks}
    The generalized network embeddings defined above show that measure networks, hypernetworks and augmented hypernetworks can be considered as 2-partitioned networks, so that this structure encompasses those described in Example~\ref{ex:generalized_network_examples}. 
    Another source of examples of $k$-partitioned networks is the notion of a dataset with a categorical class structure. That is, consider a (say, finite) dataset of points $X$ such that each $x \in X$ belongs to one of $k$ different classes---for example, $X$ could consist of a set of images, and the images could be assigned classes based on subject matter (e.g., cats, dogs, etc.). This class structure can be encoded as probability measures by taking $\mu_i$ to be a uniform measure supported on those points belonging to class $i$. The supports then define the required partition of $X = \sqcup_i X_i$ and any choice of network kernel $\omega$ on $X$ gives a representation of the multiclass dataset as a $k$-partitioned measure network.
\end{example}

\begin{remark}[Intuition for the embeddings]\label{rem:intuition_for_embeddings}
    Remark~\ref{remark:reps_as_graph_structures} gives interpretations of generalized networks in terms of graph structures, illustrated schematically in Figure~\ref{fig:generalised_network}. The embeddings defined above formalize these intuitive descriptions mathematically. For example, if $H \in \mathcal{H}^p$ is a representation of a hypergraph, then $\varepsilon_\mathcal{H}(H)$ gives a representation of $H$ as a bipartite graph.
\end{remark}

\subsubsection{Partitioned network distance} We now define a distance between partitioned measure networks, using the concept of a partitioned coupling. 

\begin{defn}[Partitioned coupling]\label{defn:partitioned_coupling}
Given $k$-tuples of probability spaces $(X_i,\mu_i)_{i=1}^k$ and $(X'_i,\mu_i')_{i=1}^k$, let
\[
\Pi_k\left((\mu_i),(\mu_i')\right)  \coloneqq \Pi(\mu_1,\mu_1') \times \cdots \times \Pi(\mu_k,\mu_k').
\]
An element $(\pi_i)_{i=1}^k$ of $\Pi_k\left((\mu_i),(\mu_i')\right)$ is a \define{$k$-partitioned coupling}. To simplify notation, we sometimes denote $k$-partitioned couplings as $(\pi_i)$ instead of $(\pi_i)_{i=1}^k$. 
\end{defn}

\begin{defn}[Partitioned network distance]\label{defn:partitioned_network_distance}
    For $p \in [0,\infty)$, the \define{partitioned network $p$-distance} between partitioned measure networks $P,P' \in \Pc_k^p$ is 
    \begin{equation}\label{eqn:distance_definition}
    d_{\Pc_k^p}(P,P') \coloneqq \inf_{(\pi_i) \in \Pi_k((\mu_i),(\mu'_i))} \frac{1}{2} \left(\sum_{i,j = 1}^k \|\omega - \omega'\|_{L^p(\pi_i \otimes \pi_j)}^p\right)^{1/p}.
    \end{equation}
    For $p=\infty$, we define
    \begin{equation}\label{eqn:distance_definition_infty}
    d_{\Pc_k^\infty}(P,P') \coloneqq \inf_{(\pi_i) \in \Pi_k((\mu_i),(\mu_i'))} \frac{1}{2} \max_{i,j} \|\omega - \omega'\|_{L^\infty(\pi_i \otimes \pi_j)}.
    \end{equation}
\end{defn}

\begin{remark}
    In the above definition and throughout the rest of the paper, we slightly abuse terminology and consider each $\pi_i$ as a probability measure on $X \times X' = (\sqcup_j X_j) \times (\sqcup_j X_j')$ (which is supported on the subset $X_i \times X_i'$).
\end{remark}

\begin{remark}[Connection to labelled Gromov-Wasserstein distance~\cite{ryu2024cross}]
A notion of Gromov-Wasserstein distance is introduced which is essentially equivalent to our partitioned network distance~\cite{ryu2024cross}. However, the treatment in that paper is very much from a computational perspective, and a formal, a general definition of the distance is not provided. The usefulness of such a metric is demonstrated by applications to cross-modality matching for biological data. We remark that~\cite{ryu2024cross} uses the terminology \define{labelled Gromov-Wasserstein distance}; we use the term \define{labelled} differently below, in Section \ref{sec:extension_to_labelled_partitioned_measure_networks}.
\end{remark}

When $k=1$, $d_{\Pc_1^p}$ is simply the Gromov-Wasserstein distance $d_{\Nc^p}$. We will show below that, for arbitrary $k$, $d_{\Pc_k^p}$ induces a metric on $\Pc_k^p$ modulo a natural notion of equivalence, which generalizes the known result in the Gromov-Wasserstein case. Before doing so, we show that the embeddings of generalized networks of Definition \ref{defn:generalized_network_embeddings} preserve the notions of distance that we have defined so far. 

\begin{theorem}\label{thm:isometric_embeddings}
    The maps from Definition \ref{defn:generalized_network_embeddings} preserve generalized network distances:
    \begin{enumerate}
    \item for $k < \ell$, $d_{\Pc_\ell^p}(\varepsilon_{k,\ell}(P),\varepsilon_{k,\ell}(P')) = d_{\Pc_{k}^p}(P,P')$; in particular, $d_{\Pc_2^p}(\varepsilon_{1,2}(N),\varepsilon_{1,2}(N')) = d_{\Nc^p}(N,N')$;
    \item $d_{\Pc_2^p}(\varepsilon_{\mathcal{H}}(H),\varepsilon_{\mathcal{H}}(H')) = d_{\mathcal{H}^p}(H,H')$; 
    \item and $d_{\Pc_2^p}(\varepsilon_{\mathcal{A}}(A),\varepsilon_{\mathcal{A}}(A')) = d_{\mathcal{A}^p}(A,A')$.
    \end{enumerate}
\end{theorem}

\begin{proof}
    We provide details for the $p<\infty$ case, with the proof for $p = \infty$ following by similar arguments. For the first claim, it suffices to consider the case where $\ell = k+1$. Let $P,P' \in \Pc_{k}^p$. Any $(\pi_i) = (\pi_i)_{i=1}^{k} \in \Pi_{k}((\mu_i),(\mu_i'))$ extends uniquely to a $(k+1)$-partitioned coupling of $(\mu_1,\ldots,\mu_{k},\mu_{k+1})$ and $(\mu_1',\ldots,\mu_{k}',\mu_{k+1}')$, namely, $(\pi_1,\ldots,\pi_{k},\pi_{k+1})$, where $\pi_{k+1}$ the Dirac mass on the singleton set $X_{k+1} \times X_{k+1}'$. We have that
    \begin{align*}
        \sum_{i,j=1}^{k+1} \|\varepsilon_{k,k+1}(\omega) - \varepsilon_{k,k+1}(\omega')\|_{L^p(\pi_i \otimes \pi_j)}^p &= \sum_{i=1}^{k+1} \|\varepsilon_{k,k+1}(\omega) - \varepsilon_{k,k+1}(\omega')\|_{L^p(\pi_i \otimes \pi_{k+1})}^p  \\
        &\qquad + \sum_{j=1}^{k+1} \|\varepsilon_{k,k+1}(\omega) - \varepsilon_{k,k+1}(\omega')\|_{L^p(\pi_{k+1} \otimes \pi_j)}^p \\
        & \qquad + \sum_{i,j=1}^{k} \|\varepsilon_{k,k+1}(\omega) - \varepsilon_{k,k+1}(\omega')\|_{L^p(\pi_i \otimes \pi_j)}^p \\
        &= \sum_{i,j=1}^{k} \|\omega - \omega'\|_{L^p(\pi_i \otimes \pi_j)}^p,
    \end{align*}
    where the last line follows because $\varepsilon_{k,k+1}(\omega) = \varepsilon_{k,k+1}(\omega') = 0$ on the supports of $\pi_i \otimes \pi_{k+1}$ and $\pi_{k+1} \otimes \pi_{j}$. And moreover, $\varepsilon_{k,k+1}(\omega) = \omega$ and $\varepsilon_{k,k+1}(\omega') = \omega'$ on the support of each $\pi_i \otimes \pi_j$ with $i,j < k$. 
    Since the $k$-partitioned coupling $(\pi_i)$ was arbitrary, the first claim follows. 

    Let us now prove the third claim; the case for hypernetworks is proved using similar arguments, so we omit the details here. Let $A,A' \in \Ac^p$ and let $\pi \in \Pi(\mu,\mu')$ and $\xi \in \Pi(\nu,\nu')$. Then $(\pi_1,\pi_2) = (\pi,\xi)$ is a $2$-partitioned coupling of $(\mu,\nu)$ and $(\mu',\nu')$, and we have 
    \[
    \sum_{i,j = 1}^2 \|\varepsilon_{\Ac}(\omega) - \varepsilon_{\Ac}(\omega')\|_{L^p(\pi_i \otimes \pi_j)}^p = \|\omega_{X} - \omega_{X'}'\|_{L^p(\pi \otimes \pi)}^p + \|\omega_{XY} - \omega_{X'Y'}'\|_{L^p(\pi \otimes \xi)}^p,
    \]
    by reasoning similar to the above. Since $\pi$ and $\xi$ were arbitrary, this completes the proof.
\end{proof}

\subsubsection{Metric properties of the partitioned network distance}

  To describe the exact sense in which $d_{\Pc_k^p}$ is a distance, we need to introduce some equivalence relations on  partitioned measure networks.
  
\begin{defn}[Strong and weak isomorphism]\label{defn:strong_and_weak_isomorphism}
    We say that $k$-partitioned measure networks $P$ and $P'$ are \define{strongly isomorphic} if, for each $i=1,\ldots,k$, there is a Borel measurable bijection $\phi_i:X_i \to X'_i$ (with Borel measurable inverse) such that $(\phi_i)_\# \mu_i = \mu_i'$, and  $\omega(x,y) = \omega'(\phi_i(x),\phi_j(y))$ for every pair $(x,y) \in X_i \times X_j$. The tuple $(\phi_i) = (\phi_i)_{i=1}^k$ is called a \define{strong isomorphism}. 

   A tuple $(\phi_i)$ of maps $\phi_i:X_i \to X_i'$ is called a \define{weak isomorphism} from $P$ to $P'$ if $(\phi_i)_\# \mu_i = \mu_i'$, and  $\omega(x,y) = \omega'(\phi_i(x),\phi_j(y))$ for $\mu_i \otimes \mu_j$-almost every pair $(x,y) \in X_i \times X_j$. We do not require each $\phi_i$ to be a bijection in this definition.
    
    We say that $P$ and $P'$ are \define{weakly isomorphic} if there exists a third partitioned measure network $\overline{P} = ((\overline{X}_i,\overline{\mu}_i),\overline{\omega})$ and weak isomorphisms $(\phi_i)$ from $\overline{P}$ to $P$ and $(\phi_i')$ from $\overline{P}$ to $P'$. 

    It is straightforward to show that weak isomorphism defines an equivalence relation on $\mathcal{P}$; we use $P \sim P'$ to denote that $P$ is weakly isomorphic to $P'$. For $P \in \mathcal{P}_k^p$, let $[P]$ denote its equivalence class under this relation and let $[\mathcal{P}_k^p]$ denote the collection of all equivalence classes.

    Using the embeddings from Definition \ref{defn:generalized_network_embeddings}, there is an induced notion of weak isomorphism on the spaces of generalized networks. Let $[\mathcal{N}^p]$, $[\mathcal{H}^p]$ and $[\mathcal{A}^p]$ denote the collections of weak isomorphism equivalence classes of measure networks, measure hypernetworks and augmented measure networks, respectively.
\end{defn}

Weak isomorphisms of measure networks and of measure hypernetworks are introduced in \cite{chowdhury2019gromov} and \cite{chowdhury2023hypergraph}, respectively. It is straightforward to show that the induced notions from Definition \ref{defn:strong_and_weak_isomorphism} agree with those already established in the literature. The aforementioned papers show that $d_{\mathcal{N}^p}$ and $d_{\mathcal{H}^p}$ descend to well-defined metrics on $[\mathcal{N}^p]$ and $[\mathcal{H}^p]$, respectively. The following theorem generalizes these results, in light of Theorem \ref{thm:isometric_embeddings}.

\begin{theorem}\label{thm:partitioned_metric}
    The $k$-partitioned network $p$-distance $d_{\Pc_k^p}$ induces a well-defined metric on $[\Pc_k^p]$. 
\end{theorem}

Putting Theorems \ref{thm:isometric_embeddings} and \ref{thm:partitioned_metric} together, we obtain the following corollary. 

\begin{corollary}\label{cor:metrics_and_embeddings}
    The generalized network distances $d_{\mathcal{N}^p}$, $d_{\mathcal{H}^p}$ and $d_{\mathcal{A}^p}$ induce well-defined metrics on $[\mathcal{N}^p]$, $[\mathcal{H}^p]$ and $[\mathcal{A}^p]$, respectively. The embeddings from Definition \ref{defn:generalized_network_embeddings} induce isometric embeddings of each of these spaces into $[\Pc_2^p]$. 
\end{corollary}

    We will abuse notation and continue to denote the induced metric on $[\mathcal{P}_k^p]$ as $d_{\mathcal{P}_k^p}$, and take a similar convention for the other induced metrics.

\begin{remark}
    The case of $\mathcal{N}^p$ in the corollary was proved in \cite{chowdhury2019gromov} and the case of $\mathcal{H}^p$ was proved in \cite{chowdhury2023hypergraph} (those papers assumed boundedness of the $\omega$-functions, but this restriction is easily lifted in those proofs). In \cite{demetci2023revisiting}, a relaxed version of triangle inequality was proved for the augmented network distance in the finite setting. Corollary \ref{cor:metrics_and_embeddings} strengthens \cite[Proposition 1]{demetci2023revisiting} to show that $d_{\mathcal{A}^p}$ satisfies the true (non-relaxed) triangle inequality.
\end{remark}

Theorem~\ref{thm:partitioned_metric} follows as an easy corollary of a more general result in the following section, so we defer its proof to Section~\ref{sec:consequences_and_comparisons}. 

\subsubsection{Geodesics and curvature}

We can say more about the metric properties of the partitioned network distance. To state our next result, we first recall some standard concepts from metric geometry; see \cite{bridson2013metric,burago2022course} as general references.

\begin{defn}[Geodesics and curvature]\label{def:geodesics_and_curvature}
    Let $(X,d)$ be a metric space.
    \begin{enumerate}
        \item A \define{geodesic} between points $x,y \in X$ is a path $\gamma:[0,1] \to X$ with $\gamma(0) = x$, $\gamma(1) = y$ and such that, for all $0\leq s \leq t \leq 1$, we have
        \[
        d(\gamma(s),\gamma(t)) = (t-s) d(x,y).
        \]
        If there is a geodesic joining any two points in $X$, we say that $(X,d)$ is a \define{geodesic space}.
        \item Suppose that $(X,d)$ is a geodesic space. We say that $(X,d)$ has \define{curvature bounded below by zero} if for every geodesic $\gamma:[0,1] \to X$ and every point $x \in X$, we have 
        \[
        d(\gamma(t),x)^2 \geq (1-t) d(\gamma(0),x)^2 + t d(\gamma(1),x)^2 - t(1-t)d(\gamma(0),\gamma(1))^2
        \]
        for all $t \in [0,1]$.
        \item We say that $(X,d)$ is an \define{Alexandrov space of non-negative curvature} if it is a complete geodesic space with curvature bounded below by zero.
    \end{enumerate}
\end{defn}

The next main result follows from a more general result below. We defer its proof to Section~\ref{sec:consequences_alexandrov}.

\begin{theorem}\label{thm:alexandrov_space}
    For any $k \geq 1$, $([\Pc_k^2],d_{\Pc_k^2})$ is an Alexandrov space of non-negative curvature. 
\end{theorem}

Combining Theorem \ref{thm:alexandrov_space} with Corollary \ref{cor:metrics_and_embeddings} immediately yields the following result. 

\begin{corollary}\label{cor:alexandrov_curvature_gen_networks}
    Each of the spaces $([\mathcal{N}^2],d_{\mathcal{N}^2})$, $([\mathcal{H}^2],d_{\mathcal{H}^2})$ and $([\mathcal{A}^2],d_{\mathcal{A}^2})$ is an Alexandrov space of non-negative curvature.
\end{corollary}

\begin{remark}[Prior curvature results]\label{rmk:background_on_curvature_bounds}
    The fact that $([\mathcal{N}^2],d_{\mathcal{N}^2})$ is an Alexandrov space of non-negative curvature was essentially proved by Sturm in \cite[Theorem 5.8]{sturm2023space} (there the space of \emph{symmetric measure networks} was considered, i.e., where the function $\omega:X\times X \to \R$ is assumed to be symmetric; the proof still works if this assumption is dropped, as was observed in \cite{chowdhury2020gromov}). This result is new for the other spaces of generalized networks in Corollary \ref{cor:alexandrov_curvature_gen_networks}.
\end{remark}

\section{Extension to labelled networks}\label{sec:extension_to_labelled_partitioned_measure_networks}

We consider, as an extension to the discussion so far, the setting of $k$-partitioned measure networks where each element of $X_i$ ($1 \leq i \leq k$) is associated with a label element that lives in a metric space. 

\subsection{Labelled partitioned measure networks}

Let us begin by defining a labelled notion of a $k$-partitioned measure network, and the distances between these objects.

\begin{defn}[Labelled $k$-partitioned measure networks]\label{defn:labelled_partitioned_measure_network}
  Let $(\Lambda_i,d_{\Lambda_i})$ ($1 \leq i \leq k$) be fixed metric spaces, which we consider as spaces of \define{labels}. A \define{labelled $k$-partitioned measure $p$-network} is a tuple $\left(P, (\iota_i)_{i = 1}^k\right)$, where
  \begin{itemize}
      \item $P = \left((X_i, \mu_i)_{i = 1}^k, \omega\right) \in \Pc_k^p$;
      \item each of the $\iota_i : X_i \to \Lambda_i$ is a measurable function, which we refer to as a \define{labelling function};
      \item each function $X_i \times X_i \to \R$ defined by $(x,y) \mapsto d_{\Lambda_i}(\iota_i(x),\iota_i(y))$ belongs to $L^p(\mu_i \otimes \mu_i)$.
  \end{itemize} 
  We frequently simplify notation and write $L=\left(P,(\iota_i)\right)$ for a labelled $k$-partitioned network. We denote by $\PL_k^p$ the space of labelled $k$-partitioned measure $p$-networks, where it is understood that the label spaces $(\Lambda_i,d_{\Lambda_i})$ are fixed. 
\end{defn}

\begin{example}[Node-attributed networks]
    The main examples of labelled partitioned networks come from \define{node-attributed network} structures. For example, consider a measure network $N = (X,\mu,\omega) \in \Nc^p$ representing a graph via some graph kernel $\omega$. In applications, the node set $X$ may be attributed with auxiliary data---for example, if the graph encodes user interactions on a social network, then each node may be attributed with additional user-level statistics. This situation can be modelled as a function $\iota:X \to \Lambda$, where $\Lambda$ is the attribute space (e.g., $\Lambda = \R^n$). The structure $(N,\iota)$ defines an element of $\PL_1^p$. 
\end{example}

The partitioned network distance \eqref{defn:partitioned_network_distance} can be naturally generalized to $\PL_k^p$.

\begin{defn}[Labelled partitioned network distance] \label{defn:labelled_partitioned_distance}
  Let $1 \leq p < \infty$ and let $L=\left(P, (\iota_i)\right), L'=\left(P', (\iota'_i)\right) \in \PL_k^p$ be two labelled partitioned measure networks. Then we define the \define{labelled $k$-partitioned network distance} to be
  \begin{equation}\label{eqn:labelled_partitioned_distance}
      d_{\PL_k^p}(L, L') = \inf_{(\pi_i) \in \Pi_k\left((\mu_i), (\mu_i')\right)} \frac{1}{2} \left( \sum_{i, j = 1}^k \| \omega - \omega' \|_{L^p(\pi_i \otimes \pi_j)}^p + \sum_{i = 1}^k \| d_{\Lambda_i} \circ (\iota_i, \iota'_i) \|_{L^p(\pi_i)}^p \right)^{1/p}.
  \end{equation}
  This extends to the $p=\infty$ case as 
\begin{equation}\label{eqn:labelled_partitioned_distance_infty}
      d_{\PL_k^\infty}(L, L') = \inf_{(\pi_i) \in \Pi_k\left((\mu_i), (\mu_i')\right)} \frac{1}{2} \max \left( \max_{i,j} \| \omega - \omega' \|_{L^\infty(\pi_i \otimes \pi_j)}, \max_i \| d_{\Lambda_i} \circ (\iota_i, \iota'_i) \|_{L^\infty(\pi_i)} \right).
  \end{equation}
\end{defn}

\begin{remark}
    We could include a balance parameter to weight the contributions of the network kernel term (i.e., the first summation) versus the labelling function term (the second summation) in \eqref{eqn:labelled_partitioned_distance}. Such a parameter is included in the definition of \emph{Fused Gromov-Wasserstein (FGW) distance}~\cite{vayer2020fused}, which has a similar structure. The connection between $d_{\PL_k^p}$ and FGW distance is explained precisely in Section~\ref{sec:consequences_and_comparisons}. We avoid the inclusion of the balance parameter in our formulation, as it is unimportant from a theoretical standpoint and can been absorbed into the definitions of network kernels and label functions in practical applications.
\end{remark}

A simple (but useful) observation is that the distances can be written as nested $\ell^p$-norms\footnote{In this paper, $L^p$ is the norm defined in terms of a measure, whereas $\ell^p$ is the standard norm on $\R^n$, which does not depend on any measure.}. For the rest of the paper, let $\|\cdot\|_p$ denote the $\ell^p$-norm on $\R^n$ for $p \in [1,\infty]$. We abuse notation and use the same symbol  $\|\cdot\|_p$ for the norm on spaces of various dimensions, with the specific meaning always being clear from context. Then the labelled $k$-partitioned network distance can be expressed as 
\begin{equation}\label{eqn:nested_norm_expression}
d_{\PL_k^p}(L,L') = \inf_{(\pi_i) \in \Pi_k((\mu_i), (\mu_i'))} \frac{1}{2}\left\|\left(\left\|\left( \|\omega - \omega'\|_{L^p(\pi_i \otimes \pi_j)}\right)_{i,j} \right\|_p, \left\|\left(\|d_{\Lambda_i} \circ (\iota_i,\iota_i')\|_{L^p(\pi_i)} \right)_i \right\|_p \right) \right\|_p, 
\end{equation}
for all $p \in [1,\infty]$. We have made one more abuse of notation by considering the collection $\left( \|\omega - \omega'\|_{L^p(\pi_i \otimes \pi_j)}\right)_{i,j}$, which is most naturally indexed as a $k \times k$ matrix, as an element of $\R^{k^2}$ in order to apply the $\ell^p$-norm to it. 

\subsubsection{Metric properties of the labelled distance}

We now show that $d_{\PL_k^p}$ defines a metric, up to a natural notion of equivalence. Strong and weak isomorphisms of partitioned networks (Definition \ref{defn:strong_and_weak_isomorphism}) extend to the case of labelled partitioned measure networks in a straightforward way.

\begin{defn}[Weak isomorphism of labelled partitioned measure networks]
    We say that labelled $k$-partitioned measure networks $L=(P, (\iota_i))$ and $L' = (P', (\iota_i'))$ are \define{strongly isomorphic} if the underlying partitioned measure networks $P$ and $P'$ are strongly isomorphic (see Definition~\ref{defn:strong_and_weak_isomorphism}) via bijections $\phi_i:X_i \to X_i'$ such that $\iota_i(x) = \iota'_i(\phi(x))$ for $\mu_i$-almost every $x \in X_i$. 

    We say that $L$ and $L'$ are \define{weakly  isomorphic} if there exists $\overline{L} = (\overline{P}, (\overline{\iota}_i)) \in \PL_k$, with $\overline{P} = ((\overline{X}_i,\overline{\mu}_i),\overline{\omega})$, such that 
    \begin{itemize}
        \item there exist weak isomorphisms $(\phi_i)$ and $(\phi_i')$ from $\overline{P}$ to $P$ and $P'$, respectively; that is, $\phi_i:\overline{X}_i \to X_i$ and $\phi_i':\overline{X}_i \to X_i'$ satisfy the conditions given in Definition~\ref{defn:strong_and_weak_isomorphism}; 
        \item and the maps $\phi_i$ and $\phi_i'$ additionally satisfy 
        \[
    \overline{\iota}_i(x) = \iota_i(\phi(x)) = \iota'_i(\phi'(x)), 
    \]
    for $\overline{\mu}_i$-almost every $x \in \overline{X}_i$.
    \end{itemize}
    One can easily verify that weak isomorphism again defines an equivalence relation on $\PL_k^p$, and we write $[P,(\iota_i)]$ for equivalence classes and $[\PL_k^p]$ for the quotient space.
\end{defn}

The next theorem is analogous to Theorem~\ref{thm:partitioned_metric}, which establishes the metric properties of $d_{\Pc_k^p}$. In fact, the deferred proof of Theorem~\ref{thm:partitioned_metric} will follow easily from this result (see Section \ref{sec:consequences_and_comparisons}).

\begin{theorem}\label{thm:partitioned_metric_labelled}
    The labelled $k$-partitioned network $p$-distance $d_{\PL_k^p}$ induces a well-defined metric on $[\PL_k^p]$. 
\end{theorem}

The proof will use some important technical lemmas.

\begin{lemma}\label{lem:infimum_achieved}
    The infima in \eqref{eqn:labelled_partitioned_distance} and \eqref{eqn:labelled_partitioned_distance_infty} are always realized by  partitioned couplings.
\end{lemma}

\begin{proof}
    Let $(P,(\iota_i)),(P',(\iota'_i)) \in \PL_k^p$. We have $\Pi_{k}((\mu_i),(\mu'_i)) = \Pi(\mu_1,\mu_1') \times \cdots \times \Pi(\mu_k,\mu_k')$. By \cite[Lemma 1.2]{sturm2023space}, each $\Pi(\mu_i,\mu_i')$ is compact (as a subspace of the space of probability measures on $X_i \times X_i'$, with the weak topology), so it follows that $\Pi_{k}((\mu_i),(\mu_i'))$ is compact as well. By the proof of \cite[Lemma 24]{chowdhury2023hypergraph}, for each $(i,j)$, the function
    \[
        \Pi(\mu_i,\mu'_i) \times \Pi(\mu_j,\mu'_j) \to \mathbb{R}:(\pi_i,\pi_j) \mapsto \|\omega - \omega'\|_{L^p(\pi_i \otimes \pi_j)}
    \]
    is continuous in the $p< \infty$ case and lower semicontinuous in the $p = \infty$ case. Similarly, the function $\Pi(\mu_i, \mu_i') \to \R: \pi_i \mapsto \|d_{\Lambda_i} \circ (\iota_i,\iota_i')\|_{L^p(\pi_i)}$ is continuous (respectively, lower semicontinuous) if $p < \infty$ (respectively, $p = \infty$). It is then straightforward to see that the objectives of \eqref{eqn:labelled_partitioned_distance} and \eqref{eqn:labelled_partitioned_distance_infty} inherit these properties as functions on $\Pi_k((\mu_i),(\mu_i'))$. In either case, it follows from compactness that the infima are achieved. 
\end{proof}

The proofs of Theorem \ref{thm:partitioned_metric_labelled} and results later in the paper will use the following standard result from optimal transport theory. In the statement, and throughout the paper, we use $\mathrm{proj}_i:Y^0 \times Y^1 \times \cdots \times Y^n \rightarrow Y^i$ to denote the coordinate projection map from a product of sets to its $i$th factor.

\begin{lemma}[Gluing Lemma; see, e.g., Lemma 1.4 of \cite{sturm2023space}]\label{lem:gluing_lemma}
    Let $(Y_i,\nu_i)$ be Polish probability spaces (for $i = 0,\ldots,n$). For a collection of measure couplings $\xi_i \in \Pi(\nu_{i-1},\nu_i)$, $i = 1,\ldots,n$, there is a unique probability measure $\tilde{\xi}$ on $Y_0 \times Y_1 \times \cdots \times Y_n$ with the property that
\[
\big(\mathrm{proj}_{i-1} \times \mathrm{proj}_i\big)_\# \tilde{\xi} = \xi_i, 
\]
for all $i = 1,\ldots,n$.
\end{lemma}
The measure $\tilde{\xi}$ from the lemma is called the \define{gluing} of the measures $\xi_i$. We denote it as $\tilde{\xi} = \xi_1 \boxtimes \xi_2 \boxtimes \cdots \boxtimes \xi_n$.

\begin{proof}[Proof of Theorem \ref{thm:partitioned_metric_labelled}]
    The function $d_{\PL_k^p}$ is clearly symmetric. We now establish the triangle inequality. Let $L=(P,(\iota_i)),L'=(P',(\iota'_i)),L''=(P'',(\iota''_i)) \in \PL_k^p$. By Lemma \ref{lem:infimum_achieved} there exist partitioned couplings $(\pi_i) \in \Pi_k ((\mu_i), (\mu'_i))$ and $(\pi_i')\in \Pi_{k}((\mu'_i), (\mu''_i))$ that realize the infima in the distances, respectively, between $P, P'$ and $P', P''$.
    Let $\widetilde{\xi}_i = \pi_i \boxtimes \pi_i' \boxtimes \pi_i''$ denote the probability measure on $X_i \times X_i' \times X_i''$, for $i = 1,\ldots,k$, obtained from the Gluing Lemma (Lemma \ref{lem:gluing_lemma}). Letting $\xi_i$ denote the pushforward to $\widetilde{\xi}_i$ to $X_i\times X''_i$, we have that $(\xi_i) \in \Pi_k((\mu_i),(\mu_i''))$. 
    Using the expression \eqref{eqn:nested_norm_expression}, we get the desired triangle inequality from the triangle inequality for the $L^p$- and $\ell^p$-norms and for $d_{\Lambda_i}$:
    \begin{align*}
        &2 \cdot d_{\PL_k^p}(L,L'') \\
        &\leq \left\|\left(\left\|\left( \|\omega - \omega''\|_{L^p(\xi_i \otimes \xi_j)}\right)_{i,j} \right\|_p, \left\|\left(\|d_{\Lambda_i} \circ (\iota_i,\iota_i'')\|_{L^p(\xi_i)} \right)_i \right\|_p \right) \right\|_p \\
        &= \left\|\left(\left\|\left( \|\omega - \omega''\|_{L^p(\tilde{\xi}_i \otimes \tilde{\xi}_j)}\right)_{i,j} \right\|_p, \left\|\left(\|d_{\Lambda_i} \circ (\iota_i,\iota_i'')\|_{L^p(\tilde{\xi}_i)} \right)_i \right\|_p \right) \right\|_p \\
        &\leq \left\|\left(\left\|\left( \|\omega - \omega' + \omega' - \omega''\|_{L^p(\tilde{\xi}_i \otimes \tilde{\xi}_j)}\right)_{i,j} \right\|_p, \left\|\left(\|d_{\Lambda_i} \circ (\iota_i,\iota'_i) + d_{\Lambda_i} \circ (\iota'_i,\iota_i'')\|_{L^p(\tilde{\xi}_i)} \right)_i \right\|_p \right) \right\|_p \\
        &\leq \left\|\left(\left\|\left( \|\omega - \omega'\|_{L^p(\tilde{\xi}_i \otimes \tilde{\xi}_j)} + \|\omega' - \omega''\|_{L^p(\tilde{\xi}_i \otimes \tilde{\xi}_j)}\right)_{i,j} \right\|_p, \right. \right. \\ 
        &\hspace{1.5in} \left. \left. \left\|\left(\|d_{\Lambda_i} \circ (\iota_i,\iota'_i)\|_{L^p(\tilde{\xi}_i)} + \|d_{\Lambda_i} \circ (\iota'_i,\iota_i'')\|_{L^p(\tilde{\xi}_i)} \right)_i \right\|_p \right) \right\|_p \\
        &\leq \left\|\left(\left\|\left( \|\omega - \omega'\|_{L^p(\tilde{\xi}_i \otimes \tilde{\xi}_j)}\right)_{i,j} \right\|_p + \left\|\left(\|\omega' - \omega''\|_{L^p(\tilde{\xi}_i \otimes \tilde{\xi}_j)}\right)_{i,j} \right\|_p, \right.\right. \\
        &\hspace{1.5in} \left. \left.
        \left\|\left(\|d_{\Lambda_i} \circ (\iota_i,\iota'_i)\|_{L^p(\tilde{\xi}_i)}\right)_i \right\|_p \left\| + \left( \|d_{\Lambda_i} \circ (\iota'_i,\iota_i'')\|_{L^p(\tilde{\xi}_i)} \right)_i \right\|_p \right) \right\|_p \\
        &\leq \left\|\left(\left\|\left( \|\omega - \omega'\|_{L^p(\tilde{\xi}_i \otimes \tilde{\xi}_j)}\right)_{i,j} \right\|_p, \left\|\left(\|d_{\Lambda_i} \circ (\iota_i,\iota_i')\|_{L^p(\tilde{\xi}_i)} \right)_i \right\|_p \right) \right\|_p \\
        &\hspace{1.5in} 
        +\left\|\left(\left\|\left( \|\omega' - \omega''\|_{L^p(\tilde{\xi}_i \otimes \tilde{\xi}_j)}\right)_{i,j} \right\|_p, \left\|\left(\|d_{\Lambda_i} \circ (\iota_i',\iota_i'')\|_{L^p(\tilde{\xi}_i)} \right)_i \right\|_p \right) \right\|_p \\
        &= \left\|\left(\left\|\left( \|\omega - \omega'\|_{L^p(\tilde{\pi}_i \otimes \tilde{\pi}_j)}\right)_{i,j} \right\|_p, \left\|\left(\|d_{\Lambda_i} \circ (\iota_i,\iota_i')\|_{L^p(\tilde{\pi}_i)} \right)_i \right\|_p \right) \right\|_p \\
        &\hspace{1.5in} 
        +\left\|\left(\left\|\left( \|\omega' - \omega''\|_{L^p(\tilde{\pi'}_i \otimes \tilde{\pi'}_j)}\right)_{i,j} \right\|_p, \left\|\left(\|d_{\Lambda_i} \circ (\iota_i',\iota_i'')\|_{L^p(\tilde{\pi'}_i)} \right)_i \right\|_p \right) \right\|_p \\
        &= 2 \cdot d_{\PL_k^p}(L,L') + 2 \cdot d_{\PL_k^p}(L',L'').
    \end{align*}
    We note that lines where measures are changed in the $L^p$-norms follow by marginalization (for example, the first equality which exchanges $\xi_i$ for $\tilde{\xi}_i$ and $\xi_j$ for $\tilde{\xi}_j$). This proves that the triangle inequality holds.
    
    Finally, let us show that $d_{\PL_k^p}(L,L') = 0$ if and only if $L$ and $L'$ are weakly isomorphic. Suppose that $L$ and $L'$ are weakly isomorphic. Let $\overline{L} \in \PL_k$ denote the auxiliary space from the definition of weak isomorphism. It is easy to show that $d_{\PL_k^p}(\overline{L},L) = d_{\PL_k^p}(\overline{L},L') = 0$, so $d_{\PL_k^p}(L,L') = 0$ follows by symmetry and the triangle inequality. Conversely, suppose that $d_{\PL_k^p}(L,L') = 0$. By Lemma \ref{lem:infimum_achieved}, there is a partitioned coupling $(\pi_i)$ such that $\|\omega - \omega'\|_{L^p(\pi_i \otimes \pi_j)} = \|d_{\Lambda_i} \circ (\iota_i,\iota_i')\|_{L^p(\pi_i)} = 0$ for all $i,j = 1,\ldots,k$. Define $\overline{X}_i = X_i \times X'_i$, $\overline{\mu}_i = \pi_i$ and $\overline{\omega}((x,x'),(y,y')) = \omega(x,y)$. The maps $\phi_i:\overline{X}_i \to X_i$ and $\phi'_i:\overline{X}_i \to X'_i$ from the definition of weak isomorphism are coordinate projection maps. One can then show that this gives a weak isomorphism of $P$ and $P'$. Finally, define a new labelling function $\overline{\iota}_i:\overline{X}_i \to \Lambda_i$ by $\overline{\iota}_i(x,x') = \iota_i(x)$. Since $d_{\Lambda_i}$ is a metric, it must be that $\iota_i(x) = \iota_i'(x')$ for $\pi_i$-almost every $(x,x') \in X_i \times X_i'$, so this labelling function satisfies the condition in the definition of weak isomorphism.
\end{proof}

\subsubsection{Consequences and comparisons to other results}\label{sec:consequences_and_comparisons}

We now give a proof of Theorem~\ref{thm:partitioned_metric}, which says that the (unlabelled) partitioned network distance is a metric, and which then implies that various other generalized network distances in the literature are metrics as well (Corollary~\ref{cor:metrics_and_embeddings}). 

\begin{proof}[Proof of Theorem~\ref{thm:partitioned_metric}]
    Consider the map which takes a $k$-partitioned measure network $P$ to the labelled $k$-partitioned measure network $(P,(\iota_i))$, where $(\Lambda_i,d_{\Lambda_i})$ is the one-point metric space for all $i$ (hence $\iota_i$ is the constant map for all $i$). Clearly, we have 
    \[
    d_{\Pc_k^p}(P,P') = d_{\PL_k^p}\big((P,(\iota_i)),(P',(\iota'_i))\big),
    \]
    since the labelling term in the definition of $d_{\PL_k^p}$ vanishes. Thus the map $P \mapsto (P,(\iota_i))$ induces a bijection from $[\Pc_k^p]$ to $[\PL_k^p]$ which takes $d_{\Pc_k^p}$ to $d_{\PL_k^p}$, and it follows that $d_{\Pc_k^p}$ is a metric.
\end{proof}

Next, we give a more precise comparison between the distance $d_{\PL_k^p}$ and the Fused Gromov-Wasserstein (FGW) distance of Vayer et al.~\cite{vayer2020fused}. The FGW distance is defined in the context of labelled measure networks; that is, in the $k=1$ setting, where we write elements as $(N,\iota)$, with $N  = (X,\mu,\omega) \in \Nc$ and $\iota:X \to \Lambda$, and in which case the distance $d_{\PL_1^p}$ reduces to
\begin{equation}\label{eqn:k_equals_one_version}
d_{\PL_1^p}((N,\iota),(N',\iota')) = \inf_{\pi \in \Pi(\mu,\mu')} \frac{1}{2} \left(\|\omega - \omega'\|_{L^p(\pi \otimes \pi)}^p + \|d_\Lambda \circ (\iota,\iota')\|_{L^p(\pi)}^p \right)^{1/p},
\end{equation}
for $p < \infty$. In contrast, the FGW distance depends on several more parameters, but the version of it which is closest to \eqref{eqn:k_equals_one_version} would read as 
\begin{align*}
&d_{\mathrm{FGW},p}((N,\iota),(N',\iota')) \\
&\qquad = \inf_{\pi \in \Pi(\mu,\mu')} \frac{1}{2} \left( \int_{X \times X' \times X \times X'} \left(|\omega(x,y) - \omega'(x',y')| + d_\Lambda(\iota(x),\iota'(x')) \right)^p d\pi(x,x') d\pi(y,y')  \right)^{1/p} \\
&\qquad = \inf_{\pi \in \Pi(\mu,\mu')} \frac{1}{2} \||\omega - \omega'| + d_\Lambda \circ (\iota,\iota')\|_{L^p(\pi \otimes \pi)}.
\end{align*}
Although the distances $d_{\PL_1^p}$ and $d_{\mathrm{FGW},p}$ treat the same type of object, the above shows that their formulations are subtly but legitimately distinct. 

\begin{remark}[Triangle Inequality for Fused Gromov-Wasserstein]
    The situation described above is slightly murky, as several articles following~\cite{vayer2020fused} have formulated the FGW distance more in line with \eqref{eqn:k_equals_one_version}; see  e.g.,~\cite{titouan2019optimal,brogat2022learning,flamary2021pot}. However, as far as we are aware, the primary references for FGW distance have not established a triangle inequality for any of its formulations. In particular, \cite{vayer2020fused,titouan2019optimal} both give relaxed variants of the triangle inequality for different versions of FGW, where the larger side of the inequality involves an extra scale factor. Theorem~\ref{thm:partitioned_metric_labelled} therefore gives a novel proof of the triangle inequality for FGW, when expressed in the form \eqref{eqn:k_equals_one_version}. We note that the triangle inequality for FGW was also recently established via an independent argument in \cite[Corollary 4.3]{bauer2024z}.
\end{remark}

\subsection{Alexandrov geometry of labelled partitioned networks}

Next, we characterize geodesics and curvature in the space of labelled partitioned measure networks. For the rest of the section, we will assume the following conventions:
\begin{itemize}
    \item We assume that the label spaces $(\Lambda_i, d_{\Lambda_i})$ are geodesic spaces. This is sometimes specialized further to assume that the label spaces are Hilbert spaces, but this specialization will always be pointed out explicitly in the statements of our results.
    \item We will restrict our attention to the case $p=2$, and simply write $\PL_k$ in place of $\PL_k^2$.
\end{itemize}
The following is a generalization of Theorem~\ref{thm:alexandrov_space}; recall that the proof of that theorem was deferred---we will prove it in Section~\ref{sec:consequences_alexandrov} as a corollary.

\begin{theorem}\label{thm:alexandrov_space_labelled}
    Let each $\Lambda_i$ ($1 \leq i \leq k$) be a Hilbert space with inner product $\langle \cdot, \cdot \rangle_{\Lambda_i}$. Then, for any $k \geq 1$, $([\PL_k],d_{\PL_k})$ is an Alexandrov space of non-negative curvature. 
\end{theorem}

We will prove this by establishing the necessary properties as propositions. The proof techniques used in this section are largely adapted from the seminal work of Sturm~\cite{sturm2023space}.

\subsubsection{Geodesic structure}

We first prove two results on the geodesic structure of $[\PL_k]$. 

\begin{prop}\label{prop:interpolation_geodesic_labelled}
    For any $k \geq 1$, $([\PL_k],d_{\PL_k})$ is a geodesic space. For labelled $k$-partitioned measure networks $L=(P,(\iota_i)),L'=(P',(\iota'_i)) \in \PL_k$, 
    a geodesic from $[P,(\iota_i)]$ to $[P',(\iota'_i)]$ is given by $[L^t] = [P^t,(\iota^t_i)]$, $t \in [0,1]$, defined as follows. The underlying $k$-partitioned measure network $P^t$ is
    \[
    P^t = \big((X_i \times X_i', \pi_i), \omega^t\big),
    \]
    where $(\pi_i)$ is a $k$-partitioned coupling which realizes $d_{\PL_k}(L,L')$, and $\omega^t: (\sqcup_i X_i) \times (\sqcup_i X_i') \to \R$ is defined by 
    \begin{equation}\label{eqn:omega_interpolation}
    \omega^t((x,x'),(y,y')) = (1-t)\omega(x,y) + t\omega'(x',y').
    \end{equation}
    The labelling function $(\iota_i^t)$ is given by  
    \[
    \iota_i^t: X_i \times X_i' \longrightarrow \Lambda_i, \quad \iota_i^t(x,x') = \gamma_i^{x,x'}(t), \quad x, x' \in X_i \times X_i',
    \]
where $\gamma_i^{x,x'}:[0,1] \to \Lambda_i$ is a geodesic between $\iota_i(x)$ and $\iota_i'(x')$ for each $1 \leq i \leq k$. 
\end{prop}

\begin{proof}
    It is straightforward to show that $L^0$ is weakly isomorphic to $L$ and $L^1$ is weakly isomorphic to $L'$. To show that $[L^t]$ defines a geodesic, it suffices to show that 
    \begin{equation}\label{eqn:geodesic_formula_2_labelled}
    d_{\PL_k}(L^s, L^t) \leq (t-s) d_{\PL_k}(L,L'),
    \end{equation}
    for all $s,t \in [0,1]$ with $s < t$ (see, e.g., \cite[Lemma 1.3]{chowdhury2016explicit}). Let $(\pi_i) \in \Pi_k((\mu_i),(\mu_i'))$ be optimal (Lemma \ref{lem:infimum_achieved}) and set 
    \[
    1_{\pi_i} \coloneqq (\mathrm{id}_{X_i \times X_i'} \times \mathrm{id}_{X_i \times X_i'})_\# \pi_i \in \Pi(\pi_i,\pi_i)
    \]
    for each $i = 1,\ldots,k$. Then
    \begin{equation}\label{eqn:geodesic_formula_1_labelled}
    4 \cdot d_{\PL_k}(L^s,L^t)^2 \leq \sum_{i,j=1}^k \|\omega^s - \omega^t\|_{L^2(1_{\pi_i}\otimes 1_{\pi_j})}^2 +  \sum_{i = 1}^k \| d_{\Lambda_i} \circ (\iota_i^s, \iota_i^t) \|_{L^2(1_{\pi_i})}^2 .
    \end{equation}

    Applying the various definitions, it is straightforward to show that, for each pair $(i,j)$, 
    \[
    \|\omega^s - \omega^t\|_{L^2(1_{\pi_i}\otimes 1_{\pi_j})}^2 = (t-s)^2 \|\omega - \omega'\|_{L^2(\pi_i \otimes \pi_j)}^2.
    \]
To bound the last term in \eqref{eqn:geodesic_formula_1_labelled}, observe that the term $\| d_{\Lambda_i} \circ (\iota_i^s, \iota_i^t) \|_{L^2(1_{\pi_i})}^2  $ is equal to 
\[
  \iint_{(X_i \times X_i')^2} d_\Lambda( \iota_i^s(x,x') ,\iota_i^t(y, y'))^2 \: \diff 1_{\pi}((x, x'), (y, y')) = \int_{X_i \times X_i'} d_{\Lambda_i}(\iota_i^s(x, x'), \iota_i^t(x, x'))^2 \: \diff\pi_i(x, x').
\]
We have that, for all $0 \leq s \leq t \leq 1$,
\[
  d_{\Lambda_i}(\iota_i^s(x,x'), \iota_i^t(x,x')) = d_{\Lambda_i}( \gamma^{x,x'}_i(s) ,\gamma^{x,x'}_i(t)) = (t-s) d_{\Lambda_i}(\iota_i(x) ,\iota_i(x') ),
\]
where the second equality follows by geodesity of $\gamma^{x,x'}$.
This implies 
\[
  \| d_{\Lambda_i} \circ (\iota^s_i, \iota^t_i) \|_{L^2(1_{\pi_i})}^2 = (t-s)^2 \int_{X_i \times X_i'} d_{\Lambda}(\iota_i(x) ,\iota_i(x') \big)^2 \: \diff\pi_i(x, x'). 
\]
Putting all of this together yields the desired inequality \eqref{eqn:geodesic_formula_2_labelled}. 
\end{proof}

\begin{prop} \label{prop:labelled_geodesic_normed}
  Let us now assume that each of the $\Lambda_i, 1 \leq i \leq k$ are inner product spaces with inner products $\langle \cdot,\cdot \rangle_{\Lambda_i}$, associated norms $\| \cdot \|_{\Lambda_i}$ and metrics $d_{\Lambda_i}$ induced by their norms. Then any geodesic in $[\PL_k]$ can be written in the form given in Proposition \ref{prop:interpolation_geodesic_labelled}: for any geodesic $[P^t, (\iota_i^t)], t \in [0, 1]$ between $[L]$ and $[L']$, there exists an optimal coupling $(\pi_i) \in \Pi_k((\mu_i), (\mu_i'))$ such that $[P^t, (\iota_i^t)]$ is weakly isomorphic to $((X_i \times X_i', \pi_i), \omega^t, \gamma_i^t)$ where $\omega^t$ is given by \eqref{eqn:omega_interpolation} and where 
  \[
  \gamma_i^{x,x'}(t) \coloneqq (1-t)\iota_i(x) + t \iota'_i(x').
  \]
\end{prop}

We will use some additional notation and terminology in subsequent proofs. 

\begin{defn}\label{defn:induced_inner_product_space}Let $\Lambda$ be an inner product space with inner product $\langle \cdot, \cdot \rangle_\Lambda$  and induced norm $\|\cdot\|_\Lambda$. For a probability space $(Z,\pi)$, consider the space of functions $\iota:Z \to \Lambda$ such that 
    \[
    \int_{Z} \|\iota(z)\|_{\Lambda}^2 \diff\pi(z) < \infty.
    \]
    We denote the space of such functions, considered up to almost-everywhere equality, as  $L^2(\pi,\Lambda)$. This is an inner product space with inner product defined by 
    \[
    \langle \iota, \iota' \rangle_{L^2(\pi,\Lambda)}\coloneqq \int_{Z} \langle \iota(z), \iota'(z) \rangle_{\Lambda} \diff \pi(z).
    \]
    We let $\|\cdot\|_{L^2(\pi,\Lambda)}$ denote the associated norm.
\end{defn}

\begin{proof}[Proof of Proposition \ref{prop:labelled_geodesic_normed}]
 Let $(P,(\iota_i)),(P',(\iota_i')) \in \PL_k$ and let $[P^t,(\iota_i^t)]$ be an arbitrary geodesic from $[P^0,(\iota^0_i)] = [P,(\iota_i)]$ to $[P^1,(\iota_1^i)] = [P',(\iota_i')]$ with $P^t = \left((X^t_i,\mu^t_i),\omega^t\right) \in \mathcal{P}_k$. We will show that $(P^t,(\iota_i^t))$ is (pointwise, in time) weakly isomorphic to a geodesic in the form described in Proposition \ref{prop:interpolation_geodesic_labelled}.
    
    For each $t \in [0,1]$, let $X^t = \sqcup_i X^t_i$. 
    Fix an integer $n$ and consider a dyadic decomposition of the time domain, $t_0 = 0, t_1 = \frac{1}{2^n}, \ldots, t_i =  \frac{i}{2^n},\ldots, t_{2^n} = 1$. For each $j = 1,\ldots, 2^n$, choose an optimal $k$-partitioned coupling $(\pi^j_i)_{i=1}^k \in \Pi_k\big((\mu^{t_{j-1}}_i),(\mu^{t_j}_i)\big)$ (via Lemma \ref{lem:infimum_achieved}). Consider the gluings (Lemma \ref{lem:gluing_lemma})
    \begin{align*}
    \widetilde{\pi}_i &= \pi_i^1 \boxtimes \cdots \boxtimes \pi_i^{2^n} \in \mathbb{P}\big(X^0 \times X^{2^{-n}} \times X^{2 \cdot 2^{-n}} \times \cdots \times X^1\big), \qquad i \in \{1,\ldots,k\}.
    \end{align*}
    Let $\mathrm{proj}_t: X^0 \times X^{2^{-n}} \times \cdots \times X^1 \to X^t$ denote coordinate projection for each $t \in \{0,2^{-n},\ldots,1-2^{-n},1\}$ and define
    \[
    \pi_i = (\mathrm{proj}_0 \times \mathrm{proj}_1)_\# \widetilde{\pi}_i \in \Pi_k(\mu^0,\mu^1)
    \]
    for each $i =1,\ldots,k$. 
    Then, by suboptimality,
    \begin{equation}\label{eqn:geodesic_uniqueness_0}
    4 \cdot d_{\PL_k}((P^0,(\iota_i^0),(P^1,(\iota^1_i))^2 \leq \sum_{i,j = 1}^k \|\omega^0 - \omega^1\|_{L^2(\pi_i \otimes \pi_j)}^2 + \sum_{i = 1}^k \| d_{\Lambda_i} \circ (\iota_i^0, \iota_i^1) \|_{L^2(\pi_i)}^2.
    \end{equation}

    First, consider the term $\|\omega^0 - \omega^1\|_{L^2(\pi_i \otimes \pi_j)}^2$ on the right hand side of \eqref{eqn:geodesic_uniqueness_0}. For any choice of $t \in \{0,2^{-n},\ldots,1-2^{-n},1\}$, let
    \[
    \xi^t_i \coloneqq (\mathrm{proj}_0 \times \mathrm{proj}_1 \times \mathrm{proj}_t)_\# \widetilde{\pi}_i \in \mathbb{P}(X^0 \times X^1 \times X^t).
    \]
    We have
    \begin{align}
        &\|\omega^0 - \omega^1\|_{L^2(\pi_i \otimes \pi_j)}^2 \nonumber \\
        &= \left\|t \left(\frac{1}{t}\big(\omega^0 - \omega^t\big)\right) + (1-t) \left(\frac{1}{1-t}\big(\omega^t - \omega^1\big)\right)\right\|_{L^2(\xi_i^t \otimes \xi_j^t)}^2 \label{eqn:geodesic_uniqueness_1} \\
        \begin{split}
        &= \frac{1}{t}\left\|\omega^0 - \omega^t \right\|_{L^2(\xi_i^t \otimes \xi_j^t)}^2 + \frac{1}{1-t}\left\|\omega^t - \omega^1 \right\|_{L^2(\xi_i^t \otimes \xi_j^t)}^2 \\
        & \qquad - \frac{1}{t(1-t)}\left\|(1-t)(\omega^0 - \omega^t) - t(\omega^t - \omega^1) \right\|_{L^2(\xi_i^t \otimes \xi_j^t)}^2,
        \end{split} \label{eqn:geodesic_uniqueness_2}
    \end{align}
    where \eqref{eqn:geodesic_uniqueness_1} uses marginalization to replace $\pi_i \otimes \pi_j$ with $\xi^t_i \otimes \xi^t_j$, and where  \eqref{eqn:geodesic_uniqueness_2} is derived by applying the following identity, which holds in an arbitrary inner product space with associated norm $\|\cdot\|$:
\begin{equation}\label{eqn:inner_product_identity}
    \|ta + (1-t)b\|^2 = t \|a\|^2 + (1-t)\|b\|^2 - t(1-t)\|a-b\|^2.
    \end{equation}
    Bearing in mind that $t = k2^{-n}$ for some $k$, the first term in \eqref{eqn:geodesic_uniqueness_2} satisfies
    \begin{align}
        \frac{1}{t}\left\|\omega^0 - \omega^t \right\|_{L^2(\xi_i^t \otimes \xi_j^t)}^2 &= 2^n \cdot \frac{1}{k} \left\|\omega^0 - \omega^{k2^{-n}} \right\|_{L^2(\xi_i^t \otimes \xi_j^t)}^2 \nonumber \\
        &=  2^n \cdot \frac{1}{k} \left\|\sum_{\ell = 1}^k (\omega^{(\ell-1)2^{-n}} - \omega^{\ell 2^{-n}}) \right\|_{L^2(\xi_i^t \otimes \xi_j^t)}^2 \nonumber \\
        &\leq 2^n \cdot \frac{1}{k} \left(\sum_{\ell = 1}^k \left\|\omega^{(\ell-1)2^{-n}} - \omega^{\ell 2^{-n}}\right\|_{L^2(\xi_i^t \otimes \xi_j^t)}\right)^2  \label{eqn:geodesic_uniqueness_3}\\
        &\leq 2^n \sum_{\ell = 1}^k \left\|\omega^{(\ell-1)2^{-n}} - \omega^{\ell 2^{-n}} \right\|_{L^2(\xi_i^t \otimes \xi_j^t)}^2, \label{eqn:geodesic_uniqueness_4} 
    \end{align}
    where \eqref{eqn:geodesic_uniqueness_3} follows by the triangle inequality for the $L^2$-norm and \eqref{eqn:geodesic_uniqueness_4} is Jensen's inequality.
    Similarly, the second term in \eqref{eqn:geodesic_uniqueness_2} satisfies
    \[
    \frac{1}{1-t}\left\|\omega^t - \omega^1 \right\|_{L^2(\xi_i^t \otimes \xi_j^t)}^2 \leq 2^n \sum_{\ell = k+1}^{2^n} \left\|\omega^{(\ell-1)2^{-n}} - \omega^{\ell 2^{-n}} \right\|_{L^2(\xi_i^t \otimes \xi_j^t)}^2,
    \]
    so that, after marginalizing, we have
\begin{equation}\label{eqn:unique_geodesic_estimate_1}
    \frac{1}{t}\left\|\omega^0 - \omega^t \right\|_{L^2(\xi_i^t \otimes \xi_j^t)}^2 + \frac{1}{1-t}\left\|\omega^t - \omega^1 \right\|_{L^2(\xi_i^t \otimes \xi_j^t)}^2 \leq 2^n \sum_{\ell = 1}^{2^n} \left\|\omega^{(\ell-1)2^{-n}} - \omega^{\ell2^{-n}} \right\|_{L^2\big(\pi_i^{(\ell-1)2^{-n}} \otimes \pi_j^{\ell2^{-n}}\big)}^2.
    \end{equation}

    Next, consider the term $\|d_{\Lambda_i} \circ (\iota_i^0,\iota_i^1)\|^2_{L^2(\pi_i)}$. The following uses the notation of Definition~\ref{defn:induced_inner_product_space}. We have, similar to the above,
    \begin{align*}
        &\|d_{\Lambda_i} \circ (\iota_i^0,\iota_i^1)\|^2_{L^2(\pi_i)}  \\
        &= \left\|t \left(\frac{1}{t}(\iota_i^0 - \iota_i^t)\right) + (1-t)\left(\frac{1}{1-t}(\iota_i^t - \iota_i^1)\right)\right\|^2_{L^2(\xi_i^t,\Lambda_i)} \\
        &= \frac{1}{t} \|\iota_i^0 - \iota_i^t\|^2_{L^2(\xi_i^t,\Lambda)} + \frac{1}{1-t} \|\iota_i^t - \iota_i^1\|^2_{L^2(\xi_i^t,\Lambda)} - \frac{1}{t(1-t)} \|(1-t)(\iota_i^0 - \iota_i^t) - t (\iota_i^t - \iota_i^1)\|^2_{L^2(\xi_i^t,\Lambda)},
    \end{align*}
    where we have used the definition of $d_{\Lambda_i}$, as well as marginalization and the general identity \eqref{eqn:inner_product_identity}. Repeating the arguments above, we obtain
\begin{equation}\label{eqn:unique_geodesic_estimate_2}
    \frac{1}{t}\left\|\iota_i^0 - \iota_i^t \right\|_{L^2(\xi_i^t,\Lambda_i)}^2 + \frac{1}{1-t}\left\|\iota_i^t - \iota_i^1 \right\|_{L^2(\xi_i^t,\Lambda_i)}^2 \leq 2^n \sum_{\ell = 1}^{2^n} \left\|d_{\Lambda_i} \circ \big(\iota_i^{(\ell-1)2^{-n}},\iota_i^{\ell2^{-n}}\big) \right\|_{L^2\big(\pi_i^{(\ell-1)2^{-n}}\big)}^2.
    \end{equation}

Summing the right hand sides of \eqref{eqn:unique_geodesic_estimate_1} and \eqref{eqn:unique_geodesic_estimate_2} over all $i,j = 1,\ldots,k$ gives
\begin{align}
    &\sum_{i,j= 1}^k 2^n \sum_{\ell = 1}^{2^n} \left\|\omega^{(\ell-1)2^{-n}} - \omega^{\ell2^{-n}} \right\|_{L^2\big(\pi_i^{(\ell-1)2^{-n}} \otimes \pi_j^{\ell2^{-n}}\big)}^2 + \sum_{i=1}^k \sum_{\ell = 1}^{2^n} \left\|d_{\Lambda_i} \circ \big(\iota_i^{(\ell-1)2^{-n}},\iota_i^{\ell2^{-n}}\big) \right\|_{L^2\big(\pi_i^{(\ell-1)2^{-n}}\big)}^2 \nonumber \\
    &= 2^n \sum_{\ell = 1}^{2^n} \left(\sum_{i,j= 1}^k \left\|\omega^{(\ell-1)2^{-n}} - \omega^{\ell2^{-n}} \right\|_{L^2\big(\pi_i^{(\ell-1)2^{-n}} \otimes \pi_j^{\ell2^{-n}}\big)}^2 \right. \nonumber \\
    &\qquad \qquad \qquad \left. + \sum_{i=1}^k  \left\|d_{\Lambda_i} \circ \big(\iota_i^{(\ell-1)2^{-n}},\iota_i^{\ell2^{-n}}\big) \right\|_{L^2\big(\pi_i^{(\ell-1)2^{-n}}\big)}^2\right) \nonumber \\
    &= 2^n \sum_{\ell = 1}^{2^n} 4 \cdot d_{\PL_k}\big((P^{(\ell-1)2^{-n}},(\iota_i^{(\ell-1)2^{-n}})),(P^{\ell 2^{-n}},(\iota_i^{(\ell)2^{-n}}))\big)^2 \label{eqn:geodesic_uniqueness_5} \\
    &= 4 \cdot d_{\PL_k}\big((P^0,(\iota_i^0)),(P^1,(\iota^1_i))\big)^2, \label{eqn:geodesic_uniqueness_6}
\end{align}
    where \eqref{eqn:geodesic_uniqueness_5} follows by the optimality of the $\pi_i^j$'s and \eqref{eqn:geodesic_uniqueness_6} follows because $[P^t]$ is assumed to be a geodesic. Combining this with \eqref{eqn:geodesic_uniqueness_0}, we have 
    \begin{align*}
        &d_{\PL_k}\big((P^0,(\iota_i^0)),(P^1,(\iota^1_i))\big)^2 \\
        &\qquad \leq d_{\PL_k}\big((P^0,(\iota_i^0)),(P^1,(\iota^1_i))\big)^2   \\
        &\hspace{0.75in} - \frac{1}{4 (t(1-t))}\left(\sum_{i,j = 1}^k \left\|(1-t)(\omega^0 - \omega^t) - t(\omega^t - \omega^1) \right\|_{L^2(\xi_i^t \otimes \xi_j^t)}^2\right. \\
        &\hspace{2in} \left.+ \sum_{i=1}^k \|d_{\Lambda_i} \circ \big((1-t)(\iota_i^0 - \iota_i^t),t (\iota_i^t - \iota_i^1)\big)\|^2_{L^2(\xi_i^t)}\right), 
    \end{align*}
    so that the term in parentheses on the right hand side must vanish. This shows that the partitioned coupling $(\pi_i)$ which we have constructed is, in fact, optimal. We also have that, for all $t$ in the dyadic decomposition,
    \begin{align*}
        0 &= \sum_{i,j = 1}^k  \left\|(1-t)(\omega^0 - \omega^t) - t(\omega^t - \omega^1) \right\|_{L^2(\xi_i^t \otimes \xi_j^t)}^2 = \sum_{i,j = 1}^k \left\|\big((1-t)\omega^0 + t \omega^1\big) - \omega^t\right\|_{L^2(\xi_i^t \otimes \xi_j^t)}^2.
    \end{align*}
    Similarly, 
    \begin{align*}
    0 &= \sum_{i=1}^k \|d_{\Lambda_i} \circ \big((1-t)(\iota_i^0 - \iota_i^t),t (\iota_i^t - \iota_i^1)\big)\|^2_{L^2(\xi_i^t)} \\
    &= \sum_{i=1}^k \|\big((1-t)\iota_i^0 + t\iota_i^1\big) - \iota_i^t\|^2_{L^2(\xi_i^t, \Lambda_i)} = \sum_{i=1}^k \|d_{\Lambda_i} \circ \big(\big((1-t)\iota_i^0 + t\iota_i^1\big), \iota_i^t\big)\|^2_{L^2(\xi_i^t)}.
    \end{align*}
    Observe that, by the properties described in the Gluing Lemma (Lemma \ref{lem:gluing_lemma}), we have $\xi_i^t \in \Pi(\pi_i,\mu_i^t)$, so that the above calculation shows $d_{\PL_k}\big((P^t,(\iota_i^t)),(\overline{P}^t,(\overline{\iota}^t_i))\big) = 0$, where $(\overline{P}^t,(\overline{\iota}^t_i))$ is a geodesic as in the specific construction from Proposition~\ref{prop:interpolation_geodesic_labelled}. 

    So far, we have shown that $d_{\PL_k}\big((P^t,(\iota_i^t)),(\overline{P}^t,(\overline{\iota}^t_i))\big) = 0$ for any $t$ in the form of a dyadic number, i.e., $t = j2^{-n}$ for some $j$ and $n$. By the density of the dyadic numbers in $[0,1]$ and by continuity of the maps $t \mapsto [\overline{P}^t,(\overline{\iota}^t_i)]$ and $t \mapsto [P^t,(\iota^t_i)]$, it follows that $d_{\PL_k}\big((P^t,(\iota_i^t)),(\overline{P}^t,(\overline{\iota}^t_i))\big) = 0$ holds for any $t \in [0,1]$. This completes the proof.
\end{proof}

\subsubsection{Completeness and curvature}

We now complete the proof of Theorem~\ref{thm:alexandrov_space_labelled} by establishing the remaining required properties. Throughout this section, we suppose that the label spaces $\Lambda_i$ are Hilbert spaces with the same notation as Proposition \ref{prop:interpolation_geodesic_labelled} used for inner products and norms.

We first show that the space of labelled networks is complete. The proof will use the following result.

\begin{lemma}[See, e.g.,~\cite{korevaar1993sobolev}]\label{lem:induced_hilbert_space}
    If $\Lambda$ is a Hilbert space, then so is $L^2(\pi,\Lambda)$.
\end{lemma}

\begin{prop}\label{prop:complete_labelled}
    Let each $\Lambda_i, 1 \leq i \leq k$ be a Hilbert space. Then the space $\big([\mathcal{LP}_k],d_{\PL_k}\big)$ is complete. 
\end{prop}

\begin{proof}
    The proof follows the strategy of the proofs of \cite[Theorem 5.8]{sturm2023space} or \cite[Theorem 1]{chowdhury2023hypergraph}, so we treat it somewhat tersely. Let $[P^n,(\iota^n_i)]$,  $n \ge 1$, be a Cauchy sequence of labelled partitioned networks in $[\PL_k]$, with $P^n = \left((X^n_i,\mu^n_i),\omega^n\right)$. Assume, without loss of generality (via a subsequence argument), that $d_{\PL_k}\left(\left(P_n,(\iota^n_i)\right), \left(P_{n+1},(\iota^{n+1}_i)\right)\right) \leq 2^{-n}$. Invoking Lemma \ref{lem:infimum_achieved}, we may choose  partitioned couplings  $(\pi^{n}_i)$ for each $n$ which achieve $d_{\PL_k}\left(\left(P_n,(\iota^n_i)\right), \left(P_{n+1},(\iota^{n+1}_i)\right)\right)$. Gluing the first $N$ of these measures yields a probability measure $\pi_i^1 \boxtimes \pi_i^2 \boxtimes \cdots \boxtimes \pi_i^N$ on $X_i^1 \times X_i^2 \times \cdots \times X_i^N$ for each $i = 1,\ldots,k$. Let $\pi_i$ denote the projective limit measure on $\Pi_{\ell = 1}^\infty X_i^\ell$. 
    
    For each $N$, define maps
    \begin{align*}
        \Omega^N: \left(\bigsqcup_{i=1}^k \Pi_{\ell =1}^\infty X_i^\ell\right) \times \left(\bigsqcup_{i=1}^k \Pi_{\ell =1}^\infty X_i^\ell\right)  &\to \R \\
        \left((x^\ell)_\ell, (y^\ell)_\ell\right) &\mapsto \omega^N(x^N,y^N)
    \end{align*}
    and 
    \begin{align*}
        I^N_i: \Pi_{\ell = 1}^\infty X_i^\ell &\to \Lambda_i \\
        (x^\ell)_\ell &\mapsto \iota^N_i(x^N)
    \end{align*}
    Since $d_{\PL_k}\left(\left(P_n,(\iota^n_i)\right), \left(P_{n+1},(\iota^{n+1}_i)\right)\right) \leq 2^{-n}$, it must be that 
    \[
    \frac{1}{4} \| \omega_n - \omega_{n+1} \|_{L^2(\pi_i \otimes \pi_j)}^2 \leq 2^{-2n} \quad \mbox{and} \quad \frac{1}{4} \| \iota_{i, n} - \iota_{i, n+1} \|_{L^2(\pi_i, \Lambda_i)}^2 \leq 2^{-2n},
  \]
  where we use the notation of Definition~\ref{defn:induced_inner_product_space} in the second term. It follows that the sequence $(\Omega^N)$ is Cauchy in the Hilbert space $L^2(\pi_i \otimes \pi_j)$ and that $(I_i^N)$ is Cauchy in the Hilbert space $L^2(\pi_i,\Lambda_i)$ (this is Hilbert because we assumed that $\Lambda_i$ is Hilbert; see Lemma~\ref{lem:induced_hilbert_space}). Let $\Omega \coloneqq \lim_{N \to \infty} \Omega^N$ and $I_i \coloneqq \lim_{N \to \infty} I_i^N$. 

  Putting these constructions together, we have constructed a labelled $k$-partitioned network
  \[
  \left(\left(\left(\Pi_{\ell=1}^\infty X_i^\ell, \pi_i \right), \Omega \right), \left(I_i\right) \right).
  \]
  One can then show that its weak isomorphism class  is the limit of the original Cauchy sequence.
\end{proof}

Finally, we establish a curvature bound for the space of labelled partitioned networks.

\begin{prop}
    Assume that all label spaces $\Lambda_i$ are Hilbert spaces. Then the space $([\PL_k], d_{\PL_k})$ has curvature bounded below by zero.
\end{prop}

\begin{proof}
    We need to establish the triangle comparison inequality from Definition~\ref{def:geodesics_and_curvature}. Let $[L], [L'] \in [\PL_k]$ be two labelled partitioned networks and let $[L^t], 0 \leq t \leq 1$ be a geodesic connecting them. Let $[L''] \in [\PL_k]$ be given. We seek to show that
  \begin{equation}\label{eqn:three_point_condition_labelled}
    4 d_{\PL_k}(L^t, L'')^2 + 4t(1-t) d_{\PL_k}(L, L')^2 \geq 4(1-t) d_{\PL_k}(L, L'')^2 + 4t d_{\PL_k}(L', L'')^2. 
  \end{equation}
  
  Using the characterization of geodesics in $[\PL_k]$ from Proposition  \ref{prop:labelled_geodesic_normed}, we may assume without loss of generality that $L^t = (P^t,(\iota_i^t))$ has the form described in Proposition \ref{prop:interpolation_geodesic_labelled}. 
  Let $(\xi_i)$ be an optimal $k$-partitioned coupling of $L^t$ to $L''$; then $\xi_i$ is supported on $X_i \times X_i' \times X_i''$. Expanding the left hand side of \eqref{eqn:three_point_condition_labelled}, we have 
  \begin{align*}
    \sum_{i, j = 1}^k &\left( \| \omega'' - \omega^t \|_{L^2(\xi_i \otimes \xi_j)}^2 + t(1-t) \| \omega - \omega' \|_{L^2(\pi_i \otimes \pi_j)}^2 \right) \\
    &\qquad + \sum_{i = 1}^k \left( \| \iota'' - \iota^t \|_{L^2(\xi_i, \Lambda_i)}^2 + t(1-t) \| \iota - \iota' \|_{L^2(\pi_i, \Lambda_i)}^2 \right).
  \end{align*}
  Marginalizing $\xi_i$ and using the structures of $\omega^t$ and $\iota^t$, this can be rewritten as
  \begin{align*}
    &\sum_{i, j = 1}^k \left( \| \omega'' - \omega^t \|_{L^2(\xi_i \otimes \xi_j)}^2 + t(1-t) \| \omega - \omega' \|_{L^2(\xi_i \otimes \xi_j)}^2 \right) \\
    &\qquad + \sum_{i = 1}^k \left( \| \iota'' - \iota^t \|_{L^2(\xi_i, \Lambda_i)}^2 + t(1-t) \| \iota - \iota' \|_{L^2(\xi_i, \Lambda_i)}^2 \right) \\
    &= \sum_{i, j = 1}^k \left( (1-t) \| \omega'' - \omega \|_{L^2(\xi_i \otimes \xi_j)}^2 + t \| \omega'' - \omega' \|_{L^2(\xi_i \otimes \xi_j)}^2 \right) \\
    &\qquad + \sum_{i = 1}^k \left( (1-t) \| \iota'' - \iota \|_{L^2(\xi_i, \Lambda_i)}^2 + t \| \iota'' - \iota' \|_{L^2(\xi_i, \Lambda_i)}^2 \right) \\ 
    &\geq 4(1-t) d_{\PL_k}(P, P'')^2 + 4t d_{\PL_k}(P', P'')^2,
  \end{align*}
  where the second line follows from a computation which holds in an arbitrary inner product space, explained below, and the inequality in the last line follows from sub-optimality. To expand on the inner product space calculation, we use the identity \eqref{eqn:inner_product_identity} to deduce that (for an arbitrary inner product $\langle \cdot, \cdot \rangle$ with norm $\|\cdot\|$), 
  \begin{align*}
      &\|c - ((1-t)a + tb)\|^2 + t(1-t)\|a - b\|^2 \\
      &\qquad = \|c\|^2 + \|(1-t)a + tb\|^2 - 2\langle c, (1-t)a + tb \rangle + t(1-t)\|a - b\|^2 \\
      &\qquad = \|c\|^2 + t\|b\|^2 + (1-t)\|a\|^2 - 2\langle c, (1-t)a + tb \rangle \\
      &\qquad = (1-t)\big(\|c\|^2 + \|a\|^2 - 2\langle c, a\rangle \big) + t\big(\|c\|^2 + \|b\|^2 - 2\langle c, b\rangle \big) \\
      &\qquad = (1-t)\|c - a\|^2 + t \|c - b\|^2.
  \end{align*}
\end{proof}

\subsubsection{The case of unlabelled networks}\label{sec:consequences_alexandrov}

We now proceed with the deferred proof of Theorem~\ref{thm:alexandrov_space}, which says that the space of (unlabelled) $k$-partitioned networks is an Alexandrov space of non-negative curvature.

\begin{proof}[Proof of Theorem~\ref{thm:alexandrov_space}]
    Following the proof of Theorem~\ref{thm:partitioned_metric} (Section \ref{sec:consequences_and_comparisons}), we can consider $d_{\Pc_k^2}$ as an instance of $d_{\PL_k^2}$, where the target metric spaces $\Lambda_i$ are all taken to be the one-point space, which can be considered as a trivial Hilbert space. The proof then follows immediately from Theorem~\ref{thm:alexandrov_space_labelled}.
\end{proof}

\subsection{Interpretation of partitioned distance as a labelled distance}
\label{sec:interpretation}

We end this section by proving that the $k$-partitioned network distance $d_{\Pc_k^p}$ can itself be realized as a sort of labelled distance, where labels are allowed to take the value $\infty$. To keep exposition clean, we recapitulate and rework some of our notation before precisely stating our result.

\subsubsection{Notation for networks labelled in extended metric spaces}

In this section, we write $\mathcal{LN}^p_{\mathrm{ext}}$ for the space of \define{networks with labels in a (fixed) extended metric space} $\Lambda$. This is essentially the same as $\PL_1^p$, except that the label space is allowed to be an \define{extended metric space}, or a metric space whose distances are allowed to take the value $\infty$. Elements of $\mathcal{LN}^p_{\mathrm{ext}}$ will be denoted $(N,\iota)$, where $N = (X,\mu,\omega) \in \Nc^p$ and $\iota:X \to \Lambda$ is the label function. The labelled network distance is then defined by a formula identical to \eqref{eqn:k_equals_one_version}:
\[
d_{\mathcal{LN}_{\mathrm{ext}}^p}((N,\iota),(N',\iota')) = \inf_{\pi \in \Pi(\mu,\mu')} \frac{1}{2} \left(\|\omega - \omega'\|_{L^p(\pi \otimes \pi)}^p + \|d_\Lambda \circ (\iota,\iota')\|_{L^p(\pi)}^p \right)^{1/p}.
\]
The notion of weak isomorphism extends to $\mathcal{LN}^p_{\mathrm{ext}}$; we denote the quotient space as $[\mathcal{LN}^p_{\mathrm{ext}}]$. The proofs above also extend to show that $d_{\mathcal{LN}_{\mathrm{ext}}^p}$ induces a well-defined extended metric on $[\mathcal{LN}^p_{\mathrm{ext}}]$.

We now consider the particular extended metric space $\Lambda^{(k)} \coloneqq \{1,\ldots,k\}$, with extended metric satisfying $d_{\Lambda^{(k)}}(i,j) = \infty$ for all $i \neq j$. To a $k$-partitioned measure network $P$, we associate an element of $\mathcal{LN}^p_{\mathrm{ext}}$, with labels in $\Lambda^{(k)}$, via the map
\begin{equation}\label{eqn:k_labelled_embedding}
\Pc_k^p \ni P = ((X_i,\mu_i),\omega) \mapsto \left((X, (1/k)\mu, \omega), \iota\right) \in \mathcal{LN}^p_{\mathrm{ext}}, \qquad \iota(x) = i \Leftrightarrow x \in X_i.
\end{equation}
Here, as in Definition~\ref{defn:partitioned_measure_network}, $\mu = \sum_i \mu_i$ is considered as a measure on $X = \sqcup_i X_i$, so that $\frac{1}{k}\mu$ is a probability measure.

\subsubsection{Partitioned distance as a labelled network distance}

We now prove the main result of Section~\ref{sec:interpretation}. 

\begin{theorem}\label{thm:fused_interpretation}
    The map \eqref{eqn:k_labelled_embedding} is an isometric embedding with respect to $d_{\mathcal{P}_k^p}$ and $d_{\mathcal{LN}_\mathrm{ext}^p}$.
\end{theorem}

Observe that 
\begin{equation}\label{eqn:fused_term_behavior}
\|d_{\Lambda^{(k)}} \circ (\iota,\iota')\|_{L^p(\pi)} = \left\{\begin{array}{ll}
\infty & \mbox{if there exists $(x,x') \in \mathrm{supp}(\pi)$ with $\iota(x) \neq \iota'(x')$}; \\
0 & \mbox{otherwise.}
\end{array}\right.
\end{equation}
The proof of the theorem is based on this observation and the following lemma.

\begin{lemma}\label{lem:fused_gw_coupling}
    Let $\pi \in \Pi(\sqcup_i \mu_i, \sqcup_i \mu_i')$ such that $\|d_{\Lambda^{(k)}} \circ (\iota,\iota')\|_{L^p(\pi)}=0$. Then there exists a unique $k$-partitioned coupling $(\pi_i) \in \Pi_k\big((\mu_i),(\mu_i')\big)$ such that 
    \[
    \pi = \mathrm{inc}_\#\left( \sqcup_i \pi_i\right),
    \]
    where $\mathrm{inc}:\sqcup_i (X_i \times X_i') \to (\sqcup_i X_i) \times (\sqcup_i X_i')$ is the inclusion map. Moreover, any coupling of this form yields $\|d_{\Lambda^{(k)}} \circ (\iota,\iota')\|_{L^p(\pi)} = 0$.
\end{lemma}

\begin{proof}
    Suppose that $\pi$ is a coupling with $\|d_{\Lambda^{(k)}}\circ (\iota,\iota')\|_{L^p(\pi)} < \infty$. Then $(x,x') \in \mathrm{supp}(\pi)$ implies $\iota(x) = \iota'(x')$, i.e., $x \in X_i$ and $x' \in X_i'$ for some common index $i$. Therefore the support of $\pi$ is contained in $\sqcup_i (X_i \times X_i')$. We also see that the mass of each block $X_i \times X_i'$ must be equal to $\frac{1}{k}$, since, for all $i$,
    \[
    \pi(X_i \times X_i') = \pi((\sqcup_j X_j) \times X_i') = \sqcup_j \mu_j' (X_i') = \frac{1}{k} \sum_j \mu_j'(X_i') = \frac{1}{k} \mu_i'(X_i') = \frac{1}{k}.
    \]
    
    For each $i$, define $\pi_i$ by $\pi_i(A) = k \cdot \pi(A)$ for each Borel set $A \subset X_i \times X_i'$. We claim that $\pi_i \in \Pi(\mu_i,\mu_i')$. Indeed, for any Borel set $B \subset X_i$, 
    \[
    \pi_i(B \times X_i') = k \cdot \pi(B \times X_i') = k \cdot \pi(B \times (\sqcup_i X_i')) = k \cdot \sqcup_j \mu_j (B \cap X_j) = \mu_i(B),
    \]
    and the other marginal condition follows similarly. 
    
    We will now show that
    \[
    \pi(C) = \mathrm{inc}_\#\left( \sqcup_i \pi_i\right)(C)
    \]
    holds for any Borel subset $C$ of $(\sqcup_i X_i) \times (\sqcup_i X_i')$. In light of the discussion above, we may assume without loss of generality that $C \subset \sqcup_i (X_i \times X_i')$. Then we have
    \begin{align*}
        \mathrm{inc}_\# (\sqcup_i \pi_i)(C) &= \sqcup_i \pi_i (\iota^{-1}(C)) = \frac{1}{k} \sum_i \pi_i(\mathrm{inc}^{-1}(C) \cap (X_i \times X_i')) \\
        &= \frac{1}{k} \sum_i \pi_i(C \cap (X_i \times X_i')) = \sum_i \pi(C \cap (X_i \times X_i')) = \pi(C).
    \end{align*}
    This shows the existence part of the statement.

    To prove uniqueness, suppose that $(\pi_i)$ satisfies $\pi = \mathrm{inc}_\# (\sqcup_i \pi_i)$. For a Borel set $A \subset X_i \times X_i'$
    \[
    \pi_i(A) = \sum_j \pi_j(A) = k \cdot \sqcup_j \pi_j (A) = k\cdot \mathrm{inc}_\#(\sqcup_j \pi_j) (A) = k \cdot \pi(A),
    \]
    so the formula for $\pi_i$ is unique.

    Finally, the last statement follows because  any coupling of this form is supported on $\sqcup_i (X_i \times X_i')$. 
\end{proof}

\begin{proof}[Proof of Theorem \ref{thm:fused_interpretation}]
    Let $P = ((X_i,\mu_i),\omega)$ and $P' = ((X'_i,\mu'_i),\omega')$ be elements of $\Pc_k^p$ with  images under the map \eqref{eqn:k_labelled_embedding} denoted $(N,\iota)$ and $(N',\iota')$, respectively. A $k$-partitioned coupling $(\pi_i)$ of $(\mu_i)$ and $(\mu_i')$ yields a coupling $\pi = \mathrm{inc}_\# ((1/k) \sum_i \pi_i) \in \Pi((1/k) \sum_i \mu_i, (1/k) \sum_i \mu_i')$, as in Lemma \ref{lem:fused_gw_coupling}. Then 
    \begin{align*}
        \frac{1}{2}\|\omega - \omega'\|_{L^p(\pi \otimes \pi)} + \|d_{\Lambda^{(k)}} \circ (\iota,\iota')\|_{L^p(\pi)} &= \frac{1}{2} \left(\sum_{i,j=1}^k \|\omega - \omega'\|^p_{L^p(\pi_i\otimes \pi_j)}\right)^{1/p},
    \end{align*}
    so that $d_{\mathcal{LN}_\mathrm{ext}^p}((N,\iota),(N',\iota')) \leq d_{\Pc_k^p}(P,P')$. Similarly, for any $\pi \in \Pi((1/k)\sum_i \mu_i, (1/k)\sum_i \mu_i')$ with $\|d_{\Lambda^{(k)}} \circ (\iota,\iota')\|_{L^p(\pi)}=0$, we can find a $k$-partitioned coupling $(\pi_i) \in \Pi_k((\mu_i),(\mu_i'))$ via Lemma \ref{lem:fused_gw_coupling} and use it to show that $d_{\mathcal{LN}_\mathrm{ext}^p}((N,\iota),(N',\iota')) \geq d_{\Pc_k^p}(P,P')$.
\end{proof}

The following corollary shows that the various generalized network distances which have appeared in the recent literature can all essentially be considered as special cases of the labelled extended network distance. The result follows as a direct consequence of Corollary~\ref{cor:metrics_and_embeddings} and Theorem~\ref{thm:fused_interpretation}.

\begin{corollary}\label{cor:metrics_and_embeddings2}
     The embeddings from Definition \ref{defn:generalized_network_embeddings} induce isometric embeddings of the space of measure networks  $([\mathcal{N}],d_{\mathcal{N}^p})$, the space of measure hypernetworks $([\mathcal{H}],d_{\mathcal{H}^p})$ and the space of augmented measure networks $([\mathcal{A}],d_{\mathcal{A}^p})$, respectively, into the space of labelled networks $([\mathcal{LN}^p_{\mathrm{ext}}],d_{\mathcal{LN}_{\mathrm{ext}}^p})$. 
\end{corollary}

\subsubsection{Labelled partitioned distance as a labelled network distance}

The work above can be directly adapted to show that the space of labelled $k$-partitioned networks also embeds into the space of labelled networks. Consider the space $\PL_k^p$ of labelled $k$-partitioned $p$-networks with labels in some arbitrary metric spaces $(\Lambda_i,d_{\Lambda_i})$. Now consider the extended metric space $\Lambda = (\sqcup_i \Lambda_i) \times \Lambda^{(k)}$ with extended metric $d_\Lambda$ defined by
\[
d_\Lambda((a,i),(b,j)) = \left\{
\begin{array}{cl}
d_{\Lambda_i}(a,b) & \mbox{if $i=j$ and $a,b \in \Lambda_i$};\\
\infty & \mbox{otherwise.}
\end{array}\right.
\]

Given an element $L = (P,(\iota_i))$ of $\PL_k^p$, with $P = ((X_i,\mu_i),\omega)$, we associate an element of $\mathcal{LN}^p_{\mathrm{ext}}$ with labels in $\Lambda$ via the map
\begin{equation}\label{eqn:k_labelled_embedding_2}
\PL_k^p \ni L \mapsto ((X,(1/k)\mu,\omega),\iota), \qquad \iota(x) = (\iota_i(x),i) \Leftrightarrow x \in X_i.
\end{equation}

The techniques used above likewise yield a proof of the following.

\begin{theorem}\label{thm:fused_interpretation_2}
    The map \eqref{eqn:k_labelled_embedding_2} is an isometric embedding with respect to $d_{\PL_k^p}$ and $d_{\mathcal{LN}_\mathrm{ext}^p}$.
\end{theorem}

\section{Riemannian structure of partitioned networks} \label{sec:riemannian}

We now focus again on the space $\PL_k = \PL_k^2$, endowed with the metric $d_{\PL_k} = d_{\PL_k^2}$.
We consider the scenario where the label spaces $(\Lambda_i,d_{\Lambda_i})$ are Hilbert spaces endowed with their associated distances. We have showed in Theorem \ref{thm:alexandrov_space_labelled} that $([\PL_k],d_{\PL_k})$ is a non-negatively curved Alexandrov space. This property endows $[\PL_k]$ with synthetic versions of various structures seen in Riemannian geometry, such as tangent spaces and exponential maps~\cite{plaut2001metric}. Rather than following the general constructions of these structures, we follow the approach of Sturm in~\cite[Chapter 6]{sturm2023space} and develop equivalent versions which are more specific to the metric at hand. In this section, we describe these structures and present some example applications to geometric data analysis.

\subsection{Tangent spaces}

We develop notions of tangent spaces and exponential maps for $[\Pc_k]$ and $[\PL_k]$.  These concepts are introduced in detail for $[\PL_k]$, and the case of $[\Pc_k]$ then follows by considering partitioned networks as special cases of labelled partitioned networks, as in the proof of Theorem \ref{thm:partitioned_metric} (see Section~\ref{sec:consequences_and_comparisons}).

\subsubsection{The labelled case}
\label{sec:labelled-case}

For clarity, we remind the reader of some notational conventions, while introducing some new ones. Let $L = (P,(\iota_i)) = \big((X_i,\mu_i),\omega,(\iota_i)\big) \in \PL_k$, where we continue to assume that the label spaces $\Lambda_i$ are Hilbert spaces. As above, we write $X = \sqcup_i X_i$ and endow it with the measure $\mu = \sum_i \mu_i$. Given another element $P' \in \PL_k$ and a $k$-partitioned coupling $\pi \in \Pi_k((\mu_i),(\mu_i'))$, we write $\pi = \sum_i \pi_i$, and consider this as a measure on $X \times X'$ (where $X' = \sqcup_i X_i')$. Finally, we define
 \[
 \Lambda_X \coloneqq \bigoplus_{i=1}^k L^2(\mu_i,\Lambda_i).
 \]
 We denote an element of $\Lambda_X$ as $g$, where we can canonically write $g = (g_1,\ldots,g_k)$ with $g_i \in L^2(\mu_i,\Lambda_i)$. In this way, we consider the label function data as an element $\iota = (\iota_1,\ldots,\iota_k) \in \Lambda_X$. 

\begin{defn}[Synthetic tangent space]
  We define the \define{synthetic tangent space} of $[\PL_k]$ at a point $[L]$ to be  
  \begin{align*}
    T_{[L]} [\PL_k] \coloneqq \left(\bigcup_{((X_i, \mu_i), \omega, (\iota_i)) \in [L]} (L^2(\mu \otimes \mu) \oplus \Lambda_X)\right) / \sim.
  \end{align*}
In the above, the union is taken over all labelled $k$-partitioned measure networks $((X_i, \mu_i), \omega, (\iota_i))$ in the weak isomorphism equivalence class $[L]$. The equivalence relation is defined as follows. For two representatives 
\[
((X_i, \mu_i), \omega,(\iota_i)), \quad  ((X_i', \mu_i'), \omega', (\iota_i')) \in [L]
\]
and functions 
\[
(f, g) \in L^2(\mu \otimes \mu) \oplus \Lambda_X, \quad  (f', g') \in L^2(\mu'\otimes\mu') \oplus \Lambda_{X'},
\]
we write $(f, g) \sim (f', g')$ if and only if there exists a $k$-partitioned coupling $(\pi_i) \in \Pi_k(\mu, \mu')$ such that 
\[
f(x,y) = f'(x',y') \quad \mbox{for} \quad \pi\otimes \pi-a.e. \; (x,x',y,y') \in X \times X' \times X \times X'
\]
and, writing $g = (g_1,\ldots,g_k)$ and $g' = (g_1',\ldots,g_k')$, 
\[
g_i(x_i) = g_i'(x_i') \quad \mbox{for} \quad \pi_i-a.e. \;  (x_i,x_i') \in X \times X'.
\]
The equivalence class of $(f,g)$ is denoted $[f,g]$.
\end{defn}

The space $[\PL_k]$ has a natural notion of an exponential map, defined  as follows.  

\begin{defn}[Exponential map] \label{defn:exponential_map}
    For a labelled $k$-partitioned measure network $[L] \in [\PL_k]$, let $[f, g] \in T_{[L]}[\PL_k]$ be a tangent vector with $(f, g) \in L^2(\mu \otimes \mu) \oplus \Lambda_X$. We define the \define{exponential map} by
    \begin{equation*}
        \exp_{[L]} : T_{[L]} [\PL_k] \mapsto [\PL_k], \quad \exp_{[L]}([f, g]) \coloneqq [ ( (X_i, \mu_i), \omega + f ), \iota + g ]. 
    \end{equation*}
\end{defn}

We can now provide a geodesic characterization of the exponential map on $([\PL_k], d_{\PL_k})$, analogous to the one given for measure networks in \cite{chowdhury2020gromov}, at least for labelled partitioned measure networks which are ``inherently finite". That is, we say that an element $[L] \in [\PL_k]$ is \define{finite} if the equivalence class $[L]$ contains a representative $L' \in [L]$ such that all sets $X_i'$ are finite. In the following, take the following terminology convention: if we refer to $[L] \in [\PL_k]$ as finite, we implicitly assume without loss of generality that $L$ is finite. Observe that, even if $L$ is finite, the equivalence class $[L]$ contains elements which are not finite, hence the need for care in the terminology here.

\begin{prop}
  \label{prop:geodesic_exp}
  Let $[L] \in [\PL_k]$ be a finite labelled $k$-partitioned measure network. There exists $\epsilon_{[L]} > 0$ and $\eta_{[L]} > 0$ such that for any tangent vector represented by $(f, g) \in L^2(\mu \otimes \mu) \oplus \Lambda_X$ satisfying $|f(x, y)| < \epsilon_{[L]}$ and $|g(x)| < \eta_{[L]}$ for all $(x, y) \in X \times X$, the path defined by
  \[
    [\gamma_t] = [((X_i, \mu_i), \omega + tf), \iota + tg], \quad t \in [0, 1]
  \]
  is a geodesic from $[L]$ to $\exp_{[L]}([f,g])$.
\end{prop}

\begin{proof}
  Up to weak isomorphism, we may assume that $[\gamma_t]$ takes the form
  \[
    [\gamma_t] = \left[\left( (X_i \times X_i, \Delta_{\mu_i, \mu_i}), \omega_t\right), \iota_t\right],
  \]
  where
  \begin{align*}
    \omega_t : (x, x', y, y') &\mapsto (1-t) \omega(x, y) + t (\omega(x', y') + f(x', y')) \\ 
    \iota_t : (x, x') &\mapsto (1-t) \iota(x) + t (\iota(x') + g(x'))
  \end{align*}
  and $\Delta_{\mu_i, \mu_i}$ denotes the \define{diagonal coupling} of $(\mu_i, \mu_i)$; that is, $\Delta_{\mu_i,\mu_i} = (\mathrm{id}_{X_i} \times \mathrm{id}_{X_i})_\# \mu_i$.  
  By Proposition~\ref{prop:interpolation_geodesic_labelled}, verifying that $[\gamma_t]$ is a geodesic amounts to deducing the condition on $(f, g)$ for $(\Delta_{\mu_i, \mu_i})_{i = 1}^k$ to be the optimal couplings between $L$ and $\gamma(1)$.
  Let $(\pi_i) \in \Pi_{\Pc_k}(\mu, \mu)$ be any competitor coupling. The corresponding matching cost is then
  \begin{align*}
    &\frac{1}{2} \sum_{X^4} \pi(x, x') \pi(y, y') | \omega(x, y) - \omega(x', y') - f(x', y') |^2 + \frac{1}{2} \sum_{X^2} \pi(x, x') \| \iota(x) - \iota(x') - g(x') \|^2 \\
    &= \frac{1}{2} \sum_{X^4} \pi(x, x') \pi(y, y') f(x', y')^2 \\
    & \hspace{0.5in} + \sum_{X^4} \pi(x, x') \pi(y, y') \left[ \frac{1}{2} |\omega(x, y) - \omega(x', y')|^2 - f(x', y') (\omega(x, y) - \omega(x', y')) \right] \\
    & \hspace{0.5in} + \frac{1}{2} \sum_{X^2} \pi(x, x') \| g(x') \|^2 + \sum_{X^2} \pi(x, x') \left[ \frac{1}{2} \| \iota(x) - \iota(x') \|^2 - \langle g(x'), \iota(x) - \iota(x') \rangle \right].
  \end{align*}
  In the above, the inner product and norms are the induced structures on $\oplus_i \Lambda_i$. The first and third terms amount to the matching cost between $L$ and $\gamma(1)$ under the diagonal couplings $(\Delta_{\mu_i, \mu_i})_i$ and so it is sufficient to deduce conditions on $(f, g)$ so that the second and fourth terms are non-negative. 

  Consider first the sum
  \[
    \sum_{X^4} \pi(x, x') \pi(y, y') \left[ \frac{1}{2} | \omega(x, y) - \omega(x', y') |^2 - f(x', y') (\omega(x, y) - \omega(x', y')) \right]. 
  \]
  Clearly, if $|\omega(x, y) - \omega(x', y')|^2 = 0$ $\pi \otimes \pi$-a.e. then this term vanishes. Otherwise there exists at least one $(x, x', y, y') \in X^4$ such that $|\omega(x, y) - \omega(x', y')| > 0$, since we consider finite networks.
  Among such values, pick
  \[
    \epsilon_{[L]} = \frac{1}{2} \min \left\{ |\omega(x, y) - \omega(x', y')| : |\omega(x, y) - \omega(x', y')| > 0 \}\right\}
  \]
  Then, for $f$ satisfying $|f(x', y')| \leq \epsilon_{[L]}$ for all $(x', y') \in X \times X$, we have that 
  \begin{align*}
    |f(x', y') (\omega(x, y) - \omega(x', y'))| &= |f(x', y')||\omega(x, y) - \omega(x', y')| \\
    &\leq \epsilon_{[L]} |\omega(x, y) - \omega(x', y')|  \\
    &\leq \frac{1}{2} |\omega(x, y) - \omega(x', y')|^2,
  \end{align*}
  and so the sum is non-negative. 
  Next consider the labels
  \[
    \sum_{X^2} \pi(x, x') \left[ \frac{1}{2} \| \iota(x) - \iota(x') \|^2 - \langle g(x'), \iota(x) - \iota(x') \rangle \right]. 
  \]
  Applying the same reasoning as previously, we note that if $\| \iota(x) - \iota(x') \|^2 = 0$ $\pi$-a.e. then the sum vanishes. Otherwise, there exists at least one $(x, x') \in X^2$ for which $\| \iota(x) - \iota(x') \|^2 > 0$, and pick
  \[
    \eta_{[L]} = \frac{1}{2} \min \left\{ \| \iota(x) - \iota(x') \| : \| \iota(x) - \iota(x') \| > 0 \right\}. 
  \]
  Then for $g$ such that $\| g(x') \| \leq \eta_{[L]}$ for $x \in X$, we have that
  \begin{align*}
    \langle g(x'), \iota(x) - \iota(x') \rangle &\leq \| g(x') \| \| \iota(x) - \iota(x') \| \leq \leq \eta_{[L]} \| \iota(x) - \iota(x') \| \leq \frac{1}{2} \| \iota(x) - \iota(x') \|^2,
  \end{align*}
  and this sum is also non-negative.
\end{proof}

\begin{defn}[Logarithm map]
    Let $[L],[L'] \in [\PL_k]$ and let $(\pi_i) \in \Pi_k((\mu_i), (\mu'_i))$ be an optimal $k$-partitioned coupling of $L$ to $L'$. Consider the representative $\hat{L} = ((\hat{X}_i,\hat{\mu}_i),\hat{\omega},(\hat{\iota}_i)) \in [L]$  with
    \[
    \hat{X}_i = X_i \times X_i', \quad \hat{\mu}_i = \pi_i, \quad \hat{\omega}(x,x',y,y') = \omega(x,y), \quad \hat{\iota}_i(x,x') = \iota_i(x).
    \]
    Similarly, define the representative $\hat{L}' = ((\hat{X}_i,\hat{\mu}_i),\hat{\omega}',(\hat{\iota}'_i)) \in [L']$, where 
    \[
    \hat{\omega}'(x,x',y,y') = \omega'(x',y'), \quad \hat{\iota}'_i(x,x') = \iota'_i(x').
    \]
    We define
    \begin{equation}
        \log_{[L]}^{\pi}([L']) = [\omega' - \omega, \iota' - \iota] \in T_{[L]} [\PL_k].
    \end{equation}
    It follows by definition of $\exp_{[L]}$ that
    \[
    \exp_{[L]} ( \log_{[L]}^{\pi} ([L'])) = [L'].
    \]
\end{defn}

\begin{prop}{}
  \label{prop:exp_injective}
  Let $[L] \in [\PL_k]$ be a finite labelled $k$-partitioned measure network and let $\epsilon_{[L]}$ and $\eta_{[L]}$ be as in Proposition~\ref{prop:geodesic_exp}. The exponential map $\exp_{[L]}$ is injective on the set of tangent vectors admitting representations $(f, g) \in L^2(\mu \otimes \mu) \oplus \Lambda_X$ satisfying $|f(x, y)| < \epsilon_{[L]} / 2$ and $|g(x)| < \eta_{[L]} / 2$. 
\end{prop}

\begin{proof}
  Recall that $\exp_{[L]}([f, g]) = \exp_{[L]}([f', g'])$ if and only if $d_{\PL_k}(\exp_{[L]}([f, g]), \exp_{[L]}([f', g'])) = 0$.
  Let $\pi$ be an optimal coupling between $\exp_{[L]}([f, g])$ and $\exp_{[L]}([f', g'])$. Then the corresponding distortion is 
  \begin{align*}
    &\frac{1}{2} \sum_{X^4} \pi(x, x') \pi(y, y') | \omega(x, y) + f(x, y) - \omega(x', y') - f'(x', y') |^2 \\
    &\hspace{1.5in} + \frac{1}{2} \sum_{X^2} \pi(x, x') \| \iota(x) + g(x) - \iota(x') - g'(x') \|^2  \\
    &= \sum_{X^4} \pi(x, x') \pi(y, y') \left[ \frac{1}{2} |f(x, y) - f'(x', y')|^2 + \frac{1}{2} |\omega(x, y) - \omega(x', y')|^2 \right. \\
    &\hspace{1.5in} + (\omega(x, y) - \omega(x', y')) (f(x, y) - f'(x', y')) \bigg] \\
    &\qquad + \sum_{X^2} \pi(x, x') \left[ \frac{1}{2} \| g(x) - g'(x') \|^2 + \frac{1}{2} \| \iota(x) - \iota(x') \|^2 + \langle \iota(x) - \iota(x'), g(x) - g'(x') \rangle \right]
  \end{align*}
  Assuming we have that $|f(x, y)|, |f'(x,y)| < \epsilon_{[L]} / 2$ and $|g(x)|, |g'(x)| < \eta_{[L]} / 2$, then for all $(x, x', y, y') \in X^4$:
  \[
    |f(x, y) - f'(x', y')| < \epsilon_{[L]}, \quad \| g(x) - g'(x') \| < \eta_{[L]}. 
  \]
  Then
  \begin{align*}
    (\omega(x, y) - \omega(x', y')) (f(x, y) - f'(x', y')) &\leq |\omega(x, y) - \omega(x', y')| |f(x, y) - f'(x', y')| \\ 
                                                           &\leq \epsilon_{[L]} |\omega(x, y) - \omega(x', y')| \\
                                                           &\leq \frac{1}{2} | \omega(x, y) - \omega(x', y') |^2 ,
  \end{align*}
  and by the same reasoning we have that 
  \begin{align*}
    \langle \iota(x) - \iota(x'), g(x) - g(x') \rangle \leq \frac{1}{2} \| \iota(x) - \iota(x') \|^2 . 
  \end{align*}
  It follows that the second two terms in each of the sums are non-negative:
  $\frac{1}{2} |\omega(x, y) - \omega(x', y')|^2 + (\omega(x, y) - \omega(x', y')) (f(x, y) - f'(x', y')) \ge 0$ and
  $\frac{1}{2} \| \iota(x) - \iota(x') \|^2 + \langle \iota(x) - \iota(x'), g(x) - g'(x') \rangle \ge 0$.

  As a result, $d_{\PL_k}(\exp_{[L]}([(f, g)]), \exp_{[L]}([(f', g')])) = 0$ implies that $|f(x, y) - f'(x', y')| = 0$ $\pi\otimes\pi$-a.e. and $\| g(x) - g'(x') \|^2 = 0$ $\pi$-a.e. Therefore we conclude that $[f, g] = [f', g']$. 
\end{proof}

\subsubsection{The unlabelled case}
\label{sec:unlabelled-case}

Recall from Section~\ref{sec:consequences_and_comparisons} that $[\Pc_k]$ can be considered as a subspace of $[\PL_k]$, where the attribute spaces $\Lambda_i$ are 0-dimensional Hilbert spaces. Under this identification, the concepts and results from Section~\ref{sec:labelled-case} can be specialized to $[\Pc_k]$. It is more convenient to express the specialized concepts directly in the notation of $[\Pc_k]$, rather than in the notation coming from the embedding $[\Pc_k] \hookrightarrow [\PL_k]$. For the sake of convenience, we summarize these expressions in the language of partitioned networks below.

\begin{defn}[Tangent space for partitioned networks]
  We define the \define{synthetic tangent space} of $[\Pc_k]$ at a point $[P]$ to be  
  \begin{align*}
    T_{[P]} [\Pc_k] \coloneqq \left(\bigcup_{((X_i, \mu_i), \omega) \in [P]} L^2(\mu \otimes \mu)\right) / \sim.
  \end{align*}
In the above, the union is taken over all $k$-partitioned measure networks $((X_i, \mu_i), \omega)$ in the weak isomorphism equivalence class $[P]$. The equivalence relation is defined as follows. For two representatives 
\[
\left((X_i, \mu_i), \omega\right), \quad  \left((X_i', \mu_i'), \omega'\right) \in [P]
\]
and functions 
\[
f \in L^2(\mu \otimes \mu), \quad  f' \in L^2(\mu'\otimes\mu'),
\]
we write $f \sim f'$ if and only if there exists a $k$-partitioned coupling $(\pi_i) \in \Pi_k(\mu, \mu')$ such that 
\[
f(x,y) = f'(x',y') \quad \mbox{for} \quad \pi\otimes \pi-a.e. \; (x,x',y,y') \in X \times X' \times X \times X'.
\]
The equivalence class of $f$ is denoted $[f]$.
\end{defn}

\begin{defn}[Exponential map for partitioned networks]
    For a $k$-partitioned measure network $[P] \in [\Pc_k]$, let $[f] \in T_{[P]}[\Pc_k]$ be a tangent vector with $f \in L^2(\mu \otimes \mu)$. We define the \define{exponential map} by
    \begin{equation*}
        \exp_{[P]} : T_{[P]} [\Pc_k] \mapsto [\Pc_k], \quad \exp_{[P]}([f]) \coloneqq [ (X_i, \mu_i), \omega + f ]. 
    \end{equation*}
\end{defn}

\begin{defn}[Logarithm map for partitioned networks]
    Let $[P],[P'] \in [\Pc_k]$ and let $(\pi_i) \in \Pi_k((\mu_i), (\mu'_i))$ be an optimal $k$-partitioned coupling of $P$ to $P'$. Consider the representative $\hat{P} = ((\hat{X}_i,\hat{\mu}_i),\hat{\omega}) \in [P]$  with
    \[
    \hat{X}_i = X_i \times X_i', \quad \hat{\mu}_i = \pi_i, \quad \hat{\omega}(x,x',y,y') = \omega(x,y).
    \]
    Similarly, define the representative $\hat{L}' = ((\hat{X}_i,\hat{\mu}_i),\hat{\omega}') \in [L']$, where 
    \[
    \hat{\omega}'(x,x',y,y') = \omega'(x',y').
    \]
    We define
    \begin{equation}
        \log_{[P]}^{\pi}([P']) = [\omega' - \omega] \in T_{[P]} [\Pc_k].
    \end{equation}
    It follows by definition of $\exp_{[P]}$ that
    \[
    \exp_{[P]} ( \log_{[P]}^{\pi} ([P'])) = [P'].
    \]
\end{defn}

\subsection{Gradients}

Tasks in geometric statistics such as Fr\'echet means are often formulated in terms of minimization of functionals  over a manifold~\cite{pennec2006intrinsic}. To make sense of gradient flows, we need a notion of gradients. For simplicity, we will discuss the case of $k$-partitioned measure networks $\Pc_k$ (i.e.,~without labels). However, we remark that analogous results can be obtained for labelled graphs where labels are valued in Hilbert spaces.

\begin{defn}[Gradients of functionals]
  Let $F : [\Pc_k] \to \mathbb{R}$ be a functional on the space of $k$-partitioned measure networks. For a network $[P] \in [\Pc_k]$ and a tangent vector $[f] \in T_{[P]} [\Pc_k]$, we define the \define{directional derivative} of $F$, if it exists, to be
\[
  D_{[f]} F([P]) \coloneqq \lim_{t \downarrow 0} \frac{F(\exp_{[P]}([tf])) - F([P])}{t}.
\]
 A functional $F$ is said to be \define{strongly differentiable} (following \cite[Definition 6.23]{sturm2023space}) at a point $[P] \in [\Pc_k]$ if all of its directional derivatives exist, and if there exists a tangent vector $[g] \in T_{[P]} [\Pc_k]$ such that for any $[f] \in T_{[P]} [\Pc_k]$ and for every $(\pi_i) \in \Pi_k(\mu, \mu')$ such that $\| \omega - \omega' \|_{L^2((X \times X')^2, \pi \otimes \pi)} = 0$, it holds that 
\[ 
  D_{[f]} F([P]) = \langle f, g \rangle_{L^2(\pi \otimes \pi)}. 
\]
Here, $((X_i, \mu_i), \omega)$ and $((X_i', \mu_i'), \omega')$ are two representatives of $[P]$, and $f \in L^2(X^2, \mu \otimes \mu)$ and $g \in L^2(X'^2, \mu' \otimes \mu')$ are representatives of $[f]$ and $[g]$ respectively.
 We then write $[\nabla F(P)] \coloneqq [g]$ and refer to $[\nabla F(P)]$ as the \define{gradient} of $F$ at $[P]$. 
\end{defn}

The following proposition characterizes some basic properties of the gradient. It uses the concept of the \define{norm} of a tangent vector. For $[f] \in T_{[P]}[\Pc_k]$, with $P = ((X_i,\mu_i),\omega)$ and $f \in L^2(\mu \otimes \mu)$, define
\[
\| [f] \|_{T_{[P]}[\Pc_k]} \coloneqq \|f\|_{L^2(\mu \otimes \mu)}.
\]
The fact that this is well-defined follows easily from the nature of the equivalence relation used to construct the tangent space. 

\begin{prop}
  If $F : [\Pc_k] \to \mathbb{R}$ is strongly differentiable at $[P] \in [\Pc_k]$, then the gradient $[\nabla F(P)] \in T_{[P]} [\Pc_k]$ is unique and
  \[
    \| [\nabla F (P)] \|_{T_{[P]} [\Pc_k]} = \sup\: \{ D_{[f]} F([P]) \: : \: [f] \in T_{[P]} [\Pc_k], \| [f] \|_{T_{[P]}[\Pc_k]} = 1 \}. 
  \]
\end{prop}

\begin{proof}
  This follows directly from \cite[Lemma 6.24]{sturm2023space} for the case of partitioned measure networks by requiring couplings to respect partitions where necessary. 
\end{proof}

\subsection{Calculating gradients}\label{sec:calculating_gradients}

Motivated by some practical applications, in this section, we compute expressions for gradients of two functionals defined over $[\Pc_k]$. Namely, we consider the Fr\'echet functional and its generalization to the problem of geodesic dictionary learning. We will put these expressions to use in Section \ref{sec:frechet_means}, where we conduct some numerical computations with partitioned measure networks. Since in this section we focus on practical utility for numerical applications, some computations are done formally. A rigorous theoretical treatment of gradient flows has been addressed in the context of measure networks by Sturm \cite{sturm2023space}. In particular, a rigorous analysis of the dictionary learning problem may be a useful area for future study. More generally, the barycenter computation problem remains an active area of research even in the case of measures in $\mathbb{R}^d$ (e.g. \cite{chizat2023computational, altschuler2022wasserstein}).

Our gradient computations will make repeated use of the following result.

\begin{prop}\label{prop:aligned}
    Let $[P],[P'] \in [\Pc_k]$ and let $(\pi_i)$ be an optimal $k$-partitioned coupling. There exist representatives $\overline{P} \in [P]$ and $\overline{P}' \in [P']$ whose underlying sets and measures are the same and the diagonal couplings give an optimal $k$-partitioned coupling. If $[P]$ and $[P']$ are finite, then the representatives can also be taken to be finite.

    In particular, if $(\pi_i)$ is an optimal $k$-partitioned coupling of $P$ and $P'$, then such representatives are given by
    \[
    \overline{P} = \big((X_i \times X'_i, \pi_i), ((x, x'), (y, y')) \mapsto \omega(x, y)\big) \quad \mbox{and} \quad \overline{P}' =   \big((X_i \times X'_i, \pi_i), ((x, x'), (y, y')) \mapsto \omega'(x', y')\big).
    \]
\end{prop}

The proof is a straightforward verification that the proposed $\overline{P}$ and $\overline{P}'$ satisfy the conditions; see \cite[Lemma 12]{chowdhury2023hypergraph} for details in the case of hypernetworks. When $k$-partitioned measure networks satisfy the conditions in the first paragraph, we say that the networks are \define{aligned}; the conclusion of the proposition is that, when considered up to weak isomorphism, we can assume without loss of generality that any pair of partitioned measure networks is aligned. Moreover, given a finite collection of partitioned measure networks $\{ [P_i], 1 \leq i \leq N \}$, repeated application of the proposition shows that we can assume without loss of generality that each $P_i$, $i \geq 2$, is aligned to $P_1$.

\subsubsection{Fr\'echet functional}
  Define the \define{Fr\'echet functional} for a finite collection  of partitioned measure networks $[P_1], \ldots, [P_N] \in [\Pc_k]$ to be the maps $F : [\Pc_k] \to \mathbb{R}$ given by
  \begin{equation}\label{eq:frechet_functional}
    F([R]) = \frac{1}{N} \sum_{i = 1}^N d_{\Pc_k}([R], [P_i])^2. 
  \end{equation}
  Based on the discussion following Proposition \ref{prop:aligned}, we can assume without loss of generality that $R$ and the $P_i$ are aligned, so that they are of the form $((X_i, \mu_i), \omega_R)$ and $((X_i, \mu_i), \omega_{P_i})$, respectively, and an optimal $k$-partitioned coupling of $R$ to each $P_i$ is given by diagonal couplings $(\Delta_{\mu_i,\mu_i})$. We claim that, under the assumption that all networks $R,P_1,\ldots,P_N$ are finite, the gradient of $F$ is represented by 
  $\nabla F(R) \in L^2(X^2, \mu^{\otimes 2})$, with
  \begin{equation}\label{eq:gradient_frechet_functional}
    \nabla F(R) = \omega_R - \frac{1}{N} \sum_{i = 1}^N \omega_{P_i}.
  \end{equation}

  The Fr\'echet functional was studied in \cite{chowdhury2020gromov} in the setting of measure networks. Although the case of partitioned measure networks can be treated by the same approach, we include a derivation of \eqref{eq:gradient_frechet_functional} for completeness. We first consider the setting where $N = 1$, in which case the Fr\'echet functional simplifies to
  \[
    F([R]) = d_{\Pc_k}([P], [R])^2. 
  \]
  We fix two representatives $R = ((X_i, \mu_i), \omega_R)$ and $P = ((X_i, \mu_i), \omega_P)$ such that an optimal $k$-partitioned coupling between $R$ and $P$ is given by diagonal couplings $(\Delta_i) = (\Delta_{\mu_i,\mu_i})$. Following our usual convention, we let $\Delta = \sum_i \Delta_i$, considered as a measure on $\sqcup_i X_i$. Let $f \in L^2(X^2, \mu\times \mu)$ be a representative of a tangent vector in $T_{[R]} [\Pc_k]$. Then by definition of the directional derivative,
  \begin{align*}
    D_{[f]} F([R]) &= \lim_{t \downarrow 0} \frac{F(\exp_{[R]}(t[f])) - F([R])}{t}. 
  \end{align*}
  By the proof of Proposition~\ref{prop:geodesic_exp} (specialized to the case of partitioned networks), we may assume without loss of generality that $(\Delta_i)$ is an optimal $k$-partitioned coupling between $\exp_R(tf)$ and $P$, for $t$ sufficiently small. Then, for small enough $t$,
  \begin{align*}
    &\frac{1}{t} (F(\exp_R (tf)) - F(R) ) \\
    &\qquad \qquad = \frac{1}{t} \left( d_{\Pc_k}(\exp_R(tf), P)^2 - d_{\Pc_k}(R, P)^2 \right) \\ 
                                         &\qquad \qquad = \sum_{i=1}^k \frac{1}{t} \left( \frac{1}{2} \| \omega_R + tf - \omega_P \|_{L^2(\Delta_i^{\otimes 2})}^2 - \frac{1}{2} \| \omega_R - \omega_P \|_{L^2(\Delta_i^{\otimes 2})}^2  \right) \\
                                         &\qquad \qquad = \sum_{i=1}^k \frac{t}{2} \| f \|^2_{L^2(\mu_i^{\otimes 2})} + \langle f, \omega_R - \omega_P \rangle_{L^2(\Delta_i^{\otimes 2})} + \frac{1}{2t} \left( \| \omega_R - \omega_P \|_{L^2(\Delta_i^{\otimes 2})}^2 - \| \omega_R - \omega_P \|_{L^2(\Delta_i^{\otimes 2})}^2 \right).
  \end{align*}
  Taking $t \rightarrow 0$ yields
  \[
  D_{[f]} F([R]) = \sum_{i=1}^k \langle f, \omega_R - \omega_P \rangle_{L^2(\Delta_i^{\otimes 2})} = \langle f, \omega_R - \omega_P \rangle_{L^2(\Delta^{\otimes 2})},
  \]
  hence a representative of $[\nabla F(R)]$ is
  \begin{align*}
    \nabla F(R) = \omega_R - \omega_P \in L^2(\mu \otimes \mu). 
  \end{align*}
  The formula \eqref{eq:gradient_frechet_functional} follows by linearity.

\subsubsection{Geodesic dictionary learning}\label{sec:geodesic_dictionary_learning}
  Let $[P_1], \ldots, [P_N] \in [\Pc_k]$ be a collection of finite $k$-partitioned measure networks. We consider a generalization of the Fr\'echet mean, which seeks to find a \define{dictionary} of $m$ \define{atoms} (i.e. representatives or archetypes) $[D_1], \ldots, [D_m] \in [\Pc_k]$ (each of which we assume to be finite) and a collection of $N$ vectors in the $(m-1)$-dimensional \define{probability simplex} $\alpha_1, \ldots, \alpha_N \in \Delta^{m}$ (that is, each $\alpha_j \in \R^m$ has nonnegative entries which sum to one), that provide a useful set of reference points for summarizing the original dataset. In what follows, we give a brief heuristic derivation of gradient expressions which can be used for approximating such a dictionary (see Section~\ref{sec:dictionary-learning} for applications involving such a derivation). A rigorous treatment of this difficult bi-level optimization problem is out of the scope for our present paper. We point out that, compared to the better-understood analogous problem of learning Wasserstein barycenters or dictionaries in Euclidean space \cite{schmitz2018wasserstein, chizat2023computational, altschuler2022wasserstein}, in our setting, issues such as uniqueness of barycenters have not been rigorously addressed.

  We first informally define a \define{barycenter operator} to be any assignment taking a proposed dictionary $\{D_j\}_{j=1}^m$ together with an vector $\alpha \in \Delta^m$, whose entries are denoted $\alpha(1),\ldots,\alpha(m)$, to
  \begin{equation}\label{eq:barycenter_operator}
    B\left( \{ D_j \}_{j = 1}^m, \alpha\right) \in \argmin_{R \in \Pc_k} \sum_{j = 1}^m \alpha(j) d_{\Pc_k}(D_j, R)^2.
  \end{equation}
  That is, there is not necessarily a unique minimizer, so a barycenter operator must involve a choice. In practice, barycenters are approximated by some algorithm, so the necessity of making a choice models a realistic situation (although the approximators are likely to only return a local minimizer). 
  
  Next, the loss functional for the \define{geodesic dictionary learning problem} is the bi-level optimization problem:
  \begin{equation}
    F\left( \{ D_j \}_{j = 1}^m, \{ \alpha_i \}_{j = 1}^N \right) = \frac{1}{N} \sum_{j = 1}^N d_{\Pc_k} \left( B(\{D_j\}, \alpha_j ), P_j \right)^2 . 
  \end{equation}
  For practical purposes, we replace this with an easier problem by taking the following heuristic approximation of the barycenter operator. Given a dictionary $\{D_j\}$, a basepoint $P$ and a weight vector $\alpha$, we define the \define{local barycenter operator} to return $B_P(\{D_j\},\alpha) = ((X_i,\mu_i),\omega_B)$, where it is assumed that all atoms have been aligned to $P = ((X_i,\mu_i),\omega)$, so that $D_j = ((X_i,\mu_i),\omega_{D_j})$, and $\omega_B$ is defined by
  \[
  \omega_B = \sum_{j=1}^m \alpha(j) \omega_{D_j},
  \]
  We then consider
  \[
    F\left( \{ D_j \}_{j = 1}^m, \{ \alpha_i \}_{j = 1}^N \right) = \frac{1}{N} \sum_{j = 1}^N d_{\Pc_k} \left( B_{P_j}(\{D_j\}, \alpha_j ), P_j \right)^2 . 
  \]

  To approximate this via gradient descent, we hold all arguments constant besides one and run a gradient step on the induced functional, and iterate this process through arguments. We claim that, for $\{ \alpha_i \}_{i = 1}^N$ held constant, the gradient of $F$ in each of the $D_i$ (that is, holding other atoms constant) is given by 
  \begin{equation}\label{eqn:dictionary_gradient_1}
    \nabla_{D_i} F(D_i) = \frac{1}{N} \sum_{j = 1}^m \omega_{D_j} \left( \sum_{\ell = 1}^N \alpha_{\ell}(i) \alpha_{\ell}(j) \right) - \frac{1}{N} \sum_{\ell = 1}^N \alpha_{\ell}(i) \omega_{P_\ell}.  
  \end{equation}
  where we have assume that all spaces are aligned to a common $k$-partitioned measures space, as in Proposition \ref{prop:aligned}.
  Similarly, the Fr\'echet gradient in each of the $\alpha_\ell$ is given by
  \begin{equation}\label{eqn:dictionary_gradient_2}
    \nabla_{\alpha_\ell} F(\alpha_\ell) = \frac{1}{N} \left[\langle \omega_{D_i}, \omega_{B} - \omega_{P_\ell} \rangle \right]_{i = 1}^m.
  \end{equation}
  Heuristic derivations of these expressions are provided below.

  By linearity (as in the case of Fr\'{e}chet means), let us take $N = 1$. Then the functional simplifies to
  \[
    F\left( \{ D_j \}_{j = 1}^m, \alpha \right) = d_{\Pc_k} (B_P(D, \alpha), P)^2. 
  \]
  We wish to consider the functional induced by holding all but one of the atoms fixed; without loss of generality, suppose that only $D_1$ varies and $D_2,\ldots,D_m$ and $\alpha$ are fixed. Then we consider the functional
  \[
  \tilde{F}(D_1) = F(\{D_i\}_{i=1}^m, \alpha). 
  \]
  Let $f \in L^2(X^2, \mu^{\otimes 2})$ be a tangent vector at $D_1$ and let $\omega_{B(t)}$ denote the network kernel of $\mathrm{exp}_{D_1}(tf)$---explicitly, $\omega_{B(t)} = \omega_B + \alpha_1 t f$. By a similar computation to the one used in the derivation of the gradient to the Fr\'{e}chet functional, we have, for sufficiently small $t$ (so that the ideas of Proposition~\ref{prop:geodesic_exp} apply),
  \begin{align*}
    &\frac{1}{t} (\tilde{F}(\exp_{D_1} (tf)) - \tilde{F}(D_1) ) \\
    &\qquad = \frac{1}{t} \left( d_{\Pc_k}(B_P(\{\exp_{D_1}(tf),D_2,\ldots,D_m\},\alpha), P)^2 - d_{\Pc_k}(B_P(\{D_j\},\alpha), P)^2 \right) \\ 
                                         &\qquad = \sum_{i=1}^k \frac{1}{t} \left( \frac{1}{2} \| \omega_{B(t)} - \omega_P \|_{L^2(\Delta_i^{\otimes 2})}^2 - \frac{1}{2} \| \omega_{B} - \omega_P \|_{L^2(\Delta_i^{\otimes 2})}^2  \right) \\
                                         &\qquad = \sum_{i=1}^k \frac{1}{t} \left( \frac{1}{2} \| \omega_{B} + \alpha_1 t f - \omega_P \|_{L^2(\Delta_i^{\otimes 2})}^2 - \frac{1}{2} \| \omega_{B} - \omega_P \|_{L^2(\Delta_i^{\otimes 2})}^2  \right) \\
                                         &\qquad = \sum_{i=1}^k \frac{\alpha_1^2t}{2} \| f \|^2_{L^2(\mu_i^{\otimes 2})} + \alpha_1 \langle f, \omega_{B} - \omega_P \rangle_{L^2(\Delta_i^{\otimes 2})} + \frac{1}{2t} \left( \| \omega_{B} - \omega_P \|_{L^2(\Delta_i^{\otimes 2})}^2 - \| \omega_{B} - \omega_P \|_{L^2(\Delta_i^{\otimes 2})}^2 \right).
  \end{align*}
  The claimed formula \eqref{eqn:dictionary_gradient_1} then follows by a straightforward calculation.

  When deriving the formula \eqref{eqn:dictionary_gradient_2}, we observe that this amounts to computing the derivative of a function defined on $\R^m$. The derivation then follows by elementary methods.

\section{Applications and algorithms}
\label{sec:applications}

In this section, we discuss a large number of applications using our formulation of the partitioned measure networks. 
In Section~\ref{sec:algos_intro}, we first give an overview of numerical algorithms involved in these applications. 
In Section~\ref{sec:matching_and_comparison}, we discuss network matching and comparison using partitioned measure networks. We comment on the connection between measure network matching using Gromov-Wasserstein type distances and spectral network alignment such as EigenAlign (Section~\ref{sec:relation_to_spectral_alignment}). We support our analysis using experiments involving synthetically generated graphs and hypergraphs (Section~\ref{sec:match-random}), real-world metabolic networks (Section~\ref{sec:metabolic}), and multi-omics data (Section~\ref{sec:multi-omics}). 
In Section~\ref{sec:multiscale}, we further expand the applications to study multiscale network matching using partitioned measure networks, by studying networks that arise from 3D mesh objects (Section~\ref{sec:multiscale-point-cloud}) and protein-protein intersections (Section~\ref{sec:multiscale-protein}). 
In Section~\ref{sec:frechet_means}, we demonstrate via simple examples in computing geodesic interpolations between three hypergraphs and their barycenter. 
Finally in Section~\ref{sec:dictionary-learning}, we further extend the study of barycenter to the problem of dictionary learning using partitioned measure networks, that is, given an ensemble of partitioned measure networks, learn a basis such that each ensemble member could be described as a convex combination of the basis elements. 
We study nonlinear and linear dictionary learning in Section~\ref{sec:nonlinear-dictionary-learning} and Section~\ref{sec:linear-dictionary-learning} respectively, with examples that arise from hypergraph stochastic block model (Section~\ref{sec:stochastic-block-model}) and mutagenicity data (Section~\ref{sec:mutagenicity}).  
For implementation details, see Appendix~\ref{sec:algorithms_appendix}.   

\subsection{Numerical algorithms}
\label{sec:algos_intro}

Practical applications of our framework in machine learning and statistics hinges upon numerical solution of quadratic programs that arise from the matching problem introduced in Definition \ref{defn:partitioned_network_distance} and its extensions. While significant progress has been made developing and analyzing numerical approaches for the case of Gromov-Wasserstein matchings (for measure networks) \cite{peyre2016gromov, chowdhury2019gromov, thual2022aligning}, co-optimal transport (for measure hypernetworks) \cite{titouan2020co, tran2023unbalanced}, and augmented Gromov-Wasserstein 
 (for augmented measure networks) \cite{demetci2023revisiting}, our framework allows us to consider in generality (labelled) $k$-partitioned measure networks, from which each of these algorithms emerges as a special case. We provide a brief overview in what follows, and we defer technical details of specific algorithms to Appendix \ref{sec:algorithms_appendix} due to space considerations. 

As we are interested in numerical calculations related to generalized networks, we will assume all networks to be finite in this section and where appropriate use matrix notation to represent functions defined on finite spaces. That is, for a $k$-partitioned measure network $P = ((X_i, \mu_i), \omega)$ (see \eqref{defn:partitioned_measure_network}) we take each $X_i$ to be a finite set and denote its cardinality by $|X_i|$. Thus $\mu_i$ belonging to the probability simplex  $\Delta^{|X_i|}$ (see Section \ref{sec:geodesic_dictionary_learning}) is a discrete probability distribution which we often write as a column vector, and $\omega$ a matrix of dimensions $|\sqcup_i X_i |\times| \sqcup_i X_i |$. Furthermore, we write
\[
  \omega = \begin{bmatrix} \omega_{11} && \cdots && \omega_{1k} \\ \vdots && \ddots && \vdots \\ \omega_{k1} && \cdots && \omega_{kk} \end{bmatrix} = \left[ \omega_{ij} \right]_{1 \leq i, j \leq k}, 
\]
where $\omega_{ij}$ corresponds to the restriction of $\omega$ onto $X_i \times X_j$, i.e. a submatrix of dimensions $|X_i| \times |X_j|$. In what follows, we use angle brackets to denote the Frobenius inner product between vectors or matrices, i.e. $\langle A, B \rangle = \tr(A^\top B)$. 

We take $p = 2$ in the definition of the (labelled) partitioned network distance (see \eqref{defn:labelled_partitioned_distance} and Definitions \ref{defn:partitioned_network_distance}), as this allows for a efficient scheme for evaluating computationally the value of $\| \omega - \omega' \|_{L^2(\pi \otimes \xi)}^2$. For ease of notation, we will re-write the (squared) objective from \eqref{eqn:labelled_partitioned_distance}:
\begin{align}
  \min_{\pi_i \in \Pi(\mu_i, \mu_i'), 1 \leq i \leq k} \frac{1}{2} \sum_{i, j = 1}^k \| \omega - \omega' \|_{L^2(\pi_i \otimes \pi_j)}^2 + \sum_{i = 1}^k \| d_{\Lambda_i} \circ (\iota_i, \iota_i') \|_{L^2(\pi_i)}^2.
  \label{eq:discrete_labelled_partitioned_distance_a}
\end{align}
Up to scaling $(\omega, \omega')$ and $d_{\Lambda_i}$ by constant factors, this is equivalent to the problem as written in \eqref{eqn:labelled_partitioned_distance} with $p = 2$. 
In particular, we note that the factor of $1/2$ is associated to terms quadratic in the $\pi_i$, which will simplify expressions later.
Adopting the notation of \cite{peyre2016gromov}, we define $L(\omega, \omega')$ to be the 4-way \define{distortion tensor}
\begin{equation}
  L(\omega, \omega')_{ijkl} = \frac{1}{2} | \omega_{ik} - \omega'_{jl} |^2.
  \label{eq:distortion_tensor}
\end{equation}
Introduce also for the labelled setting
\begin{equation}
  C_i(x, x') = \frac{1}{2} d_{\Lambda_i}(\iota_i(x), \iota_i'(x') )^2, 1 \leq i \leq k
\end{equation}
as cost matrices for matching labels in each label metric space $\Lambda_i$.
Although our theoretical setup in Definition \ref{defn:labelled_partitioned_measure_network} assumes the existence of labelling functions $\iota_i$ and label metric spaces $\Lambda_i$, in practice our computations depend only on the matrices $C_i$ and so the labellings are not made explicit. For instance, $C_i$ may be constructed from kernels and thus understood to correspond to squared pairwise distances in a reproducing kernel Hilbert space.

Using the quantities we have now introduced, the problem \eqref{eq:discrete_labelled_partitioned_distance_a} can be written as
\begin{align}
  \min_{\pi_i \in \Pi(\mu_i, \mu_i'), 1 \leq i \leq k} \frac{1}{2} \sum_{i, j = 1}^k \langle L(\omega_{ij}, \omega_{ij}'), \pi_i \otimes \pi_j \rangle + \sum_{i = 1}^k \langle C_i, \pi_i \rangle.
  \label{eq:discrete_labelled_partitioned_distance}
\end{align}
By setting appropriate terms to zero (following the lines of Definition \ref{defn:generalized_network_embeddings}), we can recover the optimal transport matching problems on generalized measure networks such as measure hypernetworks and measure networks (both labelled and unlabelled), as well as the standard optimal transport.
Additionally, we can consider regularized variants of this problem which may yield favourable results in practice \cite{cuturi2013sinkhorn, peyre2016gromov, blondel2018smooth}, both in terms of numerical solution schemes as well as properties of the solution:
\begin{align}
  \min_{\pi_i \in \Pi(\mu_i, \mu_i'), 1 \leq i \leq k} \frac{1}{2} \sum_{i, j = 1}^k \langle L(\omega_{ij}, \omega_{ij}'), \pi_i \otimes \pi_j \rangle + \sum_{i = 1}^k \langle C_i, \pi_i \rangle + \sum_{i = 1}^k \varepsilon_i \Omega_i(\pi_i).
  \label{eq:entropic_discrete_labelled_partitioned_distance}
\end{align}
A common choice of $\Omega$ is the relative entropy $\Omega(\pi) = \KL(\pi | \mu \otimes \mu')$, which is consistent with the existing formulations of regularized co-optimal transport \cite{titouan2020co} and Gromov-Wasserstein transport \cite{peyre2016gromov}.

Solving \eqref{eq:discrete_labelled_partitioned_distance} (or \eqref{eq:entropic_discrete_labelled_partitioned_distance}) amounts to finding minimizers of a (regularized) constrained, non-convex quadratic program in the couplings $(\pi_i)_{i = 1}^k$. Na\"ive solutions of these problems using general purpose solvers is not scalable \cite{han2023covariance}. We develop several iterative algorithms to this end: since the problem \eqref{eq:discrete_labelled_partitioned_distance} is non-convex, different choices of algorithm may converge to different minima. 
In summary, algorithmic approaches to solving \eqref{eq:discrete_labelled_partitioned_distance} or \eqref{eq:entropic_discrete_labelled_partitioned_distance} can be divided into (a) approaches relying on iterative solution of the standard linear optimal transport or Gromov-Wasserstein transport as algorithmic primitives, and (b) approaches based on gradient descent. We refer the interested reader to Appendix \ref{sec:algorithms_appendix} for details. As our examples illustrate, the algorithm of choice depends heavily on the application at hand. 

\subsection{Network matching and comparison}
\label{sec:matching_and_comparison}

In this section, we illustrate the utility of our theoretical and algorithmic framework via network matching and comparison.  We first discuss a connection between Gromov-Wasserstein measure network alignment and a spectral network alignment method, as well as their respective generalizations to hypergraphs. Together with numerical results, we show that the optimal transport framework has a better behaviour, in terms of both accuracy and scalability.

We next consider an application to metabolic network alignment. We model this problem as one of labelled hypergraph matching (i.e. $k = 2$ for our partitioned setup), and solve an \emph{unbalanced} transport problem due to the lack of a one-to-one matching between network elements. We find that, while incorporating label information alongside the hypergraph structure is essential to obtaining meaningful alignments, the hypergraph relational structure provides information that is crucial for refining the alignment. That is, incorporating the hypergraph structure improves significantly upon using labels alone. 

Last, we turn to a problem of simultaneous sampling and feature alignment in multi-omics data, wherein networks are derived from general data matrices (see Example~\ref{ex:generalized_network_examples}). This is a problem for which co-optimal transport and augmented Gromov-Wasserstein have been previously developed \cite{titouan2020co, tran2023unbalanced, demetci2023revisiting}, viewing data matrices as hypergraphs where nodes are samples and hyperedges are features. These algorithms fall under our partitioned framework with $k = 2$. We show that partitioned networks are a flexible and more general tool for modelling multi-omics  data, and results in improved alignment accuracy.

\subsubsection{Relation to spectral network alignment} 
\label{sec:relation_to_spectral_alignment}

We first comment on the connection of Gromov-Wasserstein measure network matching to a (perhaps widely known) family of \emph{spectral alignment} approaches. As introduced in Definition \ref{defn:generalized_network_distances}, for $p = 2$ and measure networks $(X, \mu, \omega), (X', \mu', \omega')$, the Gromov-Wasserstein (measure network) alignment problem is to solve
\begin{equation}
  \min_{\pi \in \Pi(\mu, \mu')} \frac{1}{2} \langle L(\omega, \omega'), \pi \otimes \pi \rangle, \quad L(\omega, \omega')_{ijkl} = \frac{1}{2} | \omega_{ik} - \omega'_{jl} |^2,
  \label{eq:gromov_alignment_a}
\end{equation}
which corresponds to partitioned measure network matching of Definition \ref{defn:partitioned_network_distance} with $k = 1$. This approach was studied in depth by \cite{xu2019scalable} for network alignment. Spectral alignment methods are a family of approaches that have gained attention for graph alignments \cite{feizi2016spectral, feizi2019spectral, nassar2018low, onaran2017projected} and also for hypergraphs \cite{shen2018genome}. Briefly, for two input graphs $G = (V, E), G' = (V', E')$, spectral network alignment seeks a node matching between $V$ and $V'$ that optimally preserves graph structure in a way similar to the Gromov-Wasserstein problem. This leads to a quadratic assignment problem (QAP), which upon being relaxed amounts to solving for the Perron-Frobenius eigenvector of a square matrix with dimensions $|V \times V'| \times |V \times V'|$ with all positive entries. We now make this problem description  concrete. We abuse notation and also write $G_{ij}, G_{ij}'$ to denote the (binary) adjacency matrices of the graphs $G, G'$ respectively. 

Feizi et al. \cite{feizi2016spectral} defined a matching score, for $(i, j), (k, l) \in (V \times V')^2$:
\begin{equation}
  A_{ijkl} = \begin{cases}
    s_1 &(i, k) \in E \wedge (j, l) \in E'; \\ 
    s_2 &(i, k) \notin E \wedge (j, l) \notin E'; \\
    s_3 &\text{otherwise}. 
  \end{cases}
\end{equation}
Here, $s_1, s_2, s_3 > 0$.  
The first case corresponds to matching edges to edges (referred to as ``matches'' in \cite{feizi2016spectral}) with score $s_1$, the second case corresponds to matching non-edges with non-edges (``neutrals'') with score $s_2$, and the final case corresponds to matching non-edges to edges, or vice versa (``mismatches'') with score $s_3$. 

It is immediately clear that $A_{ijkl}$ plays the same role (but with opposite sign, since in \cite{feizi2016spectral} the aim is to \emph{maximize} the matching score), as the tensor $L(G_{ik}, G'_{jl})$ in the Gromov-Wasserstein network alignment setting. While $A_{ijkl}$ is a \emph{matching score} (larger is better), $L_{ijkl}$ is a \emph{distortion} (smaller is better). The authors further derived an identity for $A$ (Equation 3.3 of \cite{feizi2016spectral}):
\[
  A = (s_1 + s_2 - 2s_3) (G \otimes G') + (s_3 - s_2)(G \otimes \ones + \ones \otimes G' ) + s_2 (\ones \otimes \ones), 
\]
which is also a convenient formula for the Gromov-Wasserstein setup, for computing $\langle A, x \otimes x \rangle$.

The graph alignment problem is then formulated as a QAP in terms of an unknown alignment matrix $y \in \mathbb{R}^{|V| \times |V'|}$ (which by an abuse of notation, we will also write as a vector of length $|V \times V'|$):
\begin{align}
  \begin{split}
    &\max_{y} y^\top A y\\
    \text{ s.t. } \qquad &y \ones \leq 1, y^\top \ones \leq 1, y \in \{ 0, 1 \}.
  \end{split}\label{eq:feizi_alignment_qap}
\end{align}
Since direct solution of this problem is intractable, Feizi et al. proposed an algorithm called \emph{EigenAlign} which first solves a relaxation of \eqref{eq:feizi_alignment_qap} where the integer and row/column-sum constraints are replaced with non-negativity and unit-ball constraints:
\begin{align}
  \begin{split}
    &\max_{y} y^\top A y \\
    \text{ s.t. } \qquad &y \ge 0, \| y \|_2 \leq 1.
  \end{split}\label{eq:feizi_alignment_eigenproblem}
\end{align}
The solution to this problem is shown to be $v$, the Perron-Frobenius eigenvector of the positive alignment score matrix $A$. In the second step of the algorithm, the solution $v$ to the relaxed problem \eqref{eq:feizi_alignment_eigenproblem} is projected back onto the constraint set by solving a linear assignment problem:
\begin{align}
\begin{split}
    &\max_{y} v^\top y \\
    \text{ s. t. } \qquad & y \ge 0 : y \ones \leq 1, y^\top \ones \leq 1, y \in \{ 0, 1 \}. 
    \end{split}\label{eq:feizi_alignment_linear_assignment}
\end{align}
This can be understood as solving for $y$ that maximizes its similarity to the relaxed solution $v$ that also satisfies the bijectivity constraints. 
The objective function for the Gromov-Wasserstein alignment problem has the exact same form as \eqref{eq:feizi_alignment_qap}, since $\langle L(G, G'), \pi \otimes \pi \rangle = \vec(\pi)^\top \mat(L(G, G')) \vec(\pi)$. 

  The correspondence between the Gromov-Wasserstein \eqref{eq:gromov_alignment_a} and spectral alignment problems (\ref{eq:feizi_alignment_qap}, \ref{eq:feizi_alignment_eigenproblem}, \ref{eq:feizi_alignment_linear_assignment}) has some subtlety. Under the simplex constraint $\pi \in \Pi(\mu, \mu')$ the problem \eqref{eq:gromov_alignment_a} is invariant to constant shifts in the distortion tensor $L$, since for any $c \in \mathbb{R}$, 
\[
  \langle L(\omega, \omega') + c, \pi \otimes \pi \rangle = \langle L(\omega, \omega'), \pi \otimes \pi \rangle + c. 
\]
For binary incidence matrices $\omega_{ij}, \omega'_{ij} \in \{0, 1\}$, we can introduce a \emph{shifted} version of the distortion tensor, where $\eta > 0$ is a small constant:
\[
  \overline{L}(\omega, \omega')_{ijkl} = L(\omega, \omega)_{ijkl} - \frac{1}{2} - \eta.
\]
Then $\overline{L}$ is strictly negative. We can therefore re-write the problem \eqref{eq:gromov_alignment_a} equivalently as
\begin{align*}
  \min_{\pi \in \Pi(\mu, \mu')} \langle \overline{L}, \pi \otimes \pi \rangle = \vec(\pi)^\top \mat(\overline{L}) \vec(\pi) \Longleftrightarrow \max_{\pi \in \Pi(\mu, \mu')} \vec(\pi)^\top \mat(-\overline{L}) \vec(\pi).
\end{align*}
Therefore, we may choose $A = \mat(-\overline{L})$ in (\ref{eq:feizi_alignment_qap}, \ref{eq:feizi_alignment_eigenproblem}, \ref{eq:feizi_alignment_linear_assignment}) since it is a positive matrix. This corresponds to $s_1 = s_2 = 1/2 + \eta$ and $s_3 = \eta$. On the other hand, crucially the objective \eqref{eq:feizi_alignment_eigenproblem} is \emph{not} invariant under additive shifts to the matrix $A$, since
\[
  y^\top (A + c \ones \ones^\top) y = y^\top A y + c | \ones^\top y |^2
\]
and $|\ones^\top y|$ is not constant on the 2-norm ball. Therefore, while additive shifts of the distortion tensor leave the Gromov-Wasserstein problem \eqref{eq:gromov_alignment_a} unchanged, different choices of the shift lead to different relaxed spectral problems \eqref{eq:feizi_alignment_eigenproblem}. We remark that, since the Perron-Frobenius theorem restricts EigenAlign to positive alignment matrices, one cannot straightforwardly take $A = -\mat(L)$. 

The main remaining difference between Gromov-Wasserstein and EigenAlign lies in the constraints: inequality constraints on the row and column sums of $\pi$ are replaced instead with equality constraints. When $|V| = |V'|$ and node weights are chosen to be uniform, this amounts to the set of bi-stochastic matrices.
In a sense, the relaxed problem solved by Gromov-Wasserstein departs less from \eqref{eq:feizi_alignment_qap} than EigenAlign. Noting that $\Pi(\mu, \mu') \subseteq \operatorname{Prob}(X \times X')$ and that $\{ x \in \mathbb{R}^k : |x| = 1, x \geq 0 \} \subseteq \{ x \in \mathbb{R}^k : \| x \| \leq 1, x \geq 0 \} $, the spectral problem solved by EigenAlign is in fact itself a relaxation of the corresponding Gromov-Wasserstein problem.
Together with the observation that Gromov-Wasserstein finds a solution in a single step while EigenAlign requires two consecutive steps, this suggests that Gromov-Wasserstein network alignment may behave more favorably since the matching constraints are retained throughout the algorithm and can better inform the alignment. 

This spectral alignment framework can be extended to the problem of hypergraph alignment \cite{shen2018genome, lauziere2022exact, michoel2012alignment}, although hypergraphs introduce the additional complication that in general, hyperedges of a hypergraph may have edges of differing degree. For the simpler case of $K$-uniform hypergraphs (hypergraphs in which each hyperedge spans exactly $K$ nodes), the matching score matrix $A_{ijkl}$ can be extended to a matching score \emph{tensor} $A_{(i_1,j_1), \ldots, (i_K, j_K)}$ which has dimensions $|V \times V'|^K$. Writing $y$ as a $|V \times V'|$ matching vector, a generalized matching objective is
\begin{align}
\begin{split}
    &\max_y \left\langle A, \otimes_{i = 1}^K y \right\rangle\\
    \text{ s. t. } \qquad & y \ones \leq 1, y^\top \ones \leq 1, y \in \{ 0, 1 \}.
\end{split}\label{eq:hypergraph_alignment_qap}
\end{align}
In \cite{shen2018genome}, this problem is tackled in an analogous way to the EigenAlign algorithm (\ref{eq:feizi_alignment_qap}, \ref{eq:feizi_alignment_eigenproblem}, \ref{eq:feizi_alignment_linear_assignment}), that is, a relaxation of \eqref{eq:hypergraph_alignment_qap} onto the unit norm ball is derived which amounts to a generalized tensor eigenproblem which can be approximately solved using higher-order power iterations \cite{kolda2011shifted}. This is then projected back onto the constraint set by solving a linear assignment problem. Non-uniform hypergraphs are converted to uniform hypergraphs by introducing a dummy node repeatedly to hyperedges as needed until all hyperedges have the same degree. In contrast, co-optimal transport-based matchings of hypergraphs still boils down to a quadratic problem (as opposed to higher-order) in the coupling $\pi$, regardless of hypergraph degree. Furthermore, optimal transport handles non-uniform hypergraphs naturally. 

\begin{figure}[!ht]
  \centering 
  \includegraphics[width = \linewidth]{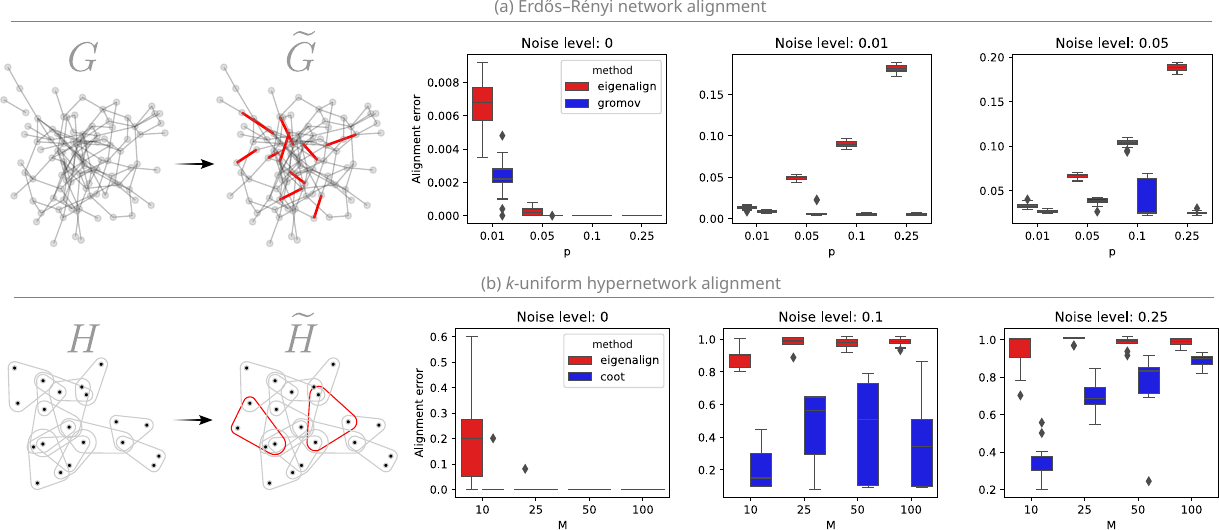}
  \caption{(a) Illustration of Erd\"os-R\'enyi random graph alignment problem with noise (edges due to noise shown in red); alignment errors (measured in terms of the distortion functional \eqref{eq:gromov_alignment_a}) achieved by EigenAlign and Gromov-Wasserstein under permutation and noise.
    (b) Illustration of random 3-uniform hypergraph alignment problem with noise (hyperedges due to noise shown in red); alignment errors (measured in terms of the objective \eqref{eq:hypergraph_alignment_qap}) achieved by higher-order EigenAlign and COOT under permutations and noise.
    }
  \label{fig:spectral_alignment_comparison}
\end{figure}

\subsubsection{Comparison to spectral network alignment for random graphs and hypergraphs}
\label{sec:match-random}

For this first set of experiments, we use synthetic datasets of graphs and hypergraphs.
In Figure \ref{fig:spectral_alignment_comparison}(a) we investigate the relative performance of spectral alignment and Gromov-Wasserstein alignment, considering Erd\"os-R\'enyi (ER) graphs of size $N = 100$ with parameter $p \in \{ 0.01, 0.05, 0.1, 0.25 \}$. For a randomly sampled ER graph $G$, we form a copy $\widetilde{G}$ in which nodes have been relabelled via a random permutation. Optionally, we also add noise in the form of random addition or deletion of edges independently with probability $q \in \{ 0, 0.01, 0.05 \}$. We align $G$ to $\widetilde{G}$ using both the implementation of EigenAlign from \cite{feizi2016spectral} and Gromov-Wasserstein using a proximal gradient algorithm (see Algorithm \ref{alg:proximal_partitioned}), similar to the approach taken by \cite{xu2019gromov}. Since the proximal gradient algorithm yields a coupling that is dense but potentially vanishingly small for most entries (i.e. strictly on the interior of the constraint set), we apply a ``rounding'' of the result onto an extreme point of the coupling polytope to yield a sparse permutation matrix. For each alignment, we measure the \emph{alignment error} by calculating the corresponding distortion functional \eqref{eq:gromov_alignment_a} to measure the alignment quality. In the absence of noise, $\widetilde{G}$ and $G$ are isomorphic since they are represented by adjacency matrices that are identical up to permutation, and a distortion of zero corresponds to a perfect matching. Non-zero noise breaks this isomorphism (so that the ground truth node matching may no longer be the ``right'' one after adding noise), so the lower the distortion the better the alignment. In this sense, the distortion is an objective measure of alignment quality rather than the coupling itself. In all cases we consider, we find that Gromov-Wasserstein finds an alignment that yields a lower distortion than EigenAlign, shown in Figure \ref{fig:spectral_alignment_comparison}(a). At a conceptual level, this can be understood since the Gromov-Wasserstein problem arises as a relaxation of the quadratic assignment problem \eqref{eq:feizi_alignment_qap} that accounts for the quadratic objective and the assignment constraints jointly, whereas the EigenAlign approach adopts a two step approach, first relaxing the assignment constraint to a norm ball constraint \eqref{eq:feizi_alignment_eigenproblem} and then projecting back onto the assignment polytope \eqref{eq:feizi_alignment_linear_assignment}. Because of this, the assignment constraints in the second step cannot inform the quadratic program in the first step. 

In Figure \ref{fig:spectral_alignment_comparison}(b) we turn to hypergraph alignments. For hypergraphs, the scope of the higher-order spectral alignment approach is limited to dealing with uniform hypergraphs, and furthermore the time and space complexity scale exponentially in the order of the hypergraph. We therefore consider random 3-uniform graphs for $N = 25$ nodes and $M = \{ 10, 25, 50, 100 \}$ hyperedges. Each hyperedge is obtained by sampling 3 nodes uniformly without replacement from the node set. Given a hypergraph $H$, we form a copy $\widetilde{H}$ by randomly relabelling nodes and hyperedges, and then replacing a fraction $q \in \{ 0, 0.1, 0.25 \}$ of hyperedges with independently sampled hyperedges. The spectral alignment approach only aligns nodes (since for uniform hypergraphs a node alignments also induces hyperedge alignments), so we quantify the quality of alignments in terms of the objective of \eqref{eq:hypergraph_alignment_qap} rather than the co-optimal transport distortion which depends on both node and hyperedge couplings. As in the graph alignment case, we find that co-optimal transport alignments (using again the proximal method of Algorithm \ref{alg:proximal_partitioned}) perform as well or better compared to spectral alignments in all cases.

Spectral hypergraph alignments are restricted to uniform hypergraphs and are computationally expensive, while co-optimal transport does not have these limitations.
Measuring the computation time for spectral alignment and co-optimal transport alignment for $N \in \{ 5, 10, 25\}$, we find that spectral alignment is several orders of magnitude more expensive in terms of runtime. Runs for $N = 50$ with spectral alignment have failed due to memory usage exceeding the available 32 GB. 

\subsubsection{Metabolic network alignment}
\label{sec:metabolic}

Metabolic networks (and chemical reaction networks more generally) are an example of systems in which higher-order relations are essential to retain information: chemical species may be modelled as nodes and reactions as hyperedges, which may involve any number of reactants simultaneously \cite{jost2019hypergraph}. 
We consider the metabolic networks of \emph{E. coli} and halophilic archaeon DL31, retrieved from the Kyoto Encyclopedia of Genes and Genomes (KEGG) database \cite{kanehisa2000kegg} with accession numbers \texttt{eco01100} and \texttt{hah01100} respectively.
We model each metabolic network as a labelled measure hypernetwork, where nodes are identified with metabolite compounds and hyperedges are identified with enzymes which catalyze reactions involving multiple compounds (multiple reactants and products). For simplicity, we discard directionality information and model the metabolic networks as undirected hypergraphs (i.e. we do not distinguish between reactants and products within each hyperedge).
For \texttt{eco01100} (the source network) we construct a measure hypernetwork with 984 metabolites and 1005 reaction terms, and for \texttt{hah01100} (the target network) a measure hypernetwork with 679 metabolites and 558 reaction terms. We find that the minimum and maximum hyperedge sizes are 2 and 9, respectively, in both the source and target hypergraphs. This verifies the heterogeneous, non-uniform nature of these hypernetworks. We visualize each network in Figure \ref{fig:metabolic_alignment}(a), showing the associations between compounds (nodes, red) and reactions (hyperedges, blue). In contrast to the previous synthetic example, we now must align two hypergraphs that are non-uniform and different in size. Within our framework, the unbalanced, fused hypergraph alignment scheme is the most suitable approach and we demonstrate the effectiveness of this method.

\begin{figure}[!ht]
  \centering
  \includegraphics[width = 0.9\linewidth]{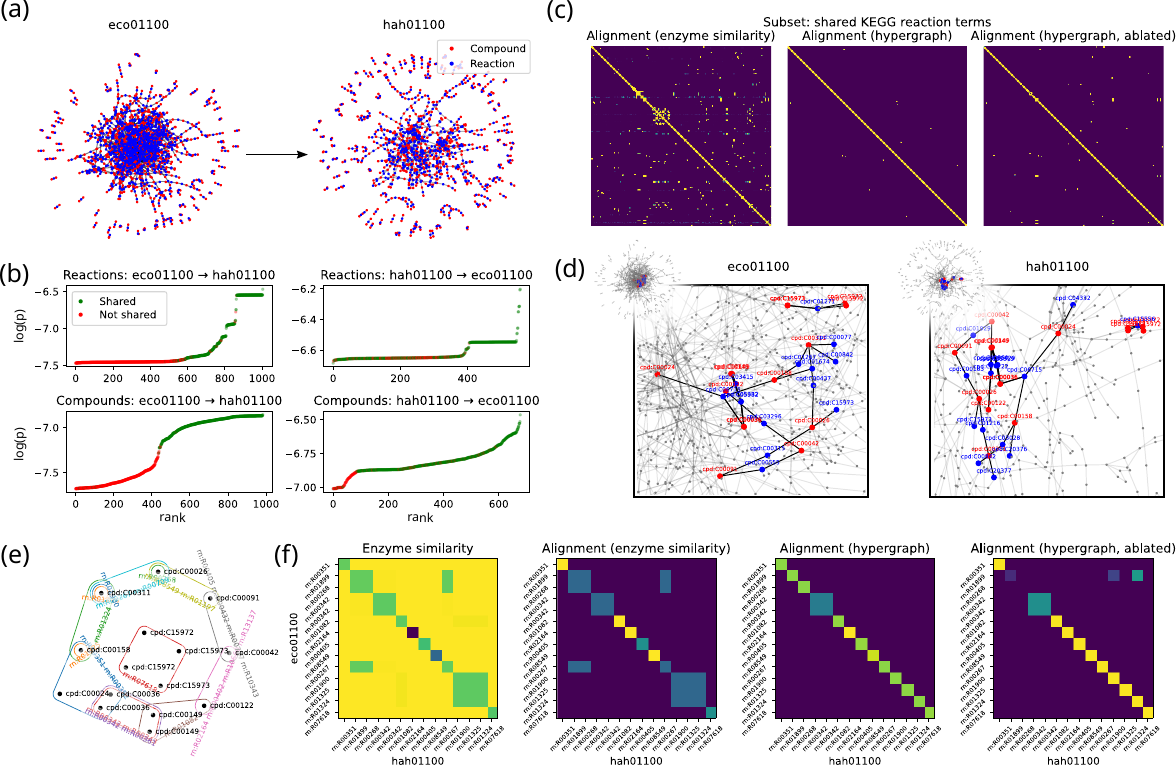}
  \caption{(a) Genome-scale metabolic networks of \emph{E. coli} (\texttt{eco01100}) and halophilic archaeon DL31 (\texttt{hah01100}).
    (b) Nodes (compounds) and hyperedges (enzymes) ranked by total out-going probability mass as found by unbalanced alignment, coloured by whether its true match is shared in the target network.
    (c) Hyperedge couplings for the subset of reaction terms (enzymes) common to both organisms, found using (left) enzyme similarity only, (middle) labelled hypergraph alignment (with both metabolite and enzyme similarities provided), and (right) ablated, labelled hypergraph alignment (with metabolite similarities provided, but not enzyme similarities). 
    (d) Zoom-in on conserved tricarboxylic acid (TCA) cycle subnetwork as shown in genome-scale metabolic network layout from (a). 
    (e) Hypergraph layout of TCA cycle subnetwork, shown as a rubber band diagram.
    (f) Alignment of reaction terms in TCA cycle subnetwork, from left to right: enzyme similarity matrix for TCA cycle reaction terms; and alignments found using enzyme similarity, labelled hypergraph alignment (with both metabolite and enzyme similarities provided), and ablated labelled hypergraph alignment (with metabolite similarities provided, but not enzyme similarities), respectively.
  }
  \label{fig:metabolic_alignment}
\end{figure}

As we mentioned previously in Section \ref{sec:relation_to_spectral_alignment}, this hypergraph alignment problem was addressed using a spectral approach by \cite{shen2018genome}. As the metabolic hypernetworks are non-uniform, dummy nodes are added to produce a $d$-uniform hypergraph, where $d$ is the maximum hyperedge degree in the original non-uniform hypergraph. This uniform hypergraph is then represented as a $|V|^d$ adjacency tensor. Due to the size of the alignment tensor, in \cite{shen2018genome} a fairly involved computational scheme is described that exploits its symmetry properties. Even so, distributed computing is necessary to speed up the alignment, which was reported to take over two hours to match the two networks (559 metabolites and 537 reactions for \texttt{hah01100}, 794 metabolites and 923 reactions for \texttt{eco01100}) \cite[Supplementary Materials]{shen2018genome}. 

For metabolite compounds and reaction terms, we construct pairwise cost matrices between the source and target using a similar approach to \cite{shen2018genome}. For metabolites, similarity scores are calculated using the cheminformatics package ChemmineR \cite{cao2008chemminer}. For any two enzymes $(e_1, e_2)$, the similarity score is taken to be $1/N(e_1, e_2)$ where $N(e_1, e_2)$ is the number of enzyme entries in the lowest common level in the Enzyme Commission (EC) classification \cite{bairoch2000enzyme}. 
Since there is not a one-to-one correspondence between the two metabolic networks, we solve an unbalanced variant (see Section \ref{sec:unbalanced_matchings}) of the labelled hypergraph alignment problem between the two hypernetworks using a proximal gradient variant of Algorithm \ref{alg:unbalanced_hypernetwork}, using $\alpha = 0.9, \varepsilon = 10^{-3}, \lambda = 0.1$ and 250 and 1000 inner and outer iterations respectively. This takes less than a minute utilizing 4 cores of an Intel Xeon Gold 6242 system, several orders of magnitude faster than the higher-order spectral method employed by \cite{shen2018genome}. 

To assess the quality of the matching, we use the fact that compounds and enzymes conserved between these two organisms are known. For a subset of nodes and hyperedges therefore, we have a biological ``ground truth'' correspondence to compare against. Since only parts of the two metabolic networks are shared, we expect the unbalanced matching to reflect this and assign more mass to shared components. In Figure \ref{fig:metabolic_alignment}(b) we show reactions and compounds in each organism, ranked by the (log) total mass assigned to it by the unbalanced matching algorithm. Components which have a true match in the target network are shown in green: it is clear that more mass is assigned to components with a true match, and non-overlapping components tend to be down-weighted. These weights can be thought of as a measure of confidence in the alignment. We remark that the \texttt{eco01100}  network is larger, so a significant fraction of its components do not have true matches in the \texttt{hah01100} network. Despite this, we still observe a separation in \texttt{hah01100} between components with and without true matches.

To understand how our alignment method depends on the node and hyperedge labels and relational information encoded in the hypergraph structure, we perform two additional alignments where some of this information is hidden (i.e.,~ablation study).
To study the performance of alignment using hyperedge label information alone without the hypergraph structure, we directly align enzymes (hyperedges) using the enzyme similarity matrix by solving a standard optimal transport problem using the proximal point method \cite{xie2020fast}. 
We also consider hyperedge alignment using node labels and the hypergraph structure, when hyperedge labels are hidden. To do this, we set the hyperedge-hyperedge cost matrix to zero and recompute the alignment with the same parameters using only the node-hyperedge incidence matrix and the node-node cost matrix.
In Figure \ref{fig:metabolic_alignment}(c), for the subset of enzymes or reaction terms shared between both organisms, we show alignments obtained using only hyperedge labels \emph{(enzyme similarity)}, the full labelled hypergraphs \emph{(hypergraph)}, and only node labels \emph{(hypergraph, ablated)}. Compared to the full labelled hypergraph alignment result, we find that using only enzyme similarity leads to a much more noisy alignment, with large amounts of mass assigned away from the diagonal. Suppressing hyperedge labels leads to a slightly worse alignment of hyperedges compared to the full labelled alignment, but still significantly cleaner than using hyperedge labels alone. 

Finally, we focus on the tricarboxylic acid (TCA) cycle, a fundamental metabolic process that is conserved between both organisms. In Figure \ref{fig:metabolic_alignment}(d) we highlight this subnetwork, and in Figure \ref{fig:metabolic_alignment}(e) we show its rubber band visualization. In Figure \ref{fig:metabolic_alignment}(f), we find that a full labelled hypergraph alignment near-perfectly matches the components, while the ablated hypergraph alignment without hyperedge labels again does slightly worse. In contrast, the enzyme similarity score does not provide full information about the matching, and hence alignment based on enzyme similarity alone performs much worse.
Together, these results indicate that utilizing the hypergraph structure in combination with label information is crucial for achieving a good alignment between the two metabolic networks, outperforming alignments where either label or relational information are suppressed.

\subsubsection{Multi-omics sample and feature alignment}
\label{sec:multi-omics}

Co-optimal transport has previously been employed for simultaneously matching samples and features between heterogeneous datasets \cite{titouan2020co}. One particularly popular example is that of multi-omics datasets, where two or more sets of features (e.g. gene expression and protein marker expression) are observed in samples (cells) \cite{tran2023unbalanced}. This problem can be cast in the setting of hypernetwork alignment by interpreting samples and features as nodes and hyperedges respectively, and the sample-by-feature data matrix as the membership function. In \cite{tran2023unbalanced} the application of unbalanced co-optimal transport was demonstrated to improve alignment quality, and in \cite{demetci2023revisiting} the augmented Gromov-Wasserstein matching (see Definition \ref{defn:generalized_network_distances}) is introduced: this corresponds to partitioned network matching when only node-node information is provided in addition to node-hyperedge relations, but not hyperedge-hyperedge information. Here we consider $1,000$ cells sampled from the same CITE-seq dataset as in \cite{demetci2023revisiting}, in which 15 genes and their corresponding marker proteins were measured. 

Partitioned measure networks allow pairwise relations \emph{within} as well as \emph{between} partitions to be modelled, so we incorporate pairwise similarities in each domain. This is in addition to the sample-feature information contained directly in the data matrix. We choose to capture this by calculating sample-sample and feature-feature correlations. Under the partitioned measure network alignment framework, we expect that pairwise structure between samples (respectively features) should be preserved by the alignment. To get an initial understanding of the data, we show the PCA embeddings and the partitioned network data constructed from each modality in Figure \ref{fig:cite_seq}(a): $\omega_{00}^{(\mathrm{RNA}, \mathrm{ADT})}, \omega_{11}^{(\mathrm{RNA}, \mathrm{ADT})}$ are correlation matrices between samples and features respectively, and $\omega_{01}^{(\mathrm{RNA}, \mathrm{ADT})}$ are the data matrices capturing cell-wise gene or marker expression. Put simply, we expect that correlated (anti-correlated) pairs of genes should be matched to correlated (anti-correlated) pairs of proteins, and similarly for cells between modalities. 

\begin{figure}[!ht]
  \centering
  \includegraphics[width = \linewidth]{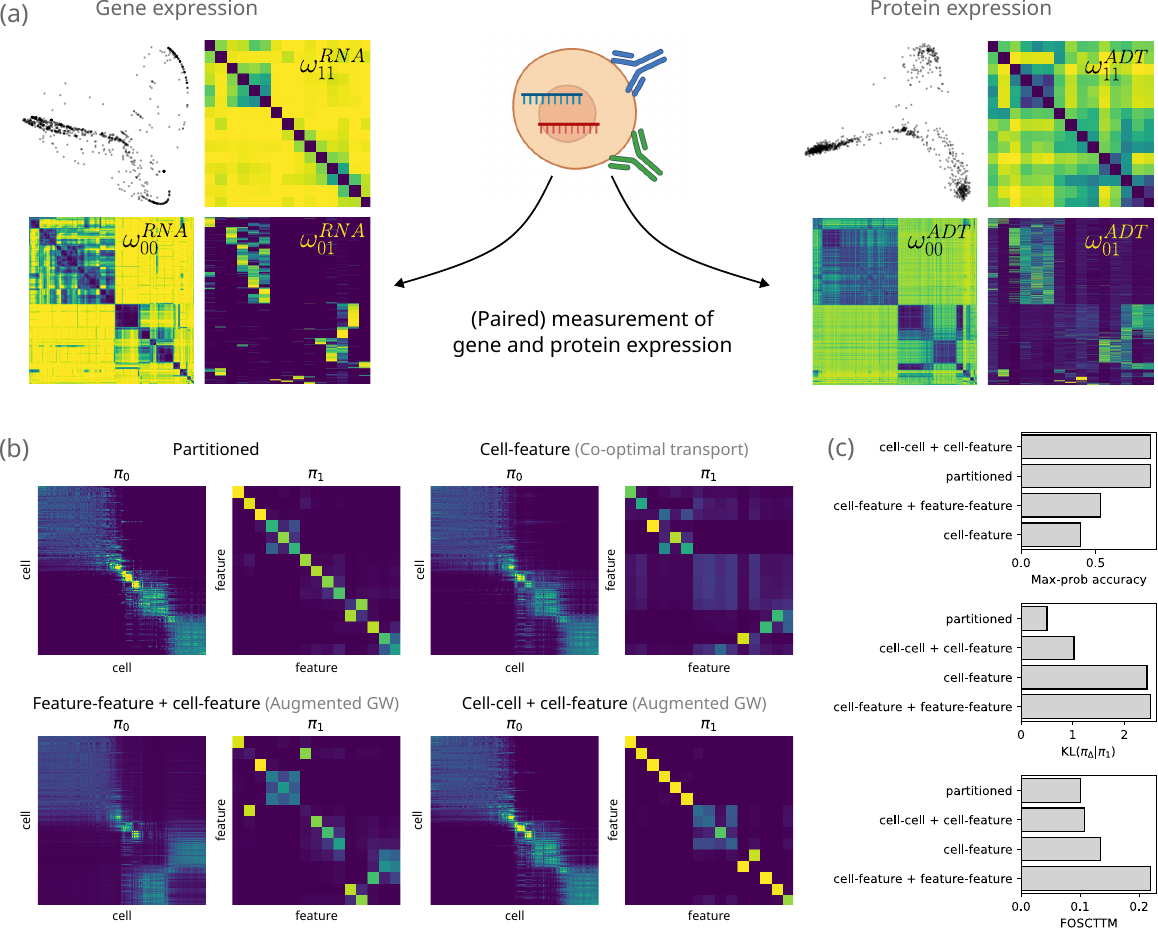}
  \caption{(a) Illustration of CITE-seq dataset shown in transcriptomic (gene expression or RNA) and surface marker (protein expression or ADT) modalities. For modality $m$, $\omega_{01}^m, \omega_{00}^m, \omega_{11}^m$ denote the data matrix, sample-sample similarities, and feature-feature similarities, respectively. (b) Sample and feature alignments obtained using partitioned network matching, co-optimal transport, and augmented Gromov-Wasserstein with sample-sample information and feature-feature information respectively. The ground truth matching for samples and features corresponds to the diagonal. (c) Quality of feature alignments in terms of (top) maximum probability assignment (higher is better);  (middle) reverse KL-divergence to the diagonal coupling (lower is better); and (bottom) quality of sample alignments in terms of FOSCTTM (lower is better).}
  \label{fig:cite_seq}
\end{figure}

We then compute an entropy-regularized alignment of the RNA and protein partitioned measure networks using Algorithm \ref{alg:projected_partitioned}. We introduce a parameter $\alpha \in [0, 1]$ to control the trade-off between the contribution of Gromov-Wasserstein type terms ($\omega_{00}, \omega_{11}$) and co-optimal transport type terms ($\omega_{01}$), scaling these inputs by $\sqrt{\alpha}$ and $\sqrt{1-\alpha}$ respectively. We choose values $\alpha \in \{ 0, 0.1, \ldots, 0.9, 1 \}$. For each partition, different levels of entropic regularization $\varepsilon_{0, 1} \in \{ \texttt{5e-4}, \texttt{1e-3}, \texttt{5e-3}, \texttt{1e-2}, \allowbreak \texttt{5e-2}, \texttt{0.1}, \texttt{0.5} \}$ are used, as it is well known that regularization level may play a role in the alignment quality \cite{demetci2023revisiting}. Finally, we consider the special cases where pairwise information (i.e. the Gromov-Wasserstein term) on samples, features, or both, are suppressed. We implement this by setting $\omega_{00}^{\mathrm{RNA}, \mathrm{ADT}}$, $\omega_{11}^{\mathrm{RNA}, \mathrm{ADT}}$ to zero as needed.

For each set of parameter values we compute the alignment, and then calculate the fraction of gene transcripts which are correctly matched to their corresponding protein in terms of maximum assigned probability. We show in Figure \ref{fig:cite_seq}(b) the best matchings obtained for the settings of partitioned matching, co-optimal transport, and augmented Gromov-Wasserstein (AGW) on samples and features respectively. In terms of identification of features (Figure \ref{fig:cite_seq}(c)), we found that the partitioned alignments and AGW with sample-sample information were both able to correctly assign 13/15 (87\%) features, in terms of maximum probability. On the other hand, AGW with feature-feature information and co-optimal transport correctly assigned 8/15 and 6/15 features respectively. While the fraction of correctly matched features by maximum probability gives an indication of the alignment accuracy, it does not account for the level of uncertainty in the matching. To account for this, we also calculate the KL divergence of the diagonal (ground truth) coupling relative to the alignment, reasoning that alignments that produce the correct matching with a higher certainty should have a lower divergence (i.e.,~lower divergence is better). We find that the partitioned matching produces a more informative alignment ($\KL = 0.502$) compared to AGW ($\KL = 1.027$), which can also be assessed visually from the couplings. Finally, in Figure \ref{fig:cite_seq}(c) bottom, we assess the quality of sample matchings in terms of the fraction of samples closer than true match (FOSCTTM) which is a standard performance metric in the single cell alignment literature \cite{demetci2023revisiting}, for which a lower value indicates a better alignment. Both the partitioned alignment and AGW with sample-sample information produce sample alignments of similar quality, whereas AGW with feature-feature information and co-optimal transport have worse performance.

\subsection{Partitioned networks for multiscale network matching}
\label{sec:multiscale}

Whereas the previous examples focus on hypergraphs (i.e.,~partitioned networks with $k = 2$), our framework can be used to model multiscale data by setting $k \ge 2$. This insight was obtained in \cite{chowdhury2023hypergraph}: a multi-scale graph with $k$ simplification levels can be modelled as $k$ coupled hypergraphs. In Section~\ref{sec:multiscale_matching_algorithm}, we show how the work of~\cite{chowdhury2023hypergraph} can be framed and extended in terms of $k$-partitioned measure networks. Specifically, we can model relations between nodes in the same simplification level, as well as between simplification levels using partitioned  measure networks. We demonstrate the application of our framework 
for matching geometric networks (obtained from 3D objects), as well as non-geometric protein-protein interaction networks. 

\subsubsection{Multiscale point cloud matching}
\label{sec:multiscale-point-cloud}

We apply multiscale matching to networks derived from 3D models of a wolf and a centaur from the TOSCA object database \cite{bronstein2008numerical}. In \cite{chowdhury2023hypergraph}, co-optimal transport was employed to find semantic matchings between two poses of the centaur graph across multiple scales. In their experiments, the two poses of the centaur graph have the same number of nodes at each level and are nearly identical in structure. The co-optimal transport framework of~\cite{chowdhury2023hypergraph} is also applicable to finding semantic matchings between a wolf and a centaur, where the two graphs are significantly different in their size, connectivity, and semantic components. Indeed, the ``true'' semantic correspondence between the objects is not one-to-one, since the wolf has four limbs and the centaur has six. In this section, we solve the multiscale object matching problem using a partitioned measure network formulation and compare it against previous approaches.  

For each input graph, a multi-scale topological simplification was produced using the heat kernel multiscale reduction of \cite[Section 5.3]{chowdhury2023hypergraph} with $k = 3$ simplification levels. We take each of the $\omega_{i, i+1}, \omega_{i, i+1}'$, for $0 \leq i < 2$ to be binary incidence matrices of node-hyperedge relations between successive reductions. Pairwise relations $\omega_{ii}, \omega_{ii}', 0 \leq i < 3$ are constructed from the graph shortest path distances on each simplification level.

\begin{figure}[!ht]
  \centering
  \includegraphics[width = \linewidth]{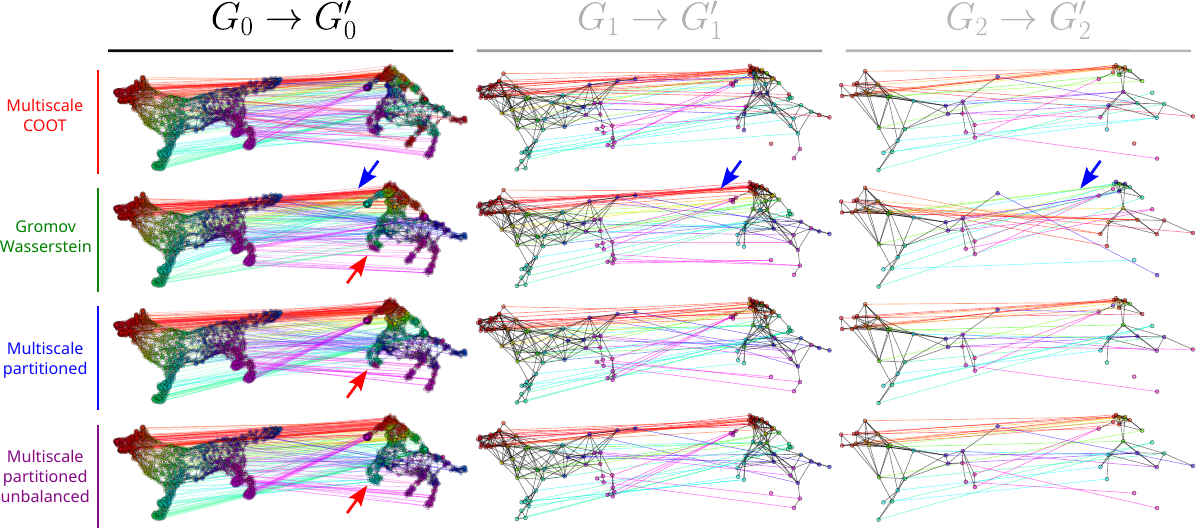} 
  \caption{Multiscale network matchings: TOSCA wolf and centaur. Each row corresponds to a different algorithm, and each column corresponds to a graph simplification level. 1st row: multiscale COOT~\cite{chowdhury2023hypergraph}. 2nd row: Gromov-Wasserstein measure network coupling obtained  independently on each simplification level. 3rd and 4th row: the multiscale alignment using $k$-partitioned networks (3rd row) and its unbalanced variant (4th row). Blue arrows: matching of the wolf's head is inconsistent across scales for Gromov-Wasserstein alignment. Red arrows: discontinuity in centaur's front leg disappears with an unbalanced alignment.}
  \label{fig:tosca_multiscale_alignment}
\end{figure}

In Figure \ref{fig:tosca_multiscale_alignment} we visualize the alignments obtained using multiple methods: (1st row) multiscale COOT (using the algorithm of \cite{chowdhury2023hypergraph}); (2nd row) Gromov-Wasserstein measure network matching obtained independently on each simplification level; (3rd row) the multiscale alignment using $k$-partitioned networks (Algorithm \ref{alg:projected_partitioned}) and (4th row) its unbalanced variant (Algorithm \ref{alg:unbalanced_partitioned}). At each simplification level, each node in the centaur graph is connected to a node in the wolf graph with the maximum matching probability.
While multiscale COOT finds a matching that is largely consistent across scales, the matching is very noisy especially at fine scales as evidenced by discontinuities in the colour gradient in the centaur. On the other hand, independent Gromov-Wasserstein matchings at each simplification level leads to more locally consistent matchings (continuous colour gradients), but the matching fails to be consistent across scales; this is apparent, for instance, by looking at the matchings of the head/neck regions across simplification levels (blue arrows), as well as the arms of the centaur. The multiscale partitioned network alignment produces a semantically reasonable, consistent matching at each scale, correctly matching the head, neck, hind legs, as well as tail regions. Due to the difference in the number of limbs across each model, we observe mismatches among the arms of the centaur, which is not entirely surprising. Interestingly, a small portion of the front legs of the centaur model is matched to the hind legs of the wolf (red arrows), seen as a discontinuity in the colour gradient. However, these mismatches disappear in the unbalanced partitioned measure network alignment. In other words, there is a lack of one-to-one correspondence between the limbs of a centaur and those of a wolf causing a number of mismatches, and unbalanced matching may alleviate these issues.  

In addition to the visual assessment of semantic matchings, we show in Table \ref{tab:loss_values}(a) different components of the objective \eqref{eq:multiscale_alignment_problem} for the matchings found by multiscale COOT, Gromov-Wasserstein, and partitioned measure network alignment, respectively. This table provides us with an unbiased quantification of alignment quality directly in terms of the distortion. These results confirm our visual observations from Figure \ref{fig:tosca_multiscale_alignment}: Gromov-Wasserstein alignment at each scale produces the minimal Gromov-Wasserstein loss reflecting preservation of pairwise relations at each individual scale, but a very high COOT loss indicates a lack of consistency across scales. Conversely, multiscale COOT minimizes the COOT loss while producing the highest Gromov-Wasserstein loss, which suggests the reverse. The partitioned alignment on the other hand yields a much lower Gromov-Wasserstein loss, while achieving a COOT loss only marginally worse than that found by multiscale COOT. These results demonstrate that the partitioned multiscale alignment is able to incorporate both pairwise and multiscale information effectively to simultaneously align networks at multiple scales. 

\begin{table}[!ht]
  \begin{subtable}{\linewidth}
    \centering
    \begin{tabular}{lcc}
        \hline
        & \textbf{COOT\_loss} & \textbf{GW\_loss} \\
        \hline
        Gromov-Wasserstein & 0.036033 & 0.010600 \\
        Multiscale COOT & 0.018261 & 0.019294 \\
        Partitioned alignment (projected gradient) & 0.019534 & 0.012787 \\
        \hline
    \end{tabular}
    \caption{TOSCA wolf and centaur: COOT and GW distortion losses.}
  \end{subtable}
  \begin{subtable}{\linewidth}
    \centering
    \begin{tabular}{lcccc}
        \hline
        \textbf{Method} & \textbf{COOT\_loss} & \textbf{GW\_loss} & \textbf{Node correctness} & \textbf{Edge correctness} \\
        \hline
        Gromov-Wasserstein               & 0.058684   & 0.016181 & 0.613546 & 0.967680 \\
        Multiscale COOT             & 0.028927   & 0.021701 & 0.036853 & 0.565782 \\ 
        Partitioned (proximal)       & 0.030732   & 0.016125 & 0.585657 & 0.964916 \\
        Partitioned (block) & 0.032357   & 0.014105 & 0.597610 & 0.951580 \\
        \hline
    \end{tabular}
    \caption{Protein-protein interaction network: COOT and GW distortion losses, as well as node and edge correctness.}
  \end{subtable}
\caption{Gromov-Wasserstein (GW) and co-optimal transport (COOT) loss for multiscale network matching. (a) TOSCA object matching. (b) Protein-protein interaction network matching.}
  \label{tab:loss_values}
\end{table}

\subsubsection{Multi-scale biological network matching}
\label{sec:multiscale-protein}

Our multi-scale network matching approach is not limited to geometric graphs, e.g., those constructed from a point cloud sampled from 3D objects. We now consider a dataset of protein-protein interaction (PPI) networks \cite{vijayan2015magna++}, in which nodes and edges correspond to protein species and biochemical interactions respectively. We take $G_0$ to be the PPI network of high-confidence interactions among $1,004$ proteins, and $G_0'$ to be the PPI network with $20\%$ more low-confidence interactions. For each of $G_0, G_0'$ we construct a progressive topological simplification using the heat kernel reduction described in \cite{chowdhury2023hypergraph}, yielding multi-scale reductions $\{ G_i \}_{i = 0}^2$ and $\{ G_i' \}_{i = 0}^2$. We visualize in Figure~\ref{fig:ppi}(a) each multiscale reduction, in which nodes are coloured by the leading non-trivial eigenvector of $L_{G_0}$, the graph Laplacian of $G_0$. Nodes in the low-confidence networks $G_0'$ are also coloured using the ground truth node correspondence. We then calculate matchings at each scale by employing partitioned measure network alignment, multiscale co-optimal transport, as well as independent Gromov-Wasserstein matchings at each simplification level. 

In this example, we care about the \emph{exact} node matchings and so we opt to solve the exact, unregularized network matching problem: for the partitioned measure network alignment as well as Gromov-Wasserstein measure network alignment, we use a proximal gradient algorithm (Algorithm \ref{alg:proximal_partitioned}). In Figure~\ref{fig:ppi}(b), the matchings from each of these algorithms are shown across all three simplification levels. We also show a set of ``ground truth'' couplings: between $G_0$ and $G_0'$, this is the identity coupling, whereas between $G_i$ and $G_i'$ (for $i = 1, 2$) an approximate ground truth coupling is found by calculating a matrix of pairwise correlations between coarse-grained nodes and then solving a linear assignment problem. 
From a visual assessment of each of the matchings, we observe that both the partitioned and Gromov-Wasserstein alignment find matchings between $G_0$ and $G_0'$ that resemble the ground truth, while multiscale COOT performs quite poorly: this reflects the fact that multiscale COOT is unaware of the pairwise adjacency information at each scale. Between $G_1$ and $G_1'$, we find that the partitioned alignment continues to resemble the ground truth, but both Gromov-Wasserstein and multiscale COOT matchings begin to deviate significantly. Finally, between $G_2$ and $G_2'$, we find that Gromov-Wasserstein continues to appear differently from the ground truth.

\begin{figure}[!ht]
  \centering 
  \includegraphics[width = \linewidth]{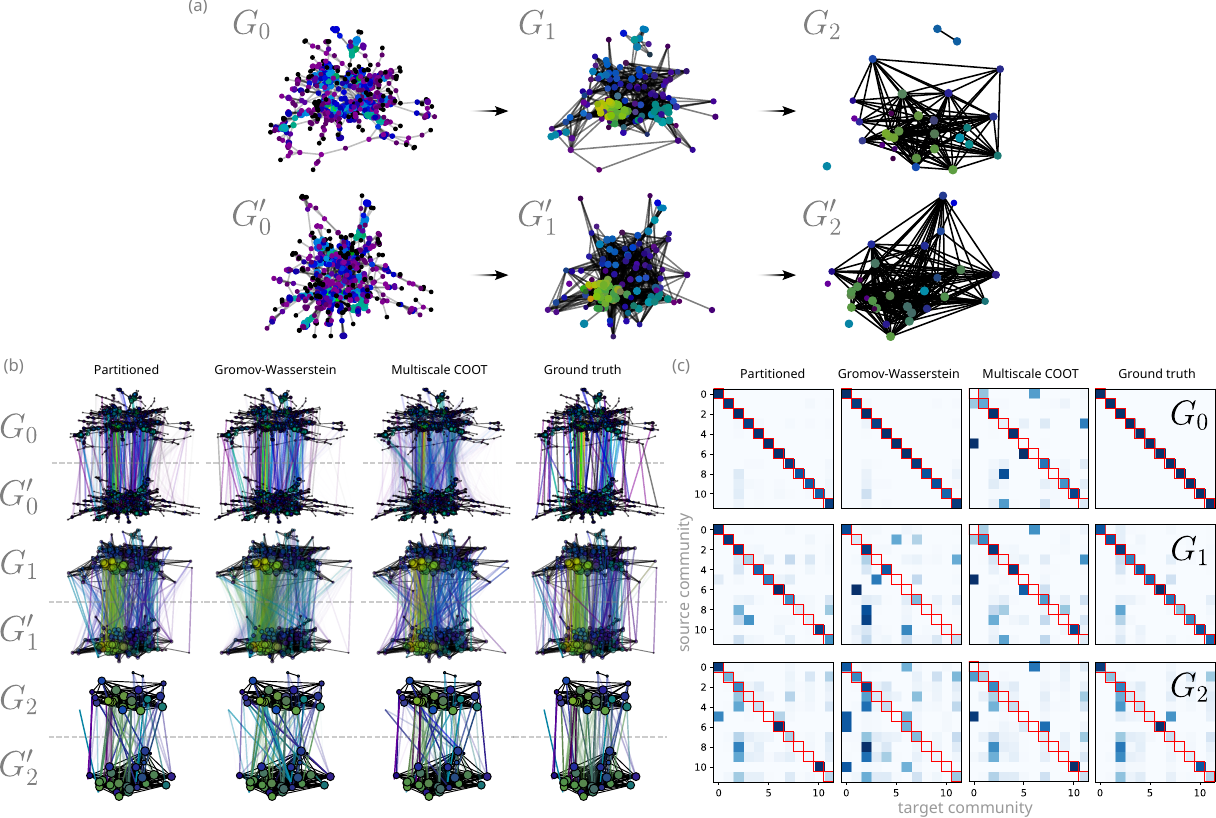}
  \caption{Multiscale matchings for protein-protein interaction networks: (a) Network layouts of successive simplifications $\{G_i \}$ and $\{ G_i' \}$ for $i = 0, 1, 2$ with nodes coloured by the leading nontrivial eigenvector of $L_G$. (b) Matchings found at each scale using (from left to right):  partitioned measure network matching; multiscale COOT matching; Gromov-Wasserstein measure network matching at each level; and the ground truth matching. (c) Matchings of Leiden communities induced by the node-level matching between $G_i$ and $G_i'$ at each simplification level $i$. ``Ground truth'' shows the best possible matching of Leiden communities using couplings at this level of granularity.}
  \label{fig:ppi}
\end{figure}

Due to the non-geometric nature of the input graphs, an effective direct visualization of the matchings is very difficult. To demonstrate the difference in matching results in a clearer way, we employ the Louvain community detection algorithm~\cite{blondel2008fast} which finds a partitioning of $G_0$ (and hence $G_0'$) into $m = 12$ communities. Together with the hypergraph coupling between simplification levels $i$ and $i+1$, the matching $\pi_i$ at simplification level $i$ induces a matching of communities. In Figure~\ref{fig:ppi}(c), we show the induced community matchings for each alignment method, as well as for the ground truth. At level 0, the partitioned alignment and Gromov-Wasserstein both produce nearly perfect alignments, while multiscale COOT performs poorly. At level 1, however, both Gromov-Wasserstein and multiscale COOT perform poorly, while partitioned alignment continues to perform well. Finally, at level 2 from the ground truth matching, it is apparent that too much information is lost by applying coarsening to the graphs to correctly identify communities. However, the matching induced by the partitioned alignment is still closer to the ground truth, for instance in $L_1$ norm ($L_1 = 0.265$) compared to Gromov-Wasserstein ($L_1 = 0.319$) and multiscale COOT ($L_1 = 0.508$).

In Table~\ref{tab:loss_values}(b), we show the loss terms similar to the previous TOSCA example: we observe that Gromov-Wasserstein fails to find consistent matchings across scales as evidenced by a high COOT loss. On the other hand, multiscale COOT leads to poor preservation of pairwise relations within each simplification level, indicated by a high Gromov-Wasserstein loss. In contrast, partitioned alignments are able to find multiscale matchings that are consistent within each scale as well as across scales. Furthermore, partitioned alignment methods yield node and edge correctness scores for level 0 that are comparable to Gromov-Wasserstein, which was found to outperform most other competing alignment methods in~\cite{xu2019scalable}. 

\subsection{Geodesics and Fr\'echet means} 
\label{sec:frechet_means}

In Section~\ref{sec:calculating_gradients}, we introduce the Fr\'echet functional \eqref{eq:frechet_functional} on the space of partitioned measure networks and calculate its gradient. Recall from Theorem~\ref{thm:isometric_embeddings} that the spaces of measure networks and measure hypernetworks isometrically embed into the space of partitioned measure networks, we recover from \eqref{eq:frechet_functional} the Fr\'echet functional on measure networks~\cite{chowdhury2020gromov, peyre2016gromov} and measure hypernetworks~\cite{chowdhury2023hypergraph} as special cases. For simplicity we consider the unlabelled case here, although in general our results can be straightforwardly extended to measure networks with labels valued in an inner product space (see e.g.,~\cite{vincent2021online}). 

In practice, a stationary point of the functional \eqref{eq:frechet_functional} can be found via gradient descent on the space of partitioned measure networks using the ``blow-up'' scheme of~\cite{chowdhury2020gromov} which progressively carries out alignment of network representatives as per Proposition~\ref{prop:aligned}. This approach exploits the empirical observation that optimal couplings of measure networks tend to be sparse (see e.g.,~\cite[Appendix C]{chowdhury2020gromov} and \cite[Theorem 2]{chowdhury2021generalized}), allowing the geodesics of Proposition~\ref{prop:interpolation_geodesic_labelled} to be explicitly constructed. While the approach of \cite{chowdhury2020gromov} handles only measure networks, we implement an extension of \cite[Algorithms 1-3]{chowdhury2020gromov} in order to handle measure hypernetworks. We remark here that, although the question of sparsity of optimal couplings is open in the general quadratic case, in the setting of co-optimal transport between measure hypernetworks, the alternating scheme of Algorithm~\ref{alg:coot} is guaranteed to yield a sparse coupling. This is because each iterate is the solution of a \emph{linear} program and is therefore sparse. In what follows, we refer to this approach as the \emph{free support} method, since the size of the networks is determined as part of the optimization procedure. 

As an alternative to the more involved free support method, we also consider fixing the networks to those described by matrices of a fixed size, as done in~\cite{peyre2016gromov}. This makes the optimization much easier: as detailed in Section \ref{sec:barycenters_fixed_support}, this problem can be solved by alternating between solving independent alignment problems and a closed form update for the barycenter. We refer to this approach as the \emph{fixed support} method. 

\begin{figure}[!ht]
  \centering 
  \includegraphics[width = \linewidth]{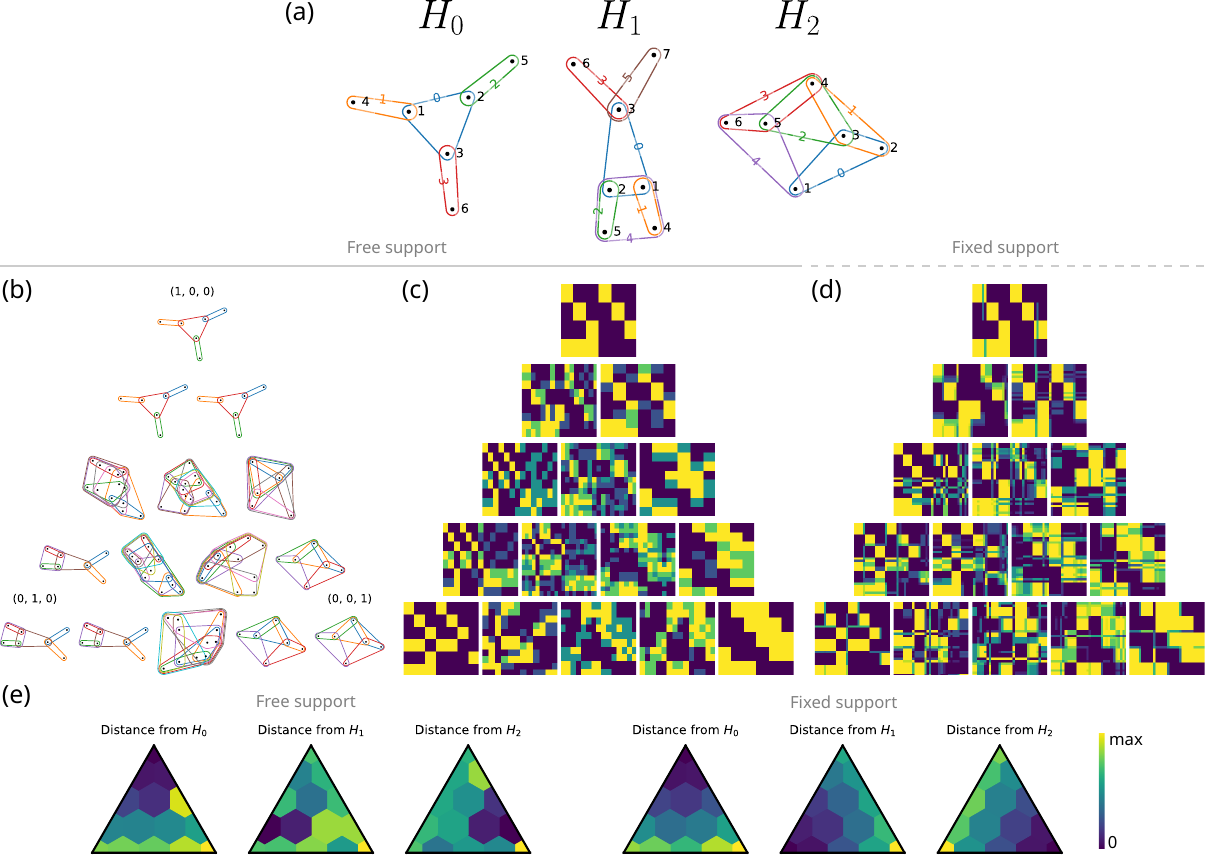}
  \caption{Geodesic interpolations between three hypergraphs $H_0, H_1, H_2$, shown in (a) as ``rubber-band'' diagrams. Interpolations found using the free-support method are visualized in (b) as rubber band diagrams and also directly as their membership functions $\omega$ in (c). (d) Fixed support (using $32 \times 32$ matrices) interpolations visualized as membership functions. (e) Hypergraph distance $d_{\Hc, 2}(H_i, H_\mathrm{interp})$ between each of the three input hypergraphs and geodesic interpolations, for the free support and fixed support barycenters, respectively.}
  \label{fig:geodesics_hypergraphs}
\end{figure}

In Figure~\ref{fig:geodesics_hypergraphs}(a), we take three hypergraphs $H_0, H_1$, and $H_2$ shown as ``rubber-band'' diagrams in which nodes are represented as points, and hyperedges are shown as ``rubber bands'' around the convex hull of contained nodes. 
We interpolate between these hypergraphs by computing the Fr\'echet means of $\{ H_0, H_1, H_2 \}$ with various weights $(w_0, w_1, w_2)$, using both the free support and fixed support methods. Since the rubber band visualization is unable to represent weighted node-hyperedge relationships, a threshold (in this case we chose $\omega(x, y) > 1/4$) is applied to the membership function $\omega(x, y)$ before visualization. As a result, Figure~\ref{fig:geodesics_hypergraphs}(b) cannot reflect the true nature of the interpolation in the measure hypernetwork space. In Figure~\ref{fig:geodesics_hypergraphs}(c), we show the interpolated function $\omega$, making the interpolation readily apparent. We remark that the size (number of nodes and hyperedges) of each interpolating measure hypernetwork is determined by the blowup scheme (in general, larger than each of the inputs), and that the ordering of rows and columns in the visualization is arbitrary. In Figure~\ref{fig:geodesics_hypergraphs}(d), we show the membership functions of barycenters computed using the fixed support method. We note a close resemblance between the results of the free and fixed support approaches.

Finally, in Figure~\ref{fig:geodesics_hypergraphs}(e), we compute the measure hypernetwork distance (per Definition~\ref{defn:generalized_network_distances}) from each interpolated hypernetwork to each input hypernetwork. For true interpolations, we would expect the distance to vary affinely across the simplex. However, in the case of free support barycenters, we observe that is not always the case. This suggests that local minima may have played a role in the calculation of the weighted Fr\'echet means and calculating the distances to each endpoint, reflecting the non-convex nature of the distance and alignment computation. On the other hand, for fixed support barycenters, we recover the expected trend. This reflects the computationally simpler nature of the fixed support method, and may indicate that the fixed support barycenters are more accurate representations of the true hypergraph barycenter.

\subsection{Linear and non-linear dictionary learning}
\label{sec:dictionary-learning}

We may extend the barycenter question and ask for \emph{several} characteristic partitioned measure networks that best describe an ensemble of partitioned measure networks, rather than a single barycenter or {Fr\'echet} mean. One common method is to learn a small basis of representatives (or archetypes), such that each ensemble member can be approximated by a convex combination of these basis elements. Also known as dictionary learning, this has become a classical analysis approach for vector-valued data~\cite{lee1999learning} and has recently been extended to graphs~\cite{vincent2021online, xu2022representing} and topological descriptors such as merge trees~\cite{LiPalandeYan2023} using the Gromov-Wasserstein framework. We extend dictionary learning to the setting of partitioned measure networks, which also covers the settings of measure networks and measure hypernetworks. 

\subsubsection{Nonlinear (geodesic) dictionary learning}
\label{sec:nonlinear-dictionary-learning}

We first recall the geodesic dictionary learning problem first stated in Section \ref{sec:geodesic_dictionary_learning}. 
Given an ensemble of $N$ $k$-partitioned measure networks $\{P_1, \dots, P_N\}=\{P_i \in \Pc_k\}_{i=1}^{N}$, we aim to find a dictionary $\Dc=\{D_1, \dots, D_m\} =\{D_j \in \Pc_k\}_{j=1}^{m}$ (where $m \ll N$) such that each $P_i$ could be described by elements in $\Dc$. 
Formally, we denote each $P_i = \big(X_i, \mu_i, \omega_i \big)$, and the geodesic dictionary learning problem is
\begin{align}
  \min_{ \{ D_j \}_{j = 1}^m \in \Pc_k,  \{ \alpha_i \}_{i = 1}^N \in \Delta^m } \frac{1}{N} \sum_{i = 1}^N d_{\Pc_k} \left( B(\Dc, \alpha_i), P_i \right)^2,
  \label{eq:nonlinear_dictionary_learning}
\end{align}
where $B(\Dc, \alpha)$ is the barycenter operator \eqref{eq:barycenter_operator} for $\Dc$ and each $\alpha \in \Delta^m$ encodes the corresponding coefficients. We derive formal expressions for the gradients of this function in Section \ref{sec:geodesic_dictionary_learning}. Like related methods \cite{schmitz2018wasserstein, xu2020gromov}, solving problem \eqref{eq:nonlinear_dictionary_learning} is a non-convex, bi-level minimization problem which is not straightforward even to find a local optimum. While problems of this nature can be solved using more involved schemes such as \cite{xu2022representing}, we propose to simplify the problem by taking $B(\Dc, \alpha)$ to be the \emph{fixed support} barycenter operator, where we fix the size (i.e. number of nodes) of the barycenter \emph{a priori} and approximate it iteratively (see Section \ref{sec:barycenters_fixed_support}, and \cite{peyre2016gromov}). In practice, we also fix the sizes of the dictionary atoms $\{ D_j \}_{j = 1}^m$ \emph{a priori} and seek a local minimum solution for dictionary networks of fixed size by a simple gradient descent.

\subsubsection{Linear dictionary learning}
\label{sec:linear-dictionary-learning}

Even after fixing the support size of barycenters and dictionary atoms, solution of the bi-level problem \eqref{eq:nonlinear_dictionary_learning} is computationally demanding due to the need for inner-loop computations of the barycenter operator. As an alternative, \emph{linear} dictionary learning approaches have been proposed \cite{vincent2021online, rolet2016fast}, in which the (Fr\'echet) barycenter operator $B(\Dc, \alpha)$ is replaced with its Euclidean equivalent, a weighted superposition of the atoms $\Dc[\alpha] = (X, \mu, \sum_{j = 1}^m \alpha_j \omega_{\Dc_j})$. 
Since reconstructions from the dictionary are carried out in the Euclidean space, this eliminates the nested optimization arising from the barycenter computation, at the cost of departing from the natural geodesic structure of the space of (partitioned) measure networks. We remark that in the setting where the data $P_i \in \Pc_k, 1 \leq i \leq N$ are sufficiently close together (in the sense of the injectivity radius of the exponential map), that linear dictionary learning is equivalent to the nonlinear case. 

The linearized equivalent of \eqref{eq:nonlinear_dictionary_learning} is
\begin{align}
  \min_{ \{ \omega_{D_j} \}_{j = 1}^m \in L^2(X^2, \mu^{\otimes 2}), \: \{ \alpha_i \}_{i = 1}^N \in \Delta^m} \frac{1}{N} \sum_{i = 1}^N d_{\Pc_k} \left( \Dc[\alpha_i], P_i \right)^2, \label{eq:dictionary_learning}
\end{align}
where, for brevity, we denote the linear combination of atoms in $\Dc$ with coefficients in $\alpha$ to be 
\[
  \Dc[\alpha] = \left(X, \mu, \sum_{j = 1}^m \alpha_j \omega_{D_j}\right) \in \Pc_k. 
\]
In this formulation, we ask for $m \ll N$ atoms, set $\{ \omega_{D_j} \}_{j = 1}^m \in L^2(X^2, \mu^{\otimes 2})$, and for each input network $P_i$, we work with a corresponding set of coefficients $ \{ \alpha_{ij} \}_{j = 1}^m$ such that the reconstructed network described by incidence matrix $\Dc[\alpha_i]$ is close to $P_i$ in the sense of the optimal transport metric. In the above, we fix the partitioned measure space to $(X, \mu)$ respectively. We note that the computation of $d_{\Pc_k}$ involves solution of a nonlinear program for the coupling $\pi$. We can expand the $d_{\Pc_k}$ terms within the objective \eqref{eq:dictionary_learning}:
\begin{align}
  \min_{ \{\omega_{D_j} \}_{j = 1}^m \in L^2(X^2, \mu^{\otimes 2}), \: \{ \alpha_i \}_{i = 1}^N \in \Delta^m, \{ \pi_i \in \Pi_k(\mu, \mu_i) \}_{i=1}^N } \frac{1}{N} \sum_{i = 1}^N \left\langle L\left( \Dc[\alpha_i], \omega_{P_i} \right), \pi_i \otimes \pi_i \right\rangle.
\end{align}
Minimizing in $\pi_i$ can be done independently and in parallel for each of the $1 \leq i \leq N$ inputs. Fixing $\pi$, we have a non-convex quadratic program in $\omega_{D_j}$ and the coefficients $\alpha_i$. We solve the problem \eqref{eq:dictionary_learning} using a stochastic projected gradient descent.

\begin{figure}[!ht]
  \centering
  \includegraphics[width = \linewidth]{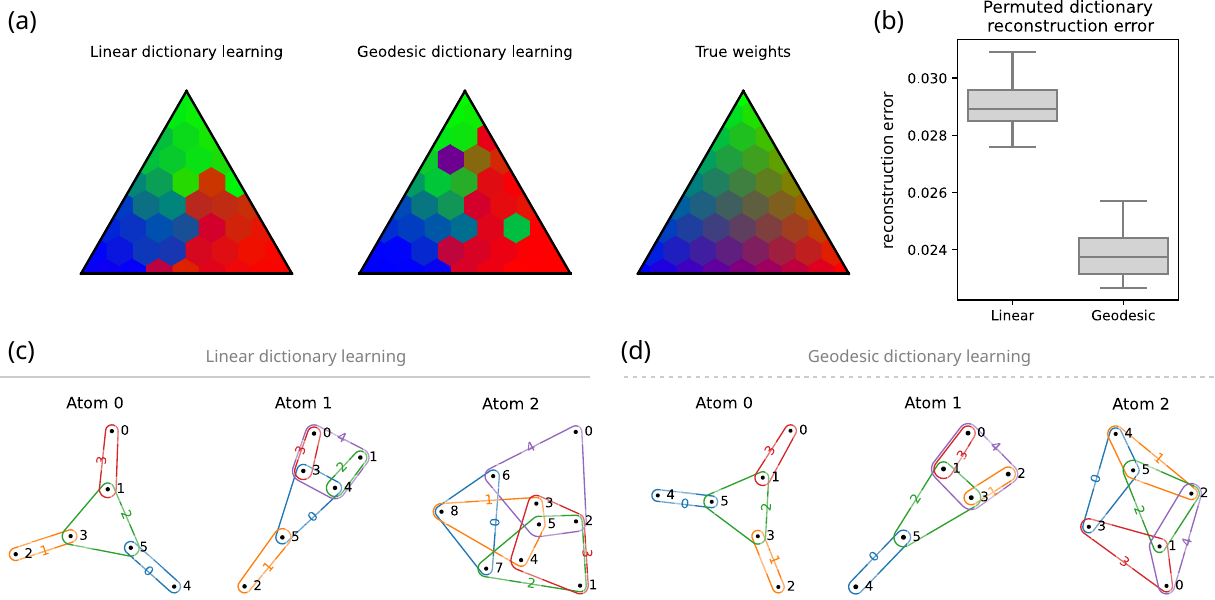}
  \caption{(a) Learned dictionary weights for each interpolation instance using linear dictionary learning (correlation = 0.894) and geodesic dictionary learning (correlation = 0.809), compared to the true weights. (b) Measure hypernetwork distances between true networks and dictionary learning reconstructions, under random permutation of dictionary elements. (c) Atoms learned by linear dictionary learning. (d) Atoms learned by geodesic dictionary learning.}
  \label{fig:dictionary_learning}
\end{figure}

For an experiment, we consider again the example from Section~\ref{sec:frechet_means} involving three hypergraphs. We generate barycenters of these three hypergraphs with mixture weights uniformly spaced across a barycentric grid inside a simplex (i.e., a triangle in this example) using the ``blow-up'' algorithm. This procedure produces an ensemble of 45 hypergraphs across the grid, in which the number of nodes ranged from 6 to 23, and the number of hyperedges between 4 and 17. By construction, the input hypergraphs $H_0, H_1, H_2$ are the ground truth atoms, and weights $\{ w_{ij} \}_{1 \leq i \leq N, 0 \leq j \leq 2}$ serve as the ground truth mixture coefficients. In Figure \ref{fig:dictionary_learning}(a) we show the dictionary weights learned by the linear and geodesic dictionary learning respectively, compared to the true weights. Although geodesic dictionary learning (at least in theory) accounts correctly for the underlying geometry of hypergraphs, we find that linearized dictionary learning yields better results in practice. A potential explanation for this result is that linearization avoids the need to solve a non-convex inner loop problem, replacing this with Euclidean averaging.

To illustrate what we gain by using geodesic over linearized dictionary learning, for each method we randomly permute the learned dictionary atoms and re-calculate the reconstruction error. We show in Figure \ref{fig:dictionary_learning}(b) the reconstruction error across 10 independent permutations. In the space of measure hypernetworks, this amounts to changing the choice of representative for the dictionary atoms. This leads to an increased reconstruction error for linearized dictionary learning, relative to geodesic dictionary learning: this tallies with geodesic dictionary learning objective being defined independent of choices of representative. Overall, we find that linearized dictionary learning has an advantage over its geodesic variant due to computational simplicity, and this is the algorithm we use in the following sections. 

Finally, we show in Figure \ref{fig:dictionary_learning}(c) and (d) the atoms learned by linear and geodesic dictionary learning respectively. Both methods could reliably reconstruct the atoms shown in Figure \ref{fig:geodesics_hypergraphs}(a), which are used to generate the input ensemble. 

\begin{figure*}
  \centering 
  \includegraphics[width = \linewidth]{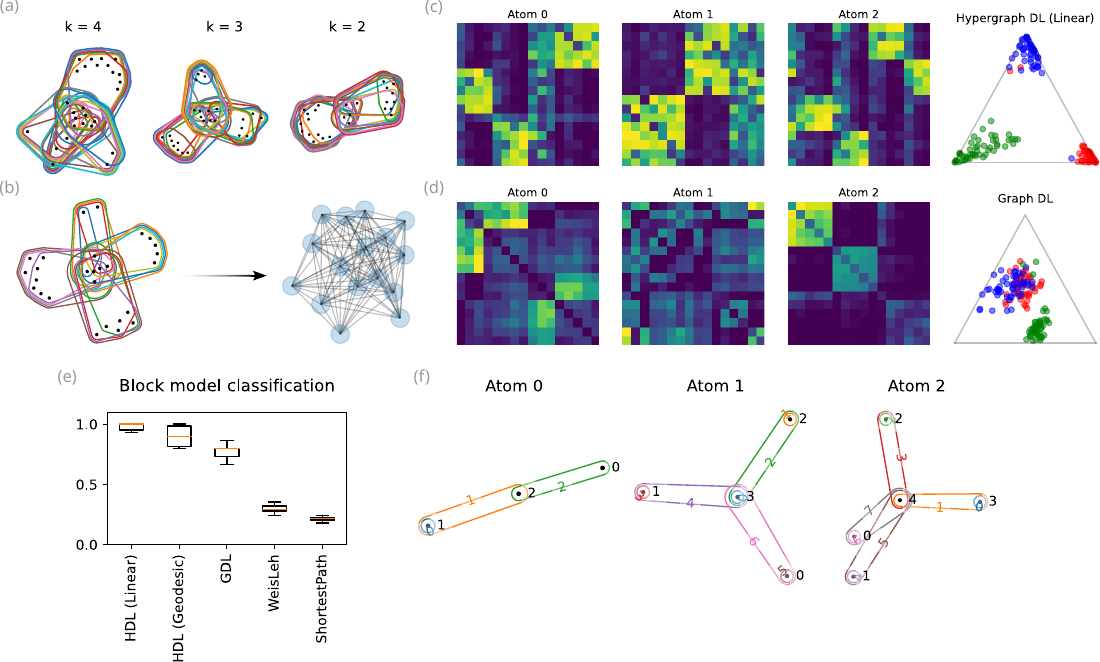}
  \caption{(a) Sample instances from the hypergraph stochastic block model for $k = 4, 3, 2$ blocks respectively.
    (b) Illustration of hypergraph flattening.
    (c) left: learned hypergraph atoms capture the three underlying block models; note the presence of hyperedges that span the node set between blocks; right: Barycentric projection of learned weights, coloured by the ground truth label. 
    (d) Learned graph atoms (left) and weights (right) from Gromov-Wasserstein dictionary learning with flattened hypergraphs as input.
    (e) 10-fold cross-validation results for SVM classification of hypergraph samples.
    (f) Coarse-grained hypergraph representations of dictionary atoms learned by linear HDL.}
  \label{fig:blockmodel}
\end{figure*}

\subsubsection{Example: stochastic block model for hypergraphs}
\label{sec:stochastic-block-model}

To further validate our numerical algorithm using synthetic data, we devise a simple generative model for random non-uniform hypergraphs that is analogous to the stochastic block model in the case of simple graphs, similar to the model introduced in~\cite{dumitriu2021partial}. For a (possibly random) number of nodes $|V|$ and a fixed number of blocks $k$, we partition the node set $V$ into $k$ blocks of size $\lfloor |V| / k \rfloor$, with one block containing any remainders. A random number, $|E|$, of hyperedges are then sampled. With probability $p$, a sampled hyperedge will span $m < \lfloor N / k \rfloor$ nodes, all residing in one of the $k$ blocks chosen uniformly at random. Otherwise the hyperedge will consist of $m$ nodes sampled uniformly from all nodes in $V$. The key parameters at play are $k$ and $p$, controlling respectively the number of blocks and the level of noise.

We set up three block models with $k = 2, 3, 4$ blocks and $p = 2/3$; 50 hypergraphs are sampled from each model. For each hypergraph, the number of nodes and hyperedges are chosen uniformly from $[16, 32]$ and $[24, 32]$ respectively. We show examples of sampled hypergraphs for each $k$ in Figure \ref{fig:blockmodel}(a). In total, we generate a dataset of 150 hypergraphs. 
We apply both linearized and geodesic hypergraph dictionary learning (HDL) algorithms to this dataset to learn three atoms.

To compare against graph-based methods, we apply the graph dictionary learning (GDL) method of \cite{vincent2021online} using flattened hypergraphs. 
These flattened hypergraphs are obtained by taking the sum of hyperedges, i.e.,~$\frac{1}{|E|} \sum_{i = 1}^{|E|} \ones_{e_i} \ones_{e_i}^\top$, that is, putting a weighted connected component in place of each hyperedge; see Figure \ref{fig:blockmodel}(b). As can be seen from the learned atoms and coefficients in Figure \ref{fig:blockmodel}(c) and (d), HDL is able to accurately learn distinct atoms corresponding to different values of $k$. By comparison, GDL finds atoms that appear to be mixed.

We use the learned coefficients from HDL and GDL as features for support vector machine (SVM) classification. In Figure \ref{fig:blockmodel}(e) we show the accuracy in 10-fold cross validation for both (linear and geodesic) HDL and GDL. For comparison, we also show accuracy achieved by two commonly used graph kernels, the Weisfeiler-Lehmann isomorphism test and the shortest path kernel \cite{siglidis2020grakel}, applied to the flattened hypergraphs. We find that both variants of HDL achieve high accuracy, outperforming GDL. We also observe that linearized HDL outperforms geodesic HDL. Strikingly, the two graph kernel approaches perform much worse than either HDL or GDL. Finally, in Figure \ref{fig:blockmodel}(f), we show rubber-band diagrams of the atoms found by linearized HDL.

\begin{figure}[!ht]
  \centering
  \includegraphics[width = \linewidth]{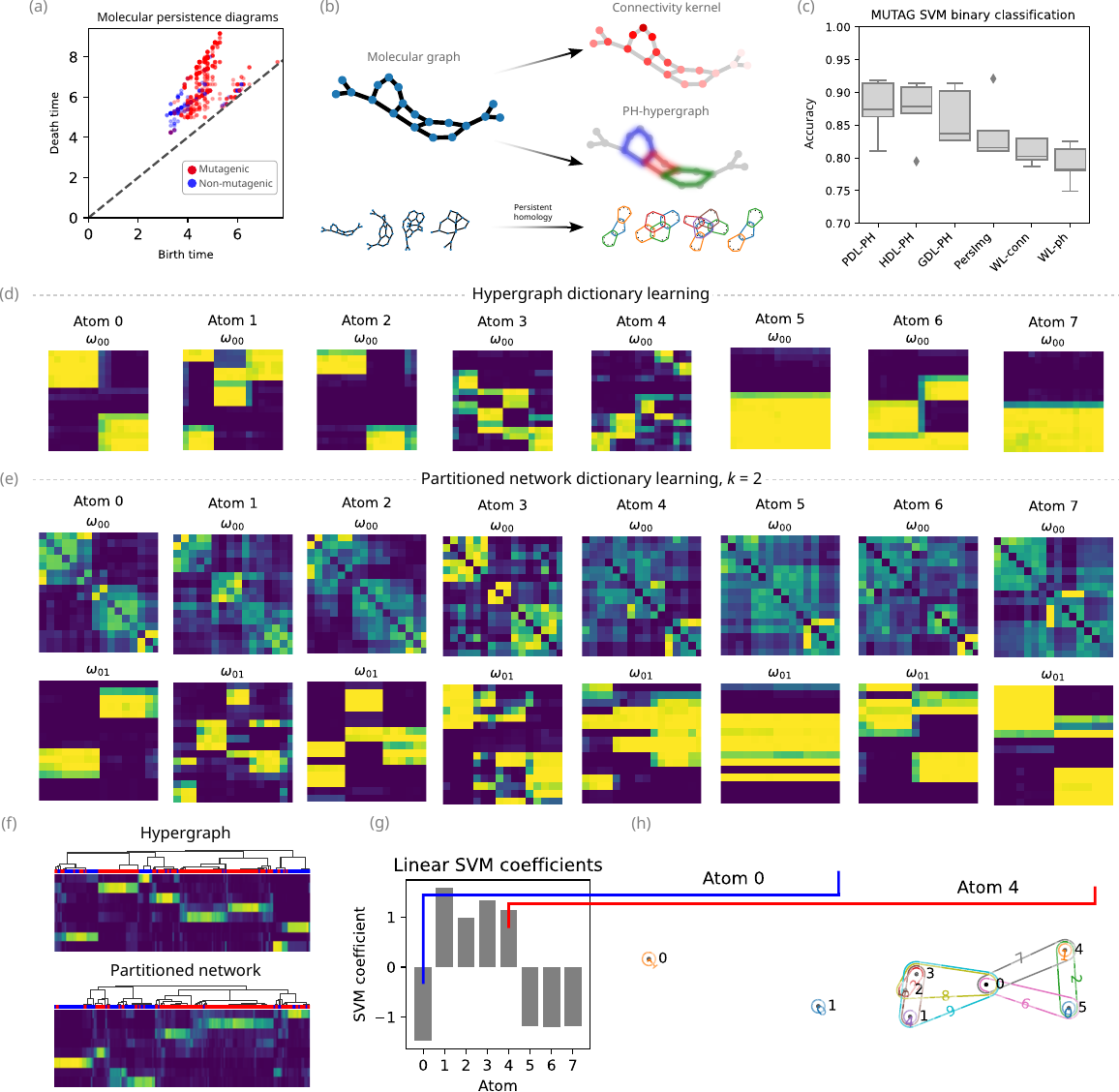}
  \caption{(a) Superimposed persistence diagrams (1-dimensional persistent homology or $H_1$) for molecular structures in the MUTAG dataset, coloured by mutagenicity. (b) Illustration of connectivity kernel and persistent homology (PH) hypergraph construction from molecular graphs. (c) SVM binary classification accuracy for mutagenicity, using different representations or kernels, shown over 5-fold cross validation. (d) Atoms learned by hypergraph dictionary learning (HDL). (e) Atoms learned by partitioned network dictionary learning (PDL). (f) Dictionary weights learned from HDL and PDL, respectively. (g) Coefficients learned from HDL weights by linear SVM. (h) Atom 0 (predicted to be associated with non-mutagenic compounds) and Atom 4 (predicted to be associated with mutagenic compounds). }
  \label{fig:mutag}
\end{figure}

\subsubsection{Example: mutagenicity dataset}
\label{sec:mutagenicity}

A real-world example is the well known MUTAG dataset \cite{debnath1991structure} containing 188 small molecule structures, labelled by their mutagenicity (mutagenic or non-mutagenic). This is a standard benchmarking dataset that has been widely adopted for testing graph learning algorithms \cite{siglidis2020grakel}. 
Starting from a molecular graph where nodes are atoms and edges are chemical bonds, we propose to lift these chemical structures to partitioned measure networks as follows. Using a graph heat kernel \cite{chowdhury2021generalized}, we construct an atom-atom connectivity network encoding proximity of atoms in the molecule. We also construct a persistent homology (PH) hypergraph following the procedure described in \cite{barbensi2022hypergraphs} to encode information about geometric cycles: nodes and hyperedges correspond to atoms and PH generators, respectively.
We show the ensemble of persistence diagrams from this dataset in Figure \ref{fig:mutag}(a), where points are coloured by the ground truth mutagenicity label; see~\cite{otter2017roadmap} for an user-friendly introduction to persistent homology and~\cite{EdelsbrunnerLetscherZomorodian2002} for the seminal work on the topic. 
A separation between mutagenic and non-mutagenic compounds can be visually discerned from these persistence diagrams, suggesting that persistent homology may be sensitive to molecular features that play a role in mutagenicity.  
In Figure \ref{fig:mutag}(b) we illustrate the computation of the connectivity kernel and the PH-hypergraph from an example molecular graph. 

Taking either the PH-hypergraphs or (heat kernel, PH)-partitioned networks as input data, we run linearized partitioned network dictionary learning for $k = 8$ atoms and extract the learned atoms and coefficients. We reason that the learned dictionaries should capture aspects of the molecular structure that are predictive of molecular properties such as mutagenicity. Using the learned coefficients, we train a RBF-kernel SVM for binary classification of mutagenicity. We also apply the linear Gromov-Wasserstein dictionary learning algorithm of \cite{vincent2021online} using the clique expansion of the PH-hypergraph as input. 
For reference, we consider several other popular approaches for encoding geometric and topological information in machine learning tasks. Specifically, we compare to persistence images \cite{adams2017persistence} (PersImg), as well as the Weisfeiler-Lehmann graph kernel \cite{shervashidze2011weisfeiler} using either the molecular connectivity graphs (WL-conn) or clique-expansion of the persistent homology hypergraph (WL-ph).

We find that both hypergraph dictionary learning (HDL) and partitioned network dictionary learning (PDL)  achieved comparable performance, suggesting that information on PH generators alone is sufficient to distinguish between mutagenic and non-mutagenic compounds in the majority of cases; see Figure~\ref{fig:mutag}(c). On the other hand, Gromov-Wasserstein dictionary learning exhibits worse performance despite being provided with PH information in the form of the clique-expansion of the PH hypergraph. This reflects the loss of higher-order information incurred by ``flattening'' of a hypergraph to a graph~\cite{srinivasan2021learning}. Similarly, we find that classification using persistence images (encoding information about the persistence diagram, but not node-generator membership information) as well as Weisfeiler-Lehman kernels (encoding only pairwise relationships) perform relatively poorly. In order to ensure a fair comparison across the different methods we consider, we do not use node labels in this analysis. We remark that our dictionary learning algorithm can be straightforwardly extended to incorporate vector-valued node labels~\cite{vincent2021online} and we expect that including that additional information would further improve classification performance. 

Figures~\ref{fig:mutag}(d) and (e) illustrate the atoms found by HDL and PDL, respectively. The learned atoms of the partitioned network dictionary consist of pairwise node-node similarities, as well as node-hyperedge memberships. The partitioned network atoms capture both local detail as well as topological information in a coupled fashion. By comparison, with only hypergraph information, nodes with membership in the same homology generators may be indistinguishable. The dictionary coefficients in Figure~\ref{fig:mutag}(f) illustrate how each molecule in the dataset is decomposed in terms of soft membership to each of the $k = 8$ archetypes (atoms or representatives). From visual inspection, it is already clear that the learned topics are largely able to disentangle mutagenic compounds from non-mutagenic ones. To quantify this, we train a linear SVM on the hypergraph dictionary coefficients, and extract the contribution of each atom towards the ``mutagenic'' class; see Figure \ref{fig:mutag}(g). We find that atoms are clearly separated into mutagenicity and non-mutagenicity contributing factors. In Figure~\ref{fig:mutag}(h), we show example PH-hypergraph archetypes that are indicative of mutagenicity or non-mutagenicity. For instance, Atom 0, which contributes towards non-mutagenicity, displays a simple PH-hypergraph structure with two generators that do not share nodes. On the other hand, Atom 4 contributes towards mutagenicity and displays a complex PH-hypergraph structure with many interlinked generators. These findings suggest that compounds with more cycles are more likely to be mutagenic, and this is consistent with the chemical literature~\cite{debnath1991structure}. 

\section{Discussion}
\label{sec:discussion}

We develop a theoretical footing for the analysis of generalized network objects, modelled by the space of $k$-partitioned measure networks. 
We equip this space with a family of metrics $d_{\Pc^p_k}$ ($p \ge 1$) that extends the well-known $p$-Gromov-Wasserstein metric originally developed for measure metric spaces \cite{sturm2023space, memoli2011gromov} and recently applied to measure networks \cite{chowdhury2019gromov} and measure hypernetworks \cite{chowdhury2023hypergraph}. When $p = 2$, we further provide a geometric characterization of the space of partitioned measure networks in terms of geodesics, curvature bounds, as well as its tangent space.
We additionally consider the case where additional \emph{labels} (valued in a metric space) are available. We prove metric properties in the labelled setting, and for $p = 2$, we show that our geometric characterizations also apply (when labels are valued in an inner product space).
Based on these ideas, we provide a range of numerical examples illustrating the applicability of our framework across multiple domains in network analysis and data science.
We believe our work will be of broad interest to the network science, geometry, and statistics communities.

Our work leaves several open avenues for future research. Among these, providing a geometric characterization of labelled networks when labels are valued in more general spaces, such as Riemannian manifolds (for instance, conditions for uniqueness of geodesics and curvature). Additionally, a rigorous analysis of functionals on the space of partitioned measure networks and their gradient flows remains to be constructed. Whereas we have introduced notions of tangent vectors and gradients and have formally shown their calculation and application, we have not addressed issues such as the existence and uniqueness of minimizer in the general case.

On the application front, some of the example applications we have presented could be fruitful problems to study separately. Among these, we think that multiscale network alignment and partitioned network dictionary learning are particularly interesting.
As an example of a specific application requiring extensions of our theory, in \cite{zhang2024topological}, we develop a \emph{partial transport} variant for applications in topological data analysis. 
Finally, we remark that there are some alternative formulations of optimal transport applied to tensors \cite{kerdoncuff2022optimal} that our framework does not cover.
It remains an open question whether similar theoretical results to ours can be established for tensors.

\section*{Acknowledgement}
SYZ acknowledges support from the Australian Government Research Training Program and an Elizabeth and Vernon Puzey Scholarship from the University of Melbourne.
MPHS acknowledges funding through an Australian Research Council Laureate Fellowship (FL220100005). 
BW, FL, and YZ were partially supported by grants from U.S. Department of Energy (DOE) DE-SC0021015 and U.S. National Science Foundation (NSF) IIS-1910733 and NSF IIS-2145499. TN was supported by NSF grants DMS 2107808 and DMS 2324962. 

\bibliographystyle{plain}
\bibliography{refs-networks}

\appendix

\section{Details on numerical algorithms}
\label{sec:algorithms_appendix}

In this section, we follow the notations introduced in Section \ref{sec:algos_intro}.

\subsection{Co-optimal transport}\label{sec:rt}

By taking $k = 2$, $\omega_{ii} = 0$, $\omega_{ij} = \omega_{ji}$, 
and $L$ defined in \eqref{eq:distortion_tensor}, it is easy to verify the following identity:
\begin{align*}
  \langle L(\omega_{12}, \omega_{12}'), \pi_1 \otimes \pi_2 \rangle &= \langle L(\omega_{12}, \omega_{12}') \otimes \pi_2, \pi_1 \rangle \\
  &= \langle L(\omega_{12}^\top, (\omega_{12}')^\top) \otimes \pi_1, \pi_2\rangle \\
  &= \langle L(\omega_{12}^\top, {(\omega_{12}')}^\top), \pi_2 \otimes \pi_1 \rangle.  
\end{align*}
We obtain from \eqref{eq:discrete_labelled_partitioned_distance} the case of matching between labelled measure hypernetworks $H = (X_1, \mu_1, X_2, \mu_2, \omega_{12})$ and $H' = (X'_1, \mu'_1, X'_2, \mu'_2, \omega'_{12})$:
\begin{align}
  \min_{\substack{\pi_1 \in \Pi(\mu_1, \mu_1'), \pi_2 \in \Pi(\mu_2, \mu_2')}} \langle L(\omega_{12}, \omega'_{12}), \pi_1 \otimes \pi_2 \rangle + \langle C_1, \pi_1 \rangle + \langle C_2, \pi_2 \rangle + \varepsilon_1 \Omega_1(\pi_1) + \varepsilon_2 \Omega_2(\pi_2), 
\end{align}
where for full generality we include the possibility of a regularization of each of the $\pi_i$ when $\varepsilon_i > 0$, as per \eqref{eq:entropic_discrete_labelled_partitioned_distance}. 

The unregularized problem is a bilinear program in $(\pi_1, \pi_2)$ since the objective can be rewritten as $\langle L(\omega_{12}, \omega_{12}') + C_1 \oplus C_2, \pi_1 \otimes \pi_2 \rangle$ where $(A \oplus B)_{ijkl} = A_{ij} + B_{kl}$. In the case of unlabelled measure hypernetworks where $C_1 = C_2 = 0$, this has been studied in detail by \cite{titouan2020co, chowdhury2023hypergraph}, among others. 
The alternating scheme of \cite{titouan2020co} for finding a stationary point presents itself:
\begin{align}
  \begin{split}
    \pi_1 \gets \min_{\pi_1 \in \Pi(\mu_1, \mu_1')} \langle &L(\omega_{12}, \omega_{12}') \otimes \pi_2 + C_1, \pi_1 \rangle + \varepsilon_1 \Omega_1(\pi_1),   \\ 
    \pi_2 \gets \min_{\pi_2 \in \Pi(\mu_2, \mu_2')} \langle &L(\omega_{12}^\top, \omega_{12}'^\top) \otimes \pi_1 + C_2, \pi_2 \rangle + \varepsilon_2 \Omega_2(\pi_2). 
  \end{split}
\end{align}
Employing the identity \cite[Proposition 1]{peyre2016gromov} we have for $\pi \in \Pi(\mu, \mu')$ that 
\begin{equation}
  L(\omega, \omega') \otimes \pi = \sum_{kl} L(\omega, \omega')_{\cdot\cdot kl} \otimes \pi_{kl} = \eta(\omega, \omega') - \omega \pi ({\omega'})^\top, 
\end{equation}
where
\begin{equation}
  \eta(\omega, \omega') = \frac{1}{2} \left( \omega^{\wedge 2} \right) \mu \ones^\top + \frac{1}{2} \ones \left( {\omega'}^{\wedge 2} \mu' \right)^\top,
\end{equation}
and $\omega^{\wedge 2}, {\omega'}^{\wedge 2}$ are understood entrywise. $\eta(\omega, \omega')$ depends only upon $(\mu, \mu')$, the marginals of $\pi$.

\begin{algorithm}
\caption{Alternating minimization: labelled hypergraphs (co-optimal transport)}
\begin{algorithmic}[1]
  \State \textbf{Input:} Incidence matrices $\omega_{12}, \omega_{12}'$, probability measures $\mu_i, \mu_i'$, $i = 1, 2$, label cost matrices $C_{1, 2}$ (optional).
  \State \textbf{Parameters:} entropic regularization levels $\varepsilon_1, \varepsilon_2 \ge 0$ (optional)
\State Initialize couplings:  $\pi_i \gets \mu_i \otimes \mu_i', \: i = 1, 2.$
\For{$t = 1, 2, \dots, \texttt{max\_iter}$}
  \State \[
    \pi_1 \gets \argmin_{\pi_1 \in \Pi(\mu_1, \mu_1')} \langle L(\omega_{12}, \omega_{12}') \otimes \pi_2 + C_1, \pi_1 \rangle + \varepsilon_1 \Omega_1(\pi_1),
  \]
  \State \[
    \pi_2 \gets \argmin_{\pi_2 \in \Pi(\mu_2, \mu_2')} \langle L(\omega_{12}^\top, \omega_{12}'^\top) \otimes \pi_1 + C_2, \pi_2 \rangle + \varepsilon_2 \Omega_2(\pi_2), 
  \]
\EndFor
\State \textbf{Output:} couplings $(\pi_1, \pi_2)$
\end{algorithmic}
\label{alg:coot}
\end{algorithm}

\subsection{General matchings of partitioned measure networks}

For general $k$-partitioned measure networks, the problem \eqref{eq:discrete_labelled_partitioned_distance} is quadratic in $(\pi_i)_{i = 1}^k$.
We propose to obtain an approximate solution to this problem using block coordinate descent separately in each of the $\pi_i$, while holding $(\pi_j)_{j \ne i}$ fixed. Each block update amounts to solution of a problem closely resembling Fused Gromov-Wasserstein matching \cite{vayer2020fused}:
\begin{equation}
  \min_{\pi_i \in \Pi(\mu_i, \mu_i')} \frac{1}{2} \langle L(\omega_{ii}, \omega'_{ii}), \pi_i \otimes \pi_i \rangle + \langle M[\pi_{-i}] + C_i, \pi_i \rangle + \varepsilon_i \Omega_i(\pi_i), \quad 1 \leq i \leq k,
  \label{eq:partitioned_bcd_update}
\end{equation}
where $\pi_{-i} = (\pi_j)_{j \ne i}$ and 
\begin{align}
  \begin{split}
  M[\pi_{-i}] &= \sum_{j \ne i} \left(\frac{1}{2} L(\omega_{ij}, \omega_{ij}') + \frac{1}{2} L(\omega_{ji}^\top, \omega_{ji}'^\top)\right) \otimes \pi_j \\
              &= \frac{1}{2} \sum_{j \ne i} \left[ \eta(\omega_{ij}, \omega'_{ij}) - \omega_{ij} \pi_j ({\omega'}_{ij})^\top + \eta(\omega_{ji}^\top, {\omega'}_{ji}^\top) - \omega_{ji}^\top \pi_j \omega'_{ji} \right],
  \end{split}
  \label{eq:definition_M}
\end{align}
Each subproblem \eqref{eq:partitioned_bcd_update} in $\pi_i$ amounts to the minimization of a non-convex quadratic objective on a closed convex domain, and so a stationary point can be found using the conditional gradient algorithm of \cite[Algorithm 1]{vayer2020fused}.
We remark that a similar algorithm for the case of augmented measure networks was introduced in \cite{demetci2023revisiting}. 

\begin{algorithm}
\caption{Alternating minimization: labelled $k$-partitioned networks}
\begin{algorithmic}[1]
\State \textbf{Input:} Matrices $\{ \omega_{ij} \}_{i, j = 1}^k, \{ \omega_{ij}' \}_{i, j = 1}^k$, probability measures $\mu_i, \mu_i'$, $1 \leq i \leq k$, label cost matrices $(C_i)_{i = 1}^k$ (optional)
\State \textbf{Parameters:} entropic regularization levels $\varepsilon_i \ge 0, 1 \leq i \leq k$ (optional)
\State Initialize couplings:  $\pi_i \gets \mu_i \otimes \mu_i', \: 1 \leq i \leq k.$
\For{$t = 1, 2, \dots, \texttt{max\_iter}$}
  \For{$1 \leq i \leq k$}
    \State
    \[
      \pi_{i} \gets \argmin_{\pi_i \in \Pi(\mu_i, \mu_i')} \frac{1}{2} \langle L(\omega_{ii}, \omega'_{ii}), \pi_i \otimes \pi_i \rangle + \langle M[\pi_{-i}] + C_i, \pi_i \rangle + \varepsilon_i \Omega_i(\pi_i), \quad \text{with } M[\pi_{-i}] \text{ as per \eqref{eq:definition_M}}
    \]
  \EndFor
\EndFor
\State \textbf{Output:} couplings $\{ \pi_i \}_{i = 1}^k$
\end{algorithmic}
\label{alg:alternating_partitioned}
\end{algorithm}

\subsection{Proximal gradient methods}\label{sec:proximal_gradient_algorithms}

As an alternative to relying on exact solvers for the unregularized problem, an entropic proximal gradient algorithm can also be used to solve the problem \eqref{eq:discrete_labelled_partitioned_distance}. These algorithms have been shown to perform favorably in terms of computational complexity as well as empirical results \cite{xu2019gromov, xie2020fast}. Writing $\Lc$ to be the objective function of \eqref{eq:discrete_labelled_partitioned_distance}, 
\begin{equation}
  \Lc(\pi_1, \ldots, \pi_k) = \frac{1}{2} \sum_{i, j = 1}^k \langle L(\omega_{ij}, \omega_{ij}'), \pi_i \otimes \pi_j \rangle + \sum_{i = 1}^k \langle C_i, \pi_i \rangle,  
  \label{eq:discrete_labelled_partitioned_distance_obj}
\end{equation}
 and choosing a regularization level (inverse step size schedule) $\lambda^t > 0$, a proximal gradient descent on the objective $\Lc$ starting from an initialization $(\pi_i^{0})_{i = 1}^k$ generates the iterates for $t \ge 0$:
\begin{equation}\label{eq:proximal_point_algorithm}
  (\pi_i^{t+1})_{i = 1}^k \gets \argmin_{\pi_i \in \Pi(\mu_i, \mu_i'), \: 1 \leq i \leq k} \Lc(\pi_1, \ldots, \pi_k) + \lambda^t \KL(\otimes_{i} \pi_i | \otimes_{i} \pi_i^{t}), 
\end{equation}
where $\KL(\alpha | \beta)$ denotes the (generalized) Kullback-Leibler divergence between probability distributions (positive measures)
\begin{equation}
  \KL(\alpha | \beta) = \langle \alpha, \log (\diff\alpha / \diff\beta) \rangle - |\alpha| + |\beta|. \label{eq:kullback_leibler}
\end{equation}
Replacing $\Lc$ with its linearization about $(\pi_1^{t}, \ldots, \pi_k^{t})$ yields the proximal gradient method \cite{parikh2014proximal, xu2019scalable},
\begin{equation}
  ( \pi_i^{t+1} )_{i = 1}^k \gets \argmin_{\pi_i \in \Pi(\mu_i, \mu_i'), 1 \leq i \leq k} \sum_{i = 1}^k \langle \nabla_i \Lc(\pi_1^{t}, \ldots, \pi_k^{t}), \pi_i \rangle + \lambda^t \KL(\otimes_i \pi_i | \otimes_i \pi_i^{t}), 
\end{equation}
where $\nabla_i$ denotes the gradient of $\Lc$ in its $i$th argument. Since (for all probability measure inputs) $\KL(\otimes_i \pi_i | \otimes_i \pi_i^{t}) = \sum_{i} \KL(\pi_i | \pi_i^{t})$, the proximal gradient update decouples in each of the $\pi_i$:
\begin{equation}
  \pi_i^{t+1} \gets \argmin_{\pi_i \in \Pi(\mu_i, \mu_i')} \langle \nabla_i \Lc(\pi_1^{t}, \ldots, \pi_k^{t}), \pi_i \rangle + \lambda^t \KL(\pi_i | \pi_i^{t}), \quad 1 \leq i \leq k. 
\end{equation}
Rewriting each problem leads to an entropic optimal transport problem which can be solved via Sinkhorn iterations \cite{cuturi2013sinkhorn}:
\begin{align}
  \begin{split}
    &\pi_i^{t+1} \gets \argmin_{\pi_i \in \Pi(\mu_i, \mu_i')} \lambda^t \KL( \pi_i | \pi_i^{t} \odot e^{-M_i/\lambda^t}), \quad M_i = \nabla_i \Lc(\pi_1^{t}, \ldots, \pi_i^{t}), \quad 1 \leq i \leq k.
  \end{split}\label{eq:proximal_gradient_update}
\end{align}
Noting that 
\[
  \frac{\partial}{\partial \pi} \frac{1}{2} \langle L(\omega, \omega'), \pi \otimes \pi \rangle = \left( \frac{1}{2} L(\omega, \omega') + \frac{1}{2} L(\omega^\top, {\omega'}^\top) \right) \otimes \pi, 
\]
we have the following formula for $\nabla_i \Lc$:
\begin{align*}
  \nabla_i \Lc(\pi_1, \ldots, \pi_k) &= \frac{1}{2} \left( L(\omega_{ii}, \omega'_{ii}) + L(\omega_{ii}^\top, {\omega'}_{ii}^\top) \right) \otimes \pi_i + M[\pi_{-i}] + C_i 
\end{align*}
where $M[\pi_{-i}]$ is defined in \eqref{eq:definition_M}.

\begin{algorithm}
\caption{Proximal gradient: labelled $k$-partitioned networks}
\begin{algorithmic}[1]
\State \textbf{Input:} Matrices $\{ \omega_{ij} \}_{i, j = 1}^k, \{ \omega_{ij}' \}_{i, j = 1}^k$, probability measures $\mu_i, \mu_i'$, $1 \leq i \leq k$, label cost matrices $(C_i)_{i = 1}^k$ (optional).
\State \textbf{Parameters:} inverse step size schedule $\lambda^t, t \ge 0$.
\State Initialize couplings:  $\pi_i^1 \gets \mu_i \otimes \mu_i', \: 1 \leq i \leq k.$
\For{$t = 1, 2, \dots, \texttt{max\_iter}$}
  \For{$1 \leq i \leq k$}
    \State $M_i \gets \frac{1}{2} \left( L(\omega_{ii}, \omega'_{ii}) + L(\omega_{ii}^\top, {\omega'}_{ii}^\top) \right) \otimes \pi_i^t + M[\pi_{-i}^t] + C_i$ \quad (see \eqref{eq:definition_M})
    \State $\pi_i^{t+1} \gets \argmin_{\pi_i \in \Pi(\mu_i, \mu_i')} \lambda^t \KL( \pi_i | \pi_i^{t} \odot e^{-M_i/\lambda^t})$ \quad (solve using Sinkhorn algorithm). 
  \EndFor
\EndFor
\State \textbf{Output:} couplings $\{ \pi_i \}_{i = 1}^k$
\end{algorithmic}
\label{alg:proximal_partitioned}
\end{algorithm}
We remark that when all the $\omega_{ii} = 0$, then $\Lc$ coincides with its linearization and Algorithm \ref{alg:proximal_partitioned} is in fact a \emph{proximal point} method.

\begin{prop}[Convergence of proximal point method]
  The limiting iterates as $t \to \infty$ of the problem \eqref{eq:proximal_point_algorithm} converge to a stationary point of \eqref{eq:discrete_labelled_partitioned_distance}. 
\end{prop}

\begin{proof}
  Writing $\pi = (\pi_1, \ldots, \pi_k)$, \eqref{eq:proximal_point_algorithm} can be written as
  \[
    \min_{\pi \in U} u(\pi, \pi^t), 
  \]
  where
  \[
    u(\pi, \pi') = \Lc(\pi_1, \ldots, \pi_k) + \lambda \KL(\otimes_i \pi_i | \otimes_i \pi_i' )
  \]
  and the set $U = \times_{i = 1}^k \Pi(\mu_i, \mu_i')$ is closed and convex as a Cartesian product of closed, convex sets. Now we note that
  \begin{itemize}
    \item $u(\pi, \pi) = \Lc(\pi)$ for all $\pi \in U$.
    \item $u(\pi, \pi') \geq \Lc(\pi)$ by non-negativity of the $\KL$-divergence.
    \item $u(\pi, \pi')$ is smooth in both of its arguments. 
  \end{itemize}
  Taken together and applying \cite[Proposition 1]{razaviyayn2013unified}, we satisfy the conditions for \cite[Theorem 1]{razaviyayn2013unified}.

  We remark that the unbalanced case (where the hard marginal constraints are replaced by soft constraints) can be handled in the same way, if the marginal penalties are smooth. The constraint set is then $U = \times_{i = 1}^k \mathcal{M}_+$ which is also closed and convex. 
\end{proof}

\subsection{Projected gradient descent}\label{sec:projected_gradient_descent}

When solving the regularized problem \eqref{eq:entropic_discrete_labelled_partitioned_distance} and setting $\Omega_0 = \Omega_1 = \KL$, a projected gradient descent approach \cite{peyre2016gromov} can be used. Then the minimization problem has the form 
\begin{equation}
  \min_{\pi_i \in \Pi(\mu_i, \mu_i'), 1 \leq i \leq k} \Lc(\pi_1, \ldots, \pi_k) + \sum_{i = 1}^k \varepsilon_i \KL(\pi_i | \mu_i \otimes \mu_i'). 
\end{equation}
For a gradient step size $\eta_i > 0$, a projected mirror descent step in each $\pi_i$ reads
\begin{equation}
  \pi_i^{t+1} \gets \Proj_{\Pi(\mu_i, \mu_i')}^{\KL}\left[ \pi_i^t \odot \exp\left(-\eta_i (\nabla_i \Lc(\pi_1^t, \ldots, \pi_k^t) + \varepsilon_i \log(\pi_i^t / \mu_i \otimes \mu_i'))\right) \right], \quad 1 \leq i \leq k.
\end{equation}
Choosing $\eta_i = 1/\varepsilon_i$, we get
\begin{equation}
  \pi_i^{t+1} \gets \Proj_{\Pi(\mu_i, \mu_i')}^{\KL}\left[ e^{-\varepsilon_i^{-1} \nabla_i \Lc(\pi_1, \ldots, \pi_k)} \mu_i \otimes \mu_i' \right], \quad 1 \leq i \leq k.
\end{equation}
Each of these projections can be computed through the Sinkhorn algorithm \cite{peyre2016gromov, cuturi2013sinkhorn}. Similarly, one may consider $\Omega_i = \frac{1}{2}\| \cdot \|_{L^2(\mu_i \otimes \mu'_i)}^2$, in which case the projected $L^2$-gradient descent steps are
\begin{equation}
  \pi^{t+1}_i \gets \Proj_{\Pi(\mu_i, \mu_i')}^{F}\left[ -\frac{1}{\varepsilon_i} \nabla_i \Lc(\pi_1^t, \ldots, \pi_k^t ) (\mu_i \otimes \mu_i') \right], \quad 1 \leq i \leq k. 
\end{equation}
Here, both the gradient and projection steps are calculated in $L^2(\mu \otimes \mu')$. The $L^2$-projection $\Proj_{\Pi(\mu, \mu')}^F(A)$ can be carried out by solving a quadratically regularized optimal transport problem. This formulation has the notable advantage of producing couplings which are sparse, i.e. identically zero outside of a support set \cite{zhang2023manifold, lorenz2021quadratically}.

\begin{algorithm}
\caption{Projected gradient: labelled $k$-partitioned networks, regularized matchings}
\begin{algorithmic}[1]
\State \textbf{Input:} Matrices $\{ \omega_{ij} \}_{i, j = 1}^k, \{ \omega_{ij}' \}_{i, j = 1}^k$, probability measures $\mu_i, \mu_i'$, $1 \leq i \leq k$, label cost matrices $(C_i)_{i = 1}^k$ (optional). 
\State \textbf{Parameters:} entropic regularization levels $\varepsilon_i \ge 0, 1 \leq i \leq k$ (optional)
\State Initialize couplings:  $\pi_i^1 \gets \mu_i \otimes \mu_i', \: 1 \leq i \leq k.$
\For{$t = 1, 2, \dots, \texttt{max\_iter}$}
  \For{$1 \leq i \leq k$}
    \State $M_i \gets \frac{1}{2} \left( L(\omega_{ii}, \omega'_{ii}) + L(\omega_{ii}^\top, {\omega'}_{ii}^\top) \right) \otimes \pi_i^t + M[\pi_{-i}^t] + C_i$ \quad (see \eqref{eq:definition_M})
    \State $\pi_i^{t+1} \gets \argmin_{\pi_i \in \Pi(\mu_i, \mu_i')} \varepsilon_i \KL(\pi_i | e^{-\varepsilon_i^{-1} M_i } \mu_i \otimes \mu_i')$
  \EndFor
\EndFor
\State \textbf{Output:} couplings $\{ \pi_i \}_{i = 1}^k$
\end{algorithmic}
\label{alg:projected_partitioned}
\end{algorithm}

\subsection{Unbalanced matchings}\label{sec:unbalanced_matchings}

We now consider the setting of unbalanced transport, in which marginal constraints are relaxed and replaced with penalty functions that enforce a ``soft'' marginal constraint. Unbalanced transport has been well studied from both a theoretical and practical viewpoint for the transportation of measures \cite{chizat2018unbalanced, chizat2018scaling, liero2018optimal}, and has since been extended to the setting of (Fused) Gromov-Wasserstein matchings between metric measure spaces \cite{sejourne2021unbalanced, thual2022aligning} and co-optimal transport \cite{tran2023unbalanced}. An unbalanced formulation of the partitioned network alignment problem is valuable in practical settings when there may only be partial correspondences between networks, such as in the metabolic network alignment example of Figure \ref{fig:metabolic_alignment}. 

The unbalanced transport problem for labelled partitioned measure networks includes unbalanced (Fused) Gromov-Wasserstein \cite{sejourne2021unbalanced, thual2022aligning} and co-optimal transport \cite{tran2023unbalanced} as sub-cases. We let $\lambda_{1, 2} > 0$ enforce the soft marginal constraints for the source and target respectively, and we denote by $\mathcal{M}_+(X)$ the space of positive measures supported on $X$. For generality, and because this makes efficient computational schemes possible, we optionally allow an entropic regularization with a coefficient $\varepsilon \ge 0$. Then we pose the problem of (entropically regularized) \emph{unbalanced} matching as 
\begin{align}
  \begin{split}
    \min_{\substack{\pi_i \in \Mc_+(X_i \times X_i'), \: 1 \leq i \leq k, \\ m(\pi_i) = m(\pi_j), \: \forall 1 \leq i, j \leq k}} \: \widetilde{\Lc}(\pi_1, \ldots, \pi_k) &+ \varepsilon \sum_{i,j = 1}^k \KL(\pi_i \otimes \pi_j | \mu_i \otimes \mu_i' \otimes \mu_j \otimes \mu_j')  \\
                                                                                        &+ \lambda_1 \sum_{i, j = 1}^k \KL(\pi_i \ones \otimes \pi_j \ones | \mu_i \otimes \mu_j) \\
                                                                                        &+ \lambda_2 \sum_{i, j = 1}^k \KL(\pi_i^\top \ones \otimes \pi_j^\top \ones | \mu_i' \otimes \mu_j').
  \end{split}\label{eq:unbalanced_partitioned_measure_network_matching}
\end{align}
In the above, $\widetilde{\Lc}$ denotes a variant of the objective function \eqref{eq:discrete_labelled_partitioned_distance_obj}, modified to ensure that the function $\widetilde{\Lc}$ and overall objective to be minimized remains homogeneous in $(\pi_1, \ldots, \pi_k)$:
\begin{equation}
  \widetilde{\Lc}(\pi_1, \ldots, \pi_k) = \frac{1}{2} \sum_{i, j = 1}^k \langle L(\omega_{ij}, \omega'_{ij}), \pi_i \otimes \pi_j \rangle + \sum_{i = 1}^k m(\pi_i) \langle C_i, \pi_i \rangle,
  \label{eq:unbalanced_partitioned_measure_network_matching_obj}
\end{equation}
where $m(\pi) = \int \diff \pi$ is the total mass of $\pi$. 
This is different to the setup in \cite{thual2022aligning}, in which the quadratic nature of the Gromov-Wasserstein term in the coupling $\pi$ conflicts with the linearity of the fused term. Importantly, $\widetilde{\Lc}$ coincides with $\Lc$ when its inputs are restricted to be probability measures, i.e. $m(\pi_i) = 1$. 

Extending the definition of partitioned measure networks (Definition \ref{defn:partitioned_measure_network}), we will allow $\mu_i$ to be positive measures in $\mathcal{M}_+(X_i)$ for each $1 \leq i \leq k$, but require that $m(\mu_i) = m(\mu_j), i \ne j$. The motivation is to eliminate the non-uniqueness under scaling (e.g.,~$(\mu_i, \mu_j) \to (\lambda \mu_i, \lambda^{-1} \mu_j)$) that becomes a particular issue in the special case of co-optimal transport \cite{tran2023unbalanced}.

First, we state the following identity for the KL-divergence (see \eqref{eq:kullback_leibler}), 
\begin{align}
  \begin{split}
    \KL(\alpha \otimes \beta | \alpha' \otimes \beta') &= m(\beta) \Ent(\alpha | \alpha') + m(\alpha) \Ent(\beta | \beta') - m(\alpha) m(\beta) + m(\alpha') m(\beta') \\
                                                       &= m(\beta) \KL(\alpha | \alpha') + m(\alpha) \KL(\beta | \beta') + (m(\alpha) - m(\alpha')) (m(\beta) - m(\beta')),
  \end{split}
  \label{eq:kl_product_identity}
\end{align}
where we have defined the \emph{relative entropy} term 
\begin{equation}
  \Ent(\alpha | \alpha') = \int  \log \left( \frac{\diff\alpha}{\diff\alpha'}  \right) \diff\alpha.
  \label{eq:relative_entropy}
\end{equation}
Thus, $(\alpha, \beta) \mapsto \KL(\alpha \otimes \beta | \alpha' \otimes \beta')$ is 2-homogeneous up to additive constants \cite{sejourne2021unbalanced}. Since $\widetilde{\Lc}$ is 2-homogeneous, the objective of \eqref{eq:unbalanced_partitioned_measure_network_matching} is also 2-homogeneous. We remark that this is important, since if we used $\Lc$ instead of $\widetilde{\Lc}$, under the scaling $\pi_i \mapsto \lambda \pi_i$, quadratic terms would dominate when $\lambda \to +\infty$ and linear terms when $\lambda \to 0$.

\smallskip\noindent\textbf{Special case: measure hypernetworks.}~A special case is when $k = 2$ and $\omega_{ii} = \omega'_{ii} = 0$: this amounts to an \emph{unbalanced, fused} co-optimal transport problem. This was the focus of \cite{tran2023unbalanced} which considered the unlabelled setting, and we discuss it for completeness.
In this case, $\widetilde{\Lc}$ can be written in a bilinear form in $(\pi_1, \pi_2)$:
\begin{equation}
  \widetilde{\Lc}(\pi_1, \pi_2) = \left\langle \frac{1}{2}L(\omega_{12}, \omega_{12}') + \frac{1}{2} L(\omega_{21}^\top, \omega_{21}'^\top) + C_1 \otimes \ones + \ones \otimes C_2, \pi_1 \otimes \pi_2 \right\rangle . 
\end{equation}
Up to additive constants, the problem \eqref{eq:unbalanced_partitioned_measure_network_matching} for $k = 2$ and $\omega_{ii} = \omega'_{ii} = 0$ can be re-written as 
\begin{align}
  \begin{split}
  \min_{\pi_1, \pi_2} \widetilde{\Lc}(\pi_1, \pi_2) &+ 2\lambda_1 \left( (m(\pi_1) + m(\pi_2)) \Ent(\pi_1 \ones | \mu_1) + (m(\pi_1) + m(\pi_2)) \Ent(\pi_2 \ones | \mu_2)  \right) \\
                                                    &+ 2\lambda_2 \left( (m(\pi_1) + m(\pi_2)) \Ent(\pi_1^\top \ones | \mu_1') + (m(\pi_1) + m(\pi_2)) \Ent(\pi_2^\top \ones | \mu_2') \right) \\
                                                    &+ 2\varepsilon \left( (m(\pi_1) + m(\pi_2)) \Ent(\pi_1 | \mu_1 \otimes \mu_1') + (m(\pi_1) + m(\pi_2)) \Ent(\pi_2 | \mu_2 \otimes \mu_2')) \right) \\ 
                                                    & - (\lambda_1 + \lambda_2 + \varepsilon) \left( m(\pi_1)^2 + m(\pi_2)^2 + 2 m(\pi_1) m(\pi_2) \right),
  \end{split}
\end{align}
where the minimum is taken over $\pi_1 \in \mathcal{M}_+(X_1 \times X_1')$ and $\pi_2 \in \mathcal{M}_+(X_2 \times X_2')$ such that $m(\pi_1) = m(\pi_2)$.
Notice that $\widetilde{\Lc}$ is bilinear, but the terms corresponding to the soft marginal constraints contain quadratic terms in $\pi_1$ and $\pi_2$. However, since we have the constraint $m(\pi_1) = m(\pi_2)$, we could judiciously swap $m(\pi_1)$ and $m(\pi_2)$ to derive an objective that remains equivalent under the mass equality constraint:
\begin{align}
  \begin{split}
  \min_{\pi_1, \pi_2} \widetilde{\Lc}(\pi_1, \pi_2) & + 4\lambda_1 \left( m(\pi_2) \Ent(\pi_1 \ones | \mu_1) + m(\pi_1) \Ent(\pi_2 \ones | \mu_2)  \right) \\
                                                                                                                      &+ 4\lambda_2 \left( m(\pi_2) \Ent(\pi_1^\top \ones | \mu_1') + m(\pi_1) \Ent(\pi_2^\top \ones | \mu_2') \right) \\
                                                                                                                      &+ 4\varepsilon \left( m(\pi_2) \Ent(\pi_1 | \mu_1 \otimes \mu_1') + m(\pi_1) \Ent(\pi_2 | \mu_2 \otimes \mu_2') \right) \\ 
                                                    & - 4(\lambda_1 + \lambda_2 + \varepsilon) m(\pi_1) m(\pi_2).
  \end{split}
\end{align}
In what follows, we drop the factor of 4 as it can be absorbed into the coefficients:
\begin{align}
  \begin{split}
  \min_{\pi_1, \pi_2}  \widetilde{\Lc}(\pi_1, \pi_2) &+ \lambda_1 \left( m(\pi_2) \Ent(\pi_1 \ones | \mu_1) + m(\pi_1) \Ent(\pi_2 \ones | \mu_2)  \right) \\
                                                                                                                      &+ \lambda_2 \left( m(\pi_2) \Ent(\pi_1^\top \ones | \mu_1') + m(\pi_1) \Ent(\pi_2^\top \ones | \mu_2') \right) \\
                                                                                                                      &+ \varepsilon \left( m(\pi_2) \Ent(\pi_1 | \mu_1 \otimes \mu_1') + m(\pi_1) \Ent(\pi_2 | \mu_2 \otimes \mu_2') \right) \\ 
                                                     & - (\lambda_1 + \lambda_2 + \varepsilon) m(\pi_1) m(\pi_2).
  \end{split}
\end{align}
We will for now remove the constraint $m(\pi_1) = m(\pi_2)$. In this setting, we can tackle the problem by alternating block minimization in $(\pi_1, \pi_2)$. Fixing $\pi_2$ and rearranging the objective function, the update in $\pi_1$ amounts to
\begin{align}
  \min_{\pi_1 \in \mathcal{M}_+(X_1 \times X_1')} &\langle M[\pi_2], \pi_1 \rangle \\
  &+ \lambda_1 m(\pi_2) \KL(\pi_1 \ones | \mu_1) + \lambda_2 m(\pi_2) \KL(\pi_1^\top \ones | \mu_1') + \varepsilon m(\pi_2) \KL(\pi_1 | \mu_1 \otimes \mu_1'), 
  \label{eq:unbalanced_coot_update_1}
\end{align}
where
\begin{align}
\begin{split}
  M[\pi_2] = &\left( \frac{1}{2} L(\omega_{12}, \omega_{12}') + \frac{1}{2} L(\omega_{21}^\top, {\omega'}_{21}^\top) \right) \otimes \pi_2 \\
  &+ m(\pi_2) C_1 + \left( \langle C_2, \pi_2 \rangle + \lambda_1 \Ent(\pi_2 \ones | \mu_2) + \lambda_2 \Ent(\pi_2^\top \ones | \mu_2') + \varepsilon \Ent(\pi_2 | \mu_2 \otimes \mu_2')\right) \ones . 
  \end{split}   
\end{align}
Fixing $\pi_1$, the update in $\pi_2$ is
\begin{align}
  \min_{\pi_2 \in \mathcal{M}_+(X_2 \times X_2')} &\langle M[\pi_1], \pi_2 \rangle \\
  &+ \lambda_1 m(\pi_1) \KL(\pi_2 \ones | \mu_2) + \lambda_2 m(\pi_1) \KL(\pi_2^\top \ones | \mu_2') + \varepsilon m(\pi_1) \KL(\pi_2 | \mu_2 \otimes \mu_2'), 
  \label{eq:unbalanced_coot_update_2}
\end{align}
where
\begin{align}
\begin{split}
  M[\pi_1] = &\left( \frac{1}{2} L(\omega_{12}^\top, {\omega'}_{12}^\top) + \frac{1}{2} L(\omega_{21}, \omega'_{21}) \right) \otimes \pi_1 \\
  &+ m(\pi_1) C_2 + \left( \langle C_1, \pi_1 \rangle + \lambda_1 \Ent(\pi_1 \ones | \mu_1) + \lambda_2 \Ent(\pi_1^\top \ones | \mu_1') + \varepsilon \Ent(\pi_1 | \mu_1 \otimes \mu_1' )\right)\ones.
\end{split}
\end{align}

To enforce the equal mass constraint $m(\pi_1) = m(\pi_2)$, given $(\pi_1, \pi_2)$ with $m(\pi_1) \ne m(\pi_2)$, we use the following to project onto the constraint set:
\begin{equation}
  (\pi_1, \pi_2) \mapsto \left( \sqrt{\frac{m(\pi_2)}{m(\pi_1)}} \pi_1, \sqrt{\frac{m(\pi_1)}{m(\pi_2)}} \pi_2 \right). 
\end{equation}
This projection, also used in \cite{tran2023unbalanced}, can be shown to be equivalent to the $\KL$-projection of $(\pi_1, \pi_2)$ onto the set $m(\pi_1) = m(\pi_2)$. 

\begin{algorithm}
\caption{Unbalanced matchings: labelled measure hypernetworks}
\begin{algorithmic}[1]
  \State \textbf{Input:} Matrices $\omega_{12}, \omega_{12}'$, positive measures $\mu_i, \mu_i'$, $i = 1, 2$, label cost matrices $C_{1, 2}$ (optional).
  \State \textbf{Parameters:} Entropic regularisation parameter $\varepsilon \geq 0$, unbalanced parameters $\lambda_1, \lambda_2 > 0$.
\State Initialize couplings:  $\pi_i \gets \mu_i \otimes \mu_i' / \sqrt{m(\mu_i) m(\mu_i')}, \: i = 1, 2.$
\For{$t = 1, 2, \dots, \texttt{max\_iter}$}
  \State \begin{align*}
    \pi_1 \gets \argmin_{\pi_1 \in \mathcal{M}_+(X_1 \times X_1')} &\langle M[\pi_2], \pi_1 \rangle \\
    & + \lambda_1 m(\pi_2) \KL(\pi_1 \ones | \mu_1) + \lambda_2 m(\pi_2) \KL(\pi_1^\top \ones | \mu_1') + \varepsilon m(\pi_2) \KL(\pi_1 | \mu_1 \otimes \mu_1')
  \end{align*}
  \State $\pi_1 \gets \sqrt{\dfrac{m(\pi_2)}{m(\pi_1)}} \pi_1$
  \State \begin{align*}
    \pi_2 \gets \argmin_{\pi_2 \in \mathcal{M}_+(X_2 \times X_2')} &\langle M[\pi_1], \pi_2 \rangle \\
    &+ \lambda_1 m(\pi_1) \KL(\pi_2 \ones | \mu_2) + \lambda_2 m(\pi_1) \KL(\pi_2^\top \ones | \mu_2') + \varepsilon m(\pi_1) \KL(\pi_2 | \mu_2 \otimes \mu_2')
  \end{align*}
  \State $\pi_2 \gets \sqrt{\dfrac{m(\pi_1)}{m(\pi_2)}} \pi_2$
\EndFor
\State \textbf{Output:} couplings $\{ \pi_i \}_{i = 1}^k$
\end{algorithmic}
\label{alg:unbalanced_hypernetwork}
\end{algorithm}

\smallskip\noindent\textbf{General case.}~For general partitioned measure networks, the objective function $\widetilde{\Lc}$ introduces terms that are non-trivially quadratic in $\pi_i$ and is therefore less straightforward to solve. While in the balanced case these kinds of problems are typically tackled using a Frank-Wolfe algorithm \cite{memoli2011gromov}, such an approach is not feasible for problems with soft constraints. As done in related works \cite{thual2022aligning, sejourne2021unbalanced}, we propose to solve the problem instead via a biconvex relaxation. Consider \emph{two} partitioned couplings $(\pi_1, \ldots, \pi_k)$ and $(\xi_1, \ldots, \xi_k)$. We us the relaxation of \eqref{eq:unbalanced_partitioned_measure_network_matching}:
\begin{align}
  \begin{split}
  \min_{\substack{\pi_i, \xi_i \in \Mc_+(X_i \times X_i'), 1 \leq i \leq k \\ m(\pi_i) = m(\xi_j), 1 \leq i, j \leq k}} \: &\widetilde{\Lc}(\pi_1, \ldots, \pi_k; \xi_1, \ldots, \xi_k) + \varepsilon \sum_{i,j = 1}^k \KL(\pi_i \otimes \xi_j | \mu_i \otimes \mu_i' \otimes \mu_j \otimes \mu_j')  \\
                                                                                                                           &+ \lambda_1 \sum_{i, j = 1}^k \KL(\pi_i \ones \otimes \xi_j \ones | \mu_i \otimes \mu_j) + \lambda_2 \sum_{i, j = 1}^k \KL(\pi_i^\top \ones \otimes \xi_j^\top \ones | \mu_i' \otimes \mu_j'),
  \end{split}
  \label{eq:unbalanced_biconvex_relaxation}
\end{align}
where we define the relaxed version of \eqref{eq:unbalanced_partitioned_measure_network_matching_obj}:
\begin{align*}
  \widetilde{\Lc}(\pi_1, \ldots, \pi_k; \xi_1, \ldots, \xi_k) &= \frac{1}{2} \sum_{i, j = 1}^k \langle L(\omega_{ij}, \omega_{ij}'), \pi_i \otimes \xi_j \rangle + \sum_{i = 1}^k \left\langle \frac{1}{2}(C_i \otimes \ones + \ones \otimes C_i), \pi_i \otimes \xi_i \right\rangle . 
\end{align*}
The form of the second term ensures symmetry under the exchange of $(\pi, \xi)$ and that $\widetilde{\Lc}(\pi_1, \ldots, \pi_k; \pi_1, \ldots, \pi_k) = \Lc(\pi_1, \ldots, \pi_k)$. 
The problem \eqref{eq:unbalanced_biconvex_relaxation} is now convex separately in $(\pi_i)_{i = 1}^k$ and $(\xi_i)_{i = 1}^k$ respectively. Fixing $(\xi_i)_{i = 1}^k$ and minimizing in $(\pi_i)_{i = 1}^k$, we find that the problem decouples across partitions in each of the $\pi_i$:
\begin{align}
  \begin{split}
  \min_{\substack{\pi_i \in \Mc_+(X_i \times X_i'), 1 \leq i \leq k \\ m(\pi_i) = m(\pi_j), 1 \leq i, j \leq k}} &\widetilde{\Lc}(\pi_1, \ldots, \pi_k ; \xi_1, \ldots, \xi_k) \\
  &+ \left( \sum_i m(\pi_i) \right) \left( \varepsilon \sum_{j} \Ent(\xi_j | \mu_j \otimes \mu_j') + \lambda_1 \sum_j \Ent(\xi_j \ones | \mu_j ) + \lambda_2 \sum_j \Ent(\xi_j^\top \ones | \mu_j') \right) \\
  &+ \lambda_1 \left( \sum_j m(\xi_j) \right) \sum_i \KL(\pi_i \ones | \mu_i) + \lambda_2 \left( \sum_j m(\xi_j) \right) \sum_i \KL(\pi_i^\top \ones | \mu_i') \\ 
                                                                                                                 &+ \varepsilon \left( \sum_j m(\xi_j) \right) \sum_i \KL(\pi_i | \mu_i \otimes \mu_i').
  \end{split}
  \label{eq:unbalanced_biconvex_pi}
\end{align}
Relaxing the mass equality constraint $m(\pi_i) = m(\pi_j), 1 \leq i, j \leq k$, the above problem amounts to $k$ regularized unbalanced optimal transport problems that can be solved independently and in parallel. The resulting couplings can be projected onto the set $\{ m(\pi_i) = m(\pi_j), 1 \leq i, j \leq k \}$:
\begin{equation}
  (\pi_i)_{i = 1}^k \mapsto \left( \frac{(m(\pi_1) \ldots m(\pi_k))^{1/k}}{m(\pi_i)} \pi_i \right)_{i = 1}^k.
  \label{eq:mass_projection}
\end{equation}
Similarly, fixing $(\pi_i)_{i = 1}^k$, the problem in $(\xi_i)_{i = 1}^k$ is 
\begin{align}
  \begin{split}
  \min_{\substack{\xi_i \in \Mc_+(X_i \times X_i'), 1 \leq i \leq k \\ m(\xi_i) = m(\xi_j), 1 \leq i, j \leq k}} &\widetilde{\Lc}(\pi_1, \ldots, \pi_k ; \xi_1, \ldots, \xi_k) \\
  &+ \left( \sum_j m(\xi_j) \right) \left( \varepsilon \sum_{i} \Ent(\pi_i | \mu_i \otimes \mu_i') + \lambda_1 \sum_i \Ent(\pi_i \ones | \mu_i ) + \lambda_2 \sum_i \Ent(\pi_i^\top \ones | \mu_i') \right) \\
                                                                                                                                                                &+ \lambda_1 \left( \sum_i m(\pi_i) \right) \sum_j \KL(\xi_j \ones | \mu_j) + \lambda_2 \left( \sum_i m(\pi_i) \right) \sum_j \KL(\xi_j^\top \ones | \mu_j')  \\ 
                                                                                                                 &+ \varepsilon \left( \sum_i m(\pi_i) \right) \sum_j \KL(\xi_j | \mu_j \otimes \mu_j'),
  \end{split}
  \label{eq:unbalanced_biconvex_xi}
\end{align}
and the same projection \eqref{eq:mass_projection} can be used to enforce the mass equality constraint in $(\xi_i)_{i = 1}^k$.
It is important to note that this scheme aims to solve the biconvex relaxation \eqref{eq:unbalanced_biconvex_relaxation} which is in general only a lower bound for \eqref{eq:unbalanced_partitioned_measure_network_matching}. In particular, at convergence, we may have $\pi_i \ne \xi_i$ in general. While this biconvex relaxation scheme was studied for the Gromov-Wasserstein setting by \cite{sejourne2021unbalanced}, they were unable to prove tightness or that the two sets of couplings $(\pi_i)_i, (\xi_i)_i$ coincide.

\begin{algorithm}
\caption{Unbalanced matchings: labelled partitioned measure networks via biconvex relaxation}
\begin{algorithmic}[1]
  \State \textbf{Input:} Matrices $\{ \omega_{ij} \}_{i, j = 1}^k, \{ \omega_{ij}' \}_{i, j = 1}^k$, positive measures $\mu_i, \mu_i'$, $1 \leq i \leq k$, label cost matrices $(C_i)_{i = 1}^k$ (optional)
  \State \textbf{Parameters:} Marginal penalties $\lambda_1, \lambda_2 > 0$, entropic regularization $\varepsilon \ge 0$ (optional).
\State Initialize couplings:  $\pi_i \gets \mu_i \otimes \mu_i' / \sqrt{m(\mu_i) m(\mu_i')}, \: 1 \leq i \leq k.$
\State Initialize additional couplings:  $\xi_i \gets \pi_i, \: 1 \leq i \leq k.$
\For{$t = 1, 2, \dots, \texttt{max\_iter}$}
  \State Update $(\pi_1, \ldots, \pi_k)$ by solving \eqref{eq:unbalanced_biconvex_pi} independently for each $1 \leq i \leq k$. 
  \State Rescale $(\pi_i)_{i = 1}^k$ following \eqref{eq:mass_projection}
  \State Update $(\xi_1, \ldots, \xi_k)$ by solving \eqref{eq:unbalanced_biconvex_xi} independently for each $1 \leq i \leq k$. 
  \State Rescale $(\xi_i)_{i = 1}^k$ following \eqref{eq:mass_projection}
  \State $(\pi_i)_{i = 1}^k, (\xi_i)_{i = 1}^k \gets \left( \sqrt{\frac{m(\xi)}{m(\pi)}} \pi_i \right)_{i = 1}^k, \left( \sqrt{\frac{m(\pi)}{m(\xi)}} \xi_i \right)_{i = 1}^k$
\EndFor
\State \textbf{Output:} couplings $\{ \pi_i \}_{i = 1}^k, \{ \xi_i \}_{i = 1}^k$
\end{algorithmic}
\label{alg:unbalanced_partitioned}
\end{algorithm}

\smallskip\noindent\textbf{A remark on partial transport.} On the other hand we may consider \emph{partial} transport, where some fraction $0 \leq s \leq 1$ of mass is required to be transported with the remainder being discarded and thus incurring zero cost. This problem was considered by \cite{chapel2020partial}  in the case of Gromov-Wasserstein transport. For two probability measures $\mu, \mu'$, define the set of partial couplings of mass $s$ to be
\[
  \Pi(\mu, \mu'; s) = \{ \pi \ge 0 : \pi \ones \leq \mu, \pi^\top \ones \leq \mu', m(\pi) = s\}. 
\]
Then, the partial matching problem amounts to solving
\begin{align}
  \begin{split}
    \min_{\pi_i \in \Pi(\mu_i, \mu_i'; s)} \Lc(\pi_1, \ldots, \pi_k). 
    \end{split}
\end{align}
This amounts to the minimization of a non-convex objective on a convex and compact constraint set, and similar to \cite{chapel2020partial}, we can tackle it via conditional gradient method. In particular, to compute the descent directions in each of the $\pi_i$, we need to solve:
\begin{align*}
  &\min_{\pi_i \in \Pi(\mu_i, \mu_i'; s)} \langle \nabla_i \Lc(\pi_1^t, \ldots, \pi_k^t), \pi_i \rangle.
\end{align*}
Each of these is a partial optimal transport problem which can be solved using the virtual point approach of \cite{chapel2020partial}. 

\subsection{Partitioned networks for multiscale network matching}\label{sec:multiscale_matching_algorithm}

Chowdhury et al. \cite{chowdhury2023hypergraph} introduced a generalized co-optimal transport problem for multiscale network matching. Given an input graph $G$, they produced successive topological simplifications $\mathcal{G} = \{ G = G_1, \ldots, G_k \}$. At each level $1 \leq i \leq k-1$, the nodes of $G_i$ are partitioned among the nodes of $G_{i+1}$. In this way, the coupling of $G_i$ to $G_{i+1}$ can be modelled as a hypergraph in which nodes and hyperedges are identified with nodes in $G_i$ and $G_{i+1}$ respectively. 
We now show that $\mathcal{G}$ can be formulated as a partitioned network with $k$ partitions. Let $X_i$ be the node set of the $i$th simplification level $G_i$.
Let $\omega_{i,i+1}$ (for $1 \leq i \leq k-1$) be the function encoding relations between nodes in the $i$th and nodes the $(i+1)$th simplification:
\[
  \omega(x, y) = \begin{cases}
    \omega_{i, i+1}(x, y), \quad &x \in X_i, y \in X_{i+1} \text{ for } i = 1, \ldots, k-1; \\
    \omega_{i+1, i}(x, y), \quad &x \in X_{i+1}, y \in X_{i} \text{ for } i = 1, \ldots, k-1; \\
    0, \quad &\text{otherwise}. 
  \end{cases}
\]
Together with a choice of weights $(\mu_i)_{i = 1}^k$, $((X_i, \mu_i)_{i = 1}^k, \omega)$ is a partitioned measure network encoding the multiscale network $\mathcal{G}$. 

Given two graphs $G$ and $G'$ and their respective simplifications $\mathcal{G}$ and $ \mathcal{G}'$, we can then construct two partitioned measure networks: $((X_i, \mu_i)_{i = 1}^k, \omega)$ and $((X_i', \mu_i')_{i = 1}^k, \omega')$. For a candidate coupling $(\pi_i)_{i = 1}^k$, the corresponding distortion functional is 
\begin{equation}
  \sum_{i=1}^{k-1} \| \omega_{i,i+1} - \omega_{i,i+1}' \|_{L^p(\pi_i \otimes \pi_{i+1})}^p. 
\end{equation}
The partitioned measure network alignment problem induced by this distortion is equivalent to the one proposed in \cite[Algorithm 1]{chowdhury2023hypergraph}, i.e.
\begin{align}
  \min_{\pi_i \in \Pi(\mu_i, \mu_i'), 1 \leq i \leq k} \sum_{i = 1}^{k-1} \langle L(\omega_{i,i+1}, \omega_{i,i+1}'), \pi_i \otimes \pi_{i+1} \rangle. 
\end{align}
This multiscale graph matching problem therefore fits into the problem of matchings of partitioned measure networks. Encouragingly, while the formulation of \cite{chowdhury2023hypergraph} was in terms of pairs of couplings $(\pi_i, \xi_i), 1 \leq i \leq k-1$ under the constraint $\xi_i = \pi_{i+1}$, the derivation of the problem from the viewpoint of partitioned measure networks allows us to directly and naturally formulate the problem in terms of a single set of couplings $(\pi_1, \ldots, \pi_k)$.

Furthermore, by modifying the function $\omega$, we can incorporate pairwise information on each of the graphs $G_i$:
\[
  \omega(x, y) = \begin{cases}
    \omega_{i, i+1}(x, y), \quad &x \in X_i, y \in X_{i+1} \text{ for } i = 1, \ldots, k-1; \\
    \omega_{i+1, i}(x, y), \quad &x \in X_{i+1}, y \in X_{i} \text{ for } i = 1, \ldots, k-1; \\
    \omega_{ii}(x, y), \quad &x, y \in X_i \times X_i \text{ for } i = 1, \ldots, k; \\ 
    0, \quad &\text{otherwise}. 
  \end{cases}
\]
This choice of $\omega$ leads to the problem
\begin{equation}
  \min_{\pi_i \in \Pi(\mu_i, \mu_i')} \: \sum_{i = 1}^{k-1} \langle L(\omega_{i,i+1}, \omega_{i,i+1}'), \pi_i \otimes \pi_{i+1} \rangle + \frac{1}{2} \sum_{i = 1}^k \langle L(\omega_{ii}, \omega_{ii}'), \pi_i \otimes \pi_i \rangle + \sum_{i = 1}^k \varepsilon_i \KL(\pi_i | \mu_i \otimes \mu_i'), 
  \label{eq:multiscale_alignment_problem}
\end{equation}
which incorporates Gromov-Wasserstein like (i.e. quadratic in $\pi$) terms. In the above we allow optionally for entropy regularization, $\varepsilon_i \ge 0$. 
For $\varepsilon_i > 0$, applying the projected gradient descent approach of Section \ref{sec:projected_gradient_descent} leads to the update rule
\begin{align}
  \begin{split}
  &\pi_{i}^{t+1} \gets \Proj_{\Pi(\mu_i, \mu_i')}^{\KL} \left( e^{-\varepsilon_i^{-1} \nabla_i \Lc(\pi_1^t, \ldots, \pi_k^t)} \mu_i \otimes \mu_i' \right), \\
    &\nabla_i \Lc(\pi_1, \ldots, \pi_k) = \\
    & \begin{cases}
    L(\omega_{12}, \omega_{12}') \otimes \pi_2 + \frac{1}{2} \left( L(\omega_{11}, \omega_{11}') + L(\omega_{11}^\top, \omega_{11}'^\top) \right) \otimes \pi_1, \quad &i = 1; \\
    \begin{aligned}
        L(\omega_{i-1, i}^\top, \omega_{i-1,i}') \otimes \pi_{i-1} + L(\omega_{i,i+1}, \omega_{i,i+1}') \otimes \pi_{i+1} \\
        + \frac{1}{2} \left( L(\omega_{ii}, \omega_{ii}') + L(\omega_{ii}^\top, \omega_{ii}'^\top) \right) \otimes \pi_i 
    \end{aligned}, \quad &2 \leq i \leq k-1; \\
    L(\omega_{k-1, k}^\top, \omega_{k-1, k}'^\top) \otimes \pi_{k-1} + \frac{1}{2} \left( L(\omega_{kk}, \omega_{kk}') + L(\omega_{kk}^\top, \omega_{kk}'^\top) \right) \otimes \pi_k, \quad &i = k.
  \end{cases}
  \end{split}
\end{align}
When we look for an unregularized solution and $\varepsilon_i = 0$, a block coordinate descent scheme similar to the one proposed in \cite{chowdhury2023hypergraph} can be employed. The block update in each of the $\pi_i$ works out to be a Fused Gromov-Wasserstein problem which can be tackled for instance using the Frank-Wolfe scheme of \cite{vayer2020fused}.
Alternatively, a proximal gradient approach similar to the one described in Section \ref{sec:proximal_gradient_algorithms} can be employed, in which case the gradient steps are the same as in \eqref{eq:proximal_gradient_update}.

An unbalanced formulation of this problem can also be solved by using the same biconvex relaxation approach laid out in Section \ref{sec:unbalanced_matchings}. 
While this problem falls into the scope of Algorithm \ref{alg:unbalanced_partitioned}, it is in fact a sub-case since each partition $i$ is only coupled to its ``adjoining'' partitions (rather than all partitions in the general case). We detail below the specific updates for the biconvex relaxation, in terms of two sets of couplings, $(\pi_i)_{i = 1}^k, (\xi_i)_{i = 1}^k$. For each of the $\pi_i, 1 \leq i \leq k$, solve
\begin{align}
  \begin{split}
    \min_{\pi_i \in \Mc_+(X_i \times X_i')} \langle M_i, \pi_i \rangle &+ \left(\varepsilon \sum_j m(\xi_j)\right) \KL(\pi_i | \mu_i \otimes \mu_i') \\ &+ \left(\lambda_1 \sum_j m(\xi_j)\right) \KL(\pi_i \ones | \mu_i) + \left(\lambda_2 \sum_j m(\xi_j)\right) \KL(\pi_i^\top \ones | \mu_i')
  \end{split}
\end{align}
where
\begin{align}
  \begin{split}
  M_i &= L_i + \sum_j \left( \varepsilon \Ent(\xi_j | \mu_j \otimes \mu_j') + \lambda_1 \Ent(\xi_j \ones | \mu_j) + \lambda_2 \Ent(\xi_j^\top \ones | \mu_j') \right),  \\
  L_i &= \begin{cases}
    \frac{1}{2} L(\omega_{12}, \omega_{12}') \otimes \xi_{2} + \frac{1}{2} L(\omega_{11}, \omega_{11}') \otimes \xi_1, &i = 1; \\
    \frac{1}{2} \left( L(\omega_{i, i+1}, \omega_{i, i+1}') \otimes \xi_{i+1} + L(\omega_{i-1, i}^\top, \omega_{i-1, i}'^\top) \otimes \xi_{i-1}\right) + \frac{1}{2} L(\omega_{ii}, \omega_{ii}') \otimes \xi_i, &2 \leq i \leq k-1; \\
    \frac{1}{2} L(\omega_{k-1, k}^\top, \omega_{k-1, k}'^\top) \otimes \xi_{k-1}+ \frac{1}{2} L(\omega_{kk}, \omega_{kk}') \otimes \xi_k & i = k.
  \end{cases}
  \end{split}
\end{align}
Similarly, in each of the $\xi_i, 1 \leq i \leq k$, we solve
\begin{align}
  \begin{split}
  \min_{\xi_i \in \Mc_+(X_i \times X_i')}  \langle M'_i, \xi_i \rangle &+ \left( \varepsilon \sum_j m(\pi_j) \right) \KL(\xi_i | \mu_i \otimes \mu_i') \\
                                                                       &+ \left( \lambda_1 \sum_j m(\pi_j) \right) \KL(\xi_i \ones | \mu_i) + \left( \lambda_2 \sum_j m(\pi_j) \right) \KL(\xi_i^\top \ones | \mu_i'), 
  \end{split}
\end{align}
where
\begin{align}
  \begin{split}
  M'_i &= L'_i + \sum_j \left( \varepsilon \Ent(\pi_j | \mu_j \otimes \mu_j') + \lambda_1 \Ent(\pi_j \ones | \mu_j) + \lambda_2 \Ent(\pi_j^\top \ones | \mu_j') \right),  \\
  L'_i &= \begin{cases}
    \frac{1}{2} L(\omega_{12}, \omega_{12}') \otimes \pi_{2} + \frac{1}{2} L(\omega_{11}^\top, \omega_{11}'^\top) \otimes \pi_1, &i = 1;\\
    \frac{1}{2} \left( L(\omega_{i, i+1}, \omega_{i, i+1}') \otimes \pi_{i+1} + L(\omega_{i-1, i}^\top, \omega_{i-1, i}'^\top) \otimes \pi_{i-1}\right) + \frac{1}{2} L(\omega_{ii}^\top, \omega_{ii}'^\top) \otimes \pi_i, &2 \leq i \leq k-1;\\
    \frac{1}{2} L(\omega_{k-1, k}^\top, \omega_{k-1, k}'^\top) \otimes \pi_{k-1}+ \frac{1}{2} L(\omega_{kk}^\top, \omega_{kk}'^\top) \otimes \pi_k & i = k.
  \end{cases}
  \end{split}
\end{align}
These updates, together with the projections of Algorithm \ref{alg:unbalanced_partitioned} onto the mass equality constraint sets, give a numerical approach to approximating a solution of the general unbalanced multiscale alignment problem.  

\subsection{Barycenters with fixed support}\label{sec:barycenters_fixed_support}

As an alternative to the blow-up scheme of \cite{chowdhury2020gromov}, we can consider an approximation of the barycenter problem where we restrict our approach to seeking a minimizer over network representatives of a fixed size. 
This is the same as the approach of \cite{peyre2016gromov}, which was developed in the setting of the Gromov-Wasserstein distance. That is, for an input ensemble of partitioned measure networks $\{ P^{(i)}, 1 \leq i \leq N \}$, we consider a barycenter $\overline{P} = ((\overline{X}_i, \overline{\mu}_i)_{i = 1}^k, \overline{\omega})$ in which we have fixed the cardinalities of $\overline{X}_i$ to $|\overline{X}_i| = n_i$. We also prescribe the probability measures $\overline{\mu}_i$ for $\overline{P}$, so that it remains to find the optimal function $\overline{\omega}$. Expanding the definition of the partitioned network distance between hypernetworks, we have 
\begin{align}
  \min_{\overline{\omega}} \sum_{i = 1}^{N} w_i d_{\Pc_k}(\overline{P}, P^{(i)})^2 = \min_{\overline{\omega}} \sum_{i = 1}^N \left[ w_i \min_{\pi^{(i)} \in \Pi_k(\overline{\mu}, \mu^{(i)})} \| \overline{\omega} - \omega^{(i)} \|_{L^2(\pi^{(i)} \otimes \pi^{(i)})}^2 \right].
  \label{eq:barycenter_objective}
\end{align}
From this, it is apparent that an alternating scheme can be developed by minimizing separately in the couplings $\pi^{(i)} \in \Pi_k(\overline{\mu}, \mu^{(i)})$ and in the function $\overline{\omega}$.
Fixing $\overline{\omega}$, the objective \eqref{eq:barycenter_objective} can be minimized in each of the $\pi^{(i)}, 1 \leq i \leq N$ by solving $N$ independent partitioned network matching problems. Fixing the couplings $\{ \pi^{(i)} \}_{i = 1}^N$, the minimization problem in $\overline{\omega}$ becomes 
\[
  \min_{\overline{\omega}} \sum_i w_i \: \| \overline{\omega} - \omega^{(i)} \|_{L^2(\pi^{(i)} \otimes \pi^{(i)})}^2,  
\]
this amounts to minimizing a quadratic objective and therefore has a closed form solution. 

\begin{prop}[Barycenter update for fixed couplings]
  For fixed couplings $\pi^{(i)}, 1 \leq i \leq N$, the objective \eqref{eq:barycenter_objective} is quadratic and minimized in $\overline{\omega}$ at 
  \begin{equation}
    \overline{\omega}_{jl}^\star = \frac{1}{\overline{\mu}_j \otimes \overline{\mu}_l} \sum_{i = 1}^N w_i \pi^{(i)}_j \omega^{(i)}_{jl} {\pi^{(i)}_l}^\top, \quad 1 \leq j, l \leq k.
  \end{equation}
\end{prop}
We note that the derivation of the form of this update is identical to that of \cite{peyre2016gromov}, except for the presence of two possibly distinct couplings in the summand.

\end{document}